\documentclass[11pt]{amsart}
\usepackage {enumerate}
\usepackage{setspace}
\usepackage{caption}
\usepackage{mathrsfs}
\usepackage{hyperref}
\usepackage{esint}
\usepackage{amssymb}
\usepackage{comment}
\usepackage{color}
\usepackage{lipsum}

\newtheorem{theorem}{Theorem}[section]
\newtheorem{lemma}[theorem]{Lemma}
\newtheorem{corollary}[theorem]{Corollary}
\newtheorem{proposition}[theorem]{Proposition}

\newtheorem{conjecture}[theorem]{Conjecture}
\newtheorem{remark}[theorem]{Remark}
\numberwithin{equation}{section}
\newtheorem{definition}[theorem]{Definition}

\DeclareMathOperator{\diam}{diam}
\DeclareMathOperator{\id}{id}
\DeclareMathOperator{\Lip}{Lip}

\DeclareMathOperator{\Div}{div}
\DeclareMathOperator{\tr}{tr}
\DeclareMathOperator{\dist}{dist}

\makeatletter
\@namedef{subjclassname@2020}{%
  \textup{2020} Mathematics Subject Classification}
\makeatother

\usepackage{amsfonts}
\usepackage{amssymb}
\usepackage[utf8]{inputenc}

\usepackage[english ]{babel}

\usepackage{bbding}

\usepackage[all]{xy}

\usepackage{tikz-cd}
\usepackage[all]{xy}
\usepackage{amsfonts}
\usepackage{amssymb}
\usepackage{graphicx}
\usepackage{mathtools}
\usepackage{skak}
\usepackage{foekfont}
\usepackage{epstopdf}

\usepackage{calligra}
\usepackage[T1]{fontenc}

\usepackage{amsmath,mathdots}

\usepackage{stmaryrd}

\usepackage[utf8]{inputenc}

\usepackage{mathtools}
\usepackage{mathdots}
\usepackage{tikz}

\definecolor{amber(sae/ece)}{rgb}{1.0, 0.49, 0.0}
\newfont{\rsfsten}{rsfs10 scaled 1200}
\usepackage{MnSymbol}
\usepackage{tikz}
\usepackage{amscd}
\usepackage{MnSymbol}
\usepackage{wasysym}

\usepackage{mathrsfs}

\usepackage{pdfcolmk}

\usepackage{graphicx}

\DeclareMathOperator{\Ric}{Ric}
\DeclareMathOperator{\vol}{Vol}
\DeclareMathOperator{\Osc}{Osc}

\begin{document}

\title[Positive mass theorem with an incompressible condition]{Incompressible hypersurface, positive scalar curvature and positive mass theorem}
\author{Jie Chen}
\address[Jie Chen]{ Institute of Applied Physics and Computational Mathematics, Beijing 100088, P.R. China.}
\email{jiechern@163.com}
\thanks{}
\author{Peng Liu}
\address [Peng Liu] {Key Laboratory of Pure and Applied Mathematics, School of Mathematical Sciences, Peking University, Beijing, 100871, P.\ R.\ China}
\email{1801110011@pku.edu.cn}
\author{Yuguang Shi}
\address [Yuguang Shi] {Key Laboratory of Pure and Applied Mathematics, School of Mathematical Sciences, Peking University, Beijing, 100871, P.\ R.\ China}
\email{ygshi@math.pku.edu.cn}
\thanks{Y.Shi is partially supported by National Key R$\&$D Program of China 2020YFA0712800,  and NSFC 11731001.}

\author{Jintian Zhu}
\address [Jintian Zhu]   {Beijing international center for mathematical research, Peking University, Beijing, 100871, P. R. China}
\email{zhujt@pku.edu.cn}
\email{zhujintian@bicmr.pku.edu.cn}
\thanks{J. Zhu is supported by the China post-doctoral grant BX2021013.} 
\date{\today}
\begin{abstract}
 In this paper, we prove for $n\leq 7$ that if a differentiable $n$-manifold contains a relatively incompressible essential hypersurface in some class $\mathcal C_{deg}$, then it admits no complete metric with positive scalar curvature. Based on this result, we show for $n\leq 7$ that surgeries between orientable $n$-manifolds and $n$-torus along incompressible sub-torus with codimension no less than $2$ still preserve the obstruction for complete metrics with positive scalar curvature. As an application, we establish positive mass theorem with incompressible conditions for asymptotically flat/conical manifolds with flat fiber $F$ (including ALF and ALG manifolds), which can be viewed as a generalization of the classical positive mass theorem from \cite{SY79PMT} and \cite{SY2017}. Finally, we investigate Gromov's fill-in problem and bound the total mean curvature for nonnegative scalar curvature fill-ins of flat $2$-toruses (an optimal bound is obtained for product $2$-toruses). This confirms the validity of Mantoulidis-Miao's definition of generalized Brown-York mass in \cite{MM2017} for flat $2$-toruses.
\end{abstract}

\maketitle
\section{Introduction}
\subsection{Essential manifolds, incompressible hypersurface and positive scalar curvature} Recall that a closed manifold is called {\it aspherical} if its universal covering is contractible, i.e. homotopy equivalent to a point. The {\it aspherical conjecture} \cite{Rosenberg1983} asserts that any closed aspherical manifold cannot admit smooth metrics with positive scalar curvature. When the dimension is two, this conjecture is just a direct consequence from the Gauss-Bonnet formula and the classfication for closed surfaces. In higher dimensions, this conjecture becomes more difficult and we briefly introduce the researches on that. In 1983, Gromov and Lawson \cite{GL83} first verified this conjecture in three dimensional case. Shortly after that, Schoen and Yau wrote a short survey \cite{SY87} with a sketchy proof for the aspherical conjecture in dimension four. But until recently, Gromov \cite{Gromov2020} as well as Chodosh and Li \cite{CL2020} verified this conjecture up to dimension five independently. All of these works actually lead to the consideration for the following enhanced aspherical conjecture. {\it That is, any closed essential manifold admits no smooth metric with positive scalar curvature, where essential manifolds here mean those admitting a non-zero degree map to aspherical ones.} Just recently, the enhanced aspherical conjecture up to dimension five was proven in detail by Chodosh, Li and Liokumovich in \cite{CLL2021} although this seems to be already asserted by Gromov in his four lectures \cite{Gromov2021}.

Instead of the manifold itself being essential, the existence of an essential incompressible  closed hypersurface also provides a topological obstruction for positive scalar curvature metrics. Here an incompressible hypersurface in an $n$-manifold $M^n$ simply means a continuous map $i:\Sigma\to M$ such that $\Sigma$ is an essential closed $(n-1)$-manifold and the induced map $i_*:\pi_1(\Sigma) \to \pi_1(M)$ between fundamental groups is injective. This kind of result was first proven by Schoen and Yau in their papers \cite{SY1979} and \cite{SY1982}, where they showed that if an orientable complete $3$-manifold contains an orientable incompressible closed surface with positive genus, then it admits no smooth metric with positive scalar curvature. Later in \cite{GL83}, Gromov and Lawson made the following generalization: if $M$ is a compact manifold with enlargeable\footnote{A manifold $M^n$ is called enlargeable if for any $\epsilon>0$ we can find a covering of $M$ which is spin and admits a non-zero degree map to $\mathbf S^n(1)$ with Lipschitz contant less than $\epsilon$. Typical examples are Cantan-Hardmard manifolds. The use of this requirement here is the same as our class $\mathcal C_{deg}$ defined next.} essential boundary such that the inclusion map $\pi_1(\partial M)\to \pi_1(M)$ is injective, then it cannot admit any smooth metric with positive scalar curvature and mean convex boundary. We point out that one can even obtain a stronger conclusion from their proof. That is, if $M^n$ is a closed $n$-manifold containing an incompressible enlargeable essential closed hypersurface, then it cannot admit any smooth metric with positive scalar curvature. Based on the latter statement the original result by Gromov and Lawson follows easily from a doubling trick.

In this paper, we would like to strengthen Gromov-Lawson's result along two directions. First we are able to drop the requirement on the closeness of underlying manifold with the help of soap bubble (also called $\mu$-bubble) method. Recently, soap bubbles were used in various situations to establish results on non-compact manifolds and the audience can refer to \cite{Zhu2021}, \cite{Zhu2020}, \cite{Thomas2020}, \cite{LUY2021} and \cite{CL2020}. Second we can relax the incompressible condition to some {\it relatively incompressible condition} that allows part of the fundamental group mapped to zero. As we shall see later, the first improvement enables us to establish positive mass theorems with an incompressible condition and arbitrary ends.

Before the statement of our main theorem, let us introduce some notions for convenience. In the following, we denote $\mathcal C_{deg}$ to be the collection of those closed orientable aspherical manifolds $M$ satisfying the following property: any essential closed manifold over\footnote{A closed essential manifold over $M$ means that it admits a non-zero degree map to $M$.} $M$ admits no positive scalar curvature metric. Once the enhanced aspherical conjecture is proven, the class $\mathcal C_{deg}$ is nothing but all aspherical closed manifolds.  At this stage, the class $\mathcal C_{deg}$ is known to include the following manifolds:
\begin{itemize}
\item[(a)] closed aspherical manifolds with dimension no greater than five \cite{CLL2021};
\item[(b)] the $n$-torus $T^n$ with dimension $n\leq 8$ \cite{SY2017} (No dimensional restriction is needed in \cite{SY2017}, but we focus on the case $n\leq 8$ for safety);
\item[(c)] closed Cantan-Hardmard $n$-manifolds with dimension $n\leq 8$ \cite{Gromov2018}.
\end{itemize}
We point out that the dimensional restriction $n\leq 8$ comes from the regularity theory for minimizing currents in geometric measure theory. It is well-knwon that no singularities occur when the ambient manifold has dimension no greater than $7$ and singularities can be perturbed away in eight dimensional manifolds through Nathan Smale's work in \cite{Nathan1993}.

With the class $\mathcal C_{deg}$ we can state our theorem as following
\begin{theorem}\label{Thm: main 1}
Let $\Sigma_0$ be an orientable closed aspherical manifold in $\mathcal C_{deg}$, and $\Sigma$ be an orientable essential closed manifold over $\Sigma_0$ associated with a non-zero degree map $f:\Sigma\to \Sigma_0$. Assume that $M$ is a differentiable $n$-manifold ($n\leq 7$, compact or non-compact) with an orientable essential closed hypersurface $i:\Sigma\to M$ such that $\ker i_*\subset \ker f_*$ in $\pi_1(\Sigma)$, where $i_*$ and $f_*$ are the corresponding induced maps of $i$ and $f$ between fundamental groups respectively. Then $M$ admits no complete positive scalar curvature metric. Moreover, if $g$ is a complete smooth metric on $M$ with nonnegative scalar curvature, then $(M,g)$ is flat.
\end{theorem}

\subsection{Generalized connected sum along submanifolds}
In the second part of this paper, we are going to make a discussion on the relationship between the generalized connected sum and positive scalar curvature. First let us recall the definition of the generalized connected sum considered in \cite{SY79} as well as \cite{GL80} . Let $\Sigma^k$ be a closed $k$-manifold and define the class
\begin{equation*}
\mathcal M^n_\Sigma=\left\{(M,i)\left|
\begin{array}{c}
\text{$M$ is a differentiable $n$-manifold and $i:\Sigma\to M$}\\
\text{is an embedding with trivial normal bundle}
\end{array}\right.\right\}.
\end{equation*}
For convenience, a pair $(M,i)$ in $\mathcal M_{\Sigma}^n$ will be called a marked manifold in the following. We emphasize that there is no requirement on the compactness of marked manifolds in the definition of $\mathcal M^n_\Sigma$. Given two marked manifolds $(M_1,i_1)$ and $(M_2,i_2)$ in $\mathcal M_{\Sigma}^n$, the generalized connected sum
$$(M_1,i_1)\#_{\Sigma}(M_2,i_2)$$
is defined to be the manifold obtained from the surgery as follows. Take a tubular neighborhood $U_l$ of $i_l(\Sigma)$ in each $M_l$ for $l=1,2$. Clearly $\partial U_1$ and $\partial U_2$ are isomorphic sphere bundles over $\Sigma$ (actually over $i_1(\Sigma)$ and $i_2(\Sigma)$). Pick up $\Phi$ to be a {\it fiber preserving diffeomorphism} between $\partial U_1$ and $\partial U_2$. Then $(M_1,i_1)\#_{\Sigma}(M_2,i_2)$ is defined to be the gluing space
$$
\left(M_1-U_1\right)\sqcup_\Phi \left(M_2-U_2\right).
$$
When $\Sigma$ is a single point, it is clear that the generalized connected sum reduces to the classical one.

The researches on the relationship between the generalized connected sum and positive scalar curvature mainly focus on two opposite directions: the generalized connected sum preserving admission or obstruction for positive scalar curvature. Along the first direction, Schoen and Yau \cite{SY79} as well as Gromov and Lawson \cite{GL80} proved that: given any closed manifold $\Sigma^k$ with $k+3\leq n$, if $(M_1,i_1)$ and $(M_2,i_2)$ are two marked manifolds in $\mathcal M^n_{\Sigma}$ admitting complete metrics with positive scalar curvature, the generalized connected sum $(M_1,i_1)\#_{\Sigma}(M_2,i_2)$ also admits a complete positive scalar curvature metric. On the other hand, based on the work of Schoen and Yau \cite{SY79}\cite{SY2017} as well as the soap bubble method, Chodosh and Li \cite{CL2020} proved that $(T^n,i)\#_{\Sigma}(M_2,i_2)$ with $n\leq 7$ cannot admit any complete metric with positive scalar curvature in the case when $\Sigma$ is a single point.

Here we would like to further generalize Chodosh-Li's work and this serves as a preparation for our positive mass theorems with an incompressible conditions and more general asymptotics. For our purpose, let us limit our attention to the following special case. Namely we take $\Sigma$ to be the $k$-torus $T^k$ and $(T^n,i)$ to be the $n$-torus associated with a linear embedding $i:T^k\to T^n$. Clearly the marked manifold $(T^n,i)$ is an element in $\mathcal M^n_{T^k}$ and we can consider the generalized connected sum $(T^n,i)\#_{T^k}(M_2,i_2)$ for any $(M_2,i_2)$ in $\mathcal M^n_{T^k}$.

Given any marked manifold $(M,i)$ in $\mathcal M^n_{\Sigma}$, we say that $(M,i)$ satisfies the incompressible condition if the inclusion map $i_*:\pi_1(\Sigma)\to \pi_1(M)$ is injective. When $\Sigma$ is $\mathbf S^1$, $(M,i)$ is said to satisfy the homotopially non-trivial condition if the map $i:\mathbf S^1\to M$ is not homotopic to a point.

As an application of our main Theorem \ref{Thm: main 1}, we can establish
\begin{proposition}\label{Prop: main 2}
Let $n\leq 7$ and $k\leq n-2$. If $(M_2,i_2)$ is a marked $n$-manifold in $\mathcal M^n_{T^k}$ satisfying the incompressible condition, then the generalized connected sum $(T^n,i)\#_{T^k}(M_2,i_2)$  admits no complete metric with positive scalar curvature. Moreover, if $g$ is a complete metric on $(T^n,i)\#_{T^k}(M_2,i_2)$ with nonnegative scalar curvature, then the metric $g$ must be flat.
\end{proposition}
When $k=1$, we can further strengthen above result as following
\begin{proposition}\label{Prop: main 3}
Let $3\leq n\leq 7$. If an orientable marked manifold $(M_2,i_2)$ in $\mathcal M^n_{\mathbf S^1}$ satisfies the homotopically non-trivial condition, then the generalized connected sum $(T^n,i)\#_{\mathbf S^1}(M_2,i_2)$ admits no complete metric with positive scalar curvature. Moreover, if $g$ is a complete metric on $(T^n,i)\#_{\mathbf S^1}(M_2,i_2)$ with nonnegative scalar curvature, then the metric $g$ must be flat.
\end{proposition}

We point out that Proposition \ref{Prop: main 3} is closely related to the class of Schoen-Yau-Schick manifolds. Let us recall from \cite{Gromov2018} that an orientable closed manifold $M^n$ is called a Schoen-Yau-Schick manifold if there are cohomology classes $\beta_1,\,\beta_2,\ldots,\beta_{n-2}$ in $H^1(M,\mathbf Z)$ such that the homology class
$$
[M]\frown(\beta_1\smile\beta_2\smile\cdots\smile \beta_{n-2})
$$
in $H_2(M,\mathbf Z)$ is non-spherical, that is, it does not lie in the image of the Hurewicz homomorphism $\pi_2(M)\to H_2(M,\mathbf Z)$. In their work \cite{SY2017}, Schoen and Yau proved that every orientable closed Schoen-Yau-Schick manifold admits no smooth metric with nonnegative scalar curvature unless it is flat. Actually, the generalized connected sum $(T^n,i)\#_{\mathbf S^1}(M_2,i_2)$ turns out to be a Schoen-Yau-Schick manifold if the marked manifold $(M_2,i_2)$ is a closed orientable manifold satisfying the homotopiclly non-trivial condition (see the proof of Proposition \ref{Prop: main 3}). Even though, an extra effort needs to be made in our proof for Proposition \ref{Prop: main 3} to deal with the possible non-compactness of $M_2$.

Recall that Chodosh, Li and Liokumovich \cite{CLL2021} proved that any essential closed manifold admits no smooth metric with positive scalar curvature. In particular, the connected sum of an aspherical closed manifold with any other closed manifold preserves the topological obstruction for positive scalar curvature. For further generalizations of this fact it is natural to consider the generalized connected sum of an aspherical closed manifold with other closed or open manifolds. Since this paper is mainly devoted to establishing a class of positive mass theorems with an incompressible condition, we shall leave the discussion on this interesting topic to another paper.

In our later discussion on positive mass theorems with an incompressible condition, we will use the idea of Lohkamp compactification \cite{Lohkamp1999} so that the positive mass theorem can be reduced to a topological obstruction problem for complete metrics with positive scalar curvature, where Proposition \ref{Prop: main 2} and \ref{Prop: main 3} are ready to apply. 

\subsection{Positive mass theorems with an incompressible condition}
The positive mass theorem for asymptotically flat manifolds \cite{SY79PMT}\cite{SY2017} (see also \cite{Witten1981} for the spin case) appears to be one of the most beautiful results in both geometry of scalar curvature and general relativity. It states that the ADM mass has to be nonnegative for any asymptotically flat manifold with nonnegative scalar curvature, and the mass vanishes exactly when the manifold is isometric to the Euclidean space. 

From quantum gravity theory and string theory, gravitational instantons provide more examples of ``asymptotically flat'' manifolds with more general asymptotics at infinity. By definition a gravitational instanton means a non-compact hyperk\"ahler $4$-manifolds with decaying curvature at infinity, and various examples were discussed in \cite{K1989}, \cite{CK1998}, \cite{CH2005} as well as \cite{CK2002}. Given these plentiful examples it is a rather difficult problem to classify all possible gravitational instantons and it seems a nice idea to focus only on several special classes. In particular, Cherkis and Kapustin conjectured a classification scheme (see \cite{EJ2008}), which involves a consideration on the following four special families of gravitational instantons:
\begin{itemize}
\item[(i)] Asymptotically Locally Euclidean (ALE): asymptotical to $\mathbf R^4/\Gamma$ at infinity for some discrete finite group $\Gamma\subset O(4)$;
\item[(ii)] Asymptotically Locally Flat (ALF): asymptotical to the total space of a circle bundle over $\mathbf R^3$ or $\mathbf R^3/\mathbf Z_2$ at infinity;
\item[(iii)] ALG (no explicit meaning): asymptotical to the total space of $T^2$-bundle over $\mathbf R^2$ at infinity;
\item[(iv)] ALH (no explicit meaning): asymptotical to the total space of $F^3$-bundle over $\mathbf R$, where $F$ is a flat closed $3$-manifold.
\end{itemize}

Given these interesting examples, it is a natural question whether positive mass theorems hold for ``asymptotically flat'' manifolds with more general asymptotics. In their work \cite{HP78}, Hawking and Pope proposed the {\it generalized positive action conjecture}: any ALE 4-manifold with vanishing scalar curvature has nonnegative ADM mass, which vanishes if and only if the manifold is Ricci flat with self-dual Weyl curvature. Unfortunately, this conjecture turns out to be false in general. Actually, LeBrun \cite{Lebrun1988} constructed a family of counter-examples from the resolution of quotient spaces $\mathbf C^2/\mathbf Z_k$ with $k\geq 3$. On the other hand, counter-examples vialating the philosophy of positive mass theorems can be also found in the class of ALF manifolds, where there is a well-known example given by the product manifold $\mathbf R^2\times \mathbf S^2$ equipped with the Reissnet-Nordstr\"om metric (see \cite{Minerbe2008} for instance). We point out that this is a complete Riemannian manifold asymptotic to $\mathbf R^3\times \mathbf S^1$ at infinity with vanishing scalar curvature and negative total mass. The common feature of these counter-examples is that their ends are not incompressible.

Over the decades, attempts have been made to establish positive mass theorems with necessary additional conditions. Inspired from Witten's proof on the positive mass theorem for asymptotically flat manifolds, various authors proved positive mass theorems with an additional spin compactible condition in different scenarios. For further information, the audience can refer to \cite{Dahl1997}, \cite{Minerbe2008} and \cite{Dai2004}. Roughly speaking, the spin compactible condition can be used to guarantee the existence of spinors from a perturbation of classical ones, and then the Witten's proof can be applied in these cases without difficulty. Despite of its effectiveness, the spin compactible condition is usually not easy to verify and so this motivates us to search for other convenient conditions to guarantee the validity of positive mass theorems.

In the third part of this paper, we will research on positive mass theorems with some type of incompressible condition. In the following, we would like to consider the following class of manifolds.
\begin{definition}\label{Defn: AF with fiber F}
Let $(F,g_F)$ be a closed flat manifold. A triple $(M,g,\mathcal E)$ is called an asymptotically flat manifold with fiber $F$ if 
\begin{itemize}
\item $(M,g)$ is a complete Riemannian manifold with $d:=\dim M-\dim F\geq 3$;
\item $\mathcal E$ is an end of $M$ diffeomorphic to $(\mathbf R^d-B)\times F$;
\item Denote $g_0=g_{euc}\oplus g_F$. The metric $g$ on $\mathcal E$ satisfies
\begin{equation}\label{Eq: decay 1}
|g-g_0|_{g_0}+r|\nabla_{g_0}(g-g_0)|_{g_0}+r^2|\nabla_{g_0}^2(g-g_0)|_{g_0}=O(r^{-\mu}),\quad\mu>\frac{d-2}{2},
\end{equation}
where $\nabla_{g_0}$ is the covariant derivative with respect to $g_0$, and $r$ is the distance function on $\mathbf R^d$;
\item $R(g)\in L^1(\mathcal E,g)$.
\end{itemize}
\end{definition}
\begin{remark}
We emphasize that $(M,g)$ can have more than one ends but no additional requirement other than the completeness is imposed for ends other than $\mathcal E$. It may be better to call $(M,g,\mathcal E)$ asymptotically flat manifold with fiber $F$ and arbitrary ends as in \cite{LUY2021}, but we just omit ``arbitrary ends'' for short since the name is already very long.
\end{remark}

The definition above includes the following interesting examples:
\begin{itemize}
\item[(i)] {\it Schwarzschild-like manifolds with fiber $F$.} This class consists of the manifold $(\mathbf R^d-O)\times F^k$ equipped with the metrics
$$
g=\left(1+\frac{A}{r^{d-2}}\right)^{\frac{4}{d+k-2}}(g_{euc}\oplus g_F),\quad A\geq 0,
$$
where $O$ is the origin of $\mathbf R^d$ and $g_F$ is a flat metric on $F$. This provides a scalar-flat family of asymptotically flat manifolds with fiber $F$.

\item[(ii)] {\it Schwarzschild-product manifolds with fiber $F$.} Similarly, we can equip $(\mathbf R^d-O)\times F^k$ with the metrics
$$
g=\left(\left(1+\frac{A}{r^{d-2}}\right)^{\frac{4}{d-2}}g_{euc}\right)\oplus g_F,\quad A\geq 0.
$$
Clearly this provides another scalar-flat family of asymptotically flat manifolds with fiber $F$.

\item[(iii)] {\it Schwarzschild manifolds.} Fix $A\geq 0$. We can equip $(\mathbf R^d-B_{r_0})\times \mathbf S^1$, $r_0=A^{\frac{1}{d-2}}$, with the metric
$$
g=\left(\left(1+\frac{A}{r^{d-2}}\right)^{\frac{4}{d-2}}g_{euc}\right)\oplus \left(\left(\frac{1-Ar^{2-d}}{1+Ar^{2-d}}\right)^2\mathrm d\theta^2\right).
$$
When $A>0$, this provides a family of Ricci-flat metrics on $\mathbf S^{d-1}\times \mathbf R^2$ (after completion), which is asymptotically flat with fiber $\mathbf S^1$.

\item[(iv)] {\it Reissnet-Nordstr\"om metrics.} Fix real numbers $A$ and $B\neq 0$. Let us define the metric
$$
g=\left(\left(1+\frac{A}{r}+\frac{A^2+B^2}{4r^2}\right)^{2}g_{euc}\right)\oplus \left(\left(\frac{1-\frac{1}{4}r^{-2}(A^2+B^2)}{1+r^{-1}A+\frac{1}{4}r^{-2}(A^2+B^2)}\right)^2\mathrm d\theta^2\right)
$$
on $(\mathbf R^3-B_{r_0})\times \mathbf S^1$ with $r_0=\frac{1}{2}\sqrt{A^2+B^2}$. After completion this gives a family of scalar-flat metrics on $\mathbf S^2\times \mathbf R^2$, which is asymptotically flat with fiber $\mathbf S^1$.
\end{itemize}

For asymptotically flat manifolds with fiber $F$ we can introduce a total mass in the same spirit of the ADM mass for asymptotically flat manifolds.
\begin{definition}
Let $(M,g,\mathcal E)$ be an asymptotically flat manifold with fiber $F$. The total mass of $(M,g,\mathcal E)$ is defined to be
\begin{equation*}
\begin{split}
m(M,g,\mathcal E)=&\frac{1}{2|\mathbf S^{d-1}|\vol(F,g_F)}\\
&\qquad\cdot\lim_{\rho\to+\infty}\int_{S_\rho\times F}\ast_{g_0}\left(\Div_{g_0}g-d\tr_{g_0}g\right),
\end{split}
\end{equation*}
where $S_\rho$ is denoted to be the $\rho$-sphere in $\mathbf R^d$ centered at the origin and $\ast_{g_0}$ is the Hodge star operator with respect to the metric $g_0$.
\end{definition}
\begin{remark}
One easily checks that $m(M,g,\mathcal E)$ conincides with the ADM mass up to a positive scale if $(M,g,\mathcal E)$ is asymptotically flat in the usual sense.
\end{remark}

We are able to prove the following
\begin{theorem}\label{Thm: main 4}
If $(M,g,\mathcal E)$, $\dim M\leq 7$, is an asymptotically flat manifold with fiber $F$ such that its scalar curvature $R(g)\geq 0$ and one of the following holds:
\begin{itemize}
\item $i_*:\pi_1(\mathcal E)\to \pi_1(M)$ is injective;
\item or $F=\mathbf S^1$ and $i_*:\pi_1(\mathcal E)\to \pi_1(M)$ is non-zero,
\end{itemize} 
then we have $m(M,g,\mathcal E)\geq 0$.  In the first case, $m(M,g,\mathcal E)=0$ yields that $(M,g)$ is flat. In the second case, if the metric $g$ further satisfies
$$
\sum_{i=0}^3 r^i|\nabla^i_{g_0}(g-g_0)|_{g_0}=O(r^{-\mu}),\quad\mu>\frac{d-2}{2},
$$
then $m(M,g,\mathcal E)=0$ yields that $(M,g)$ splits as $\mathbf R^d\times \mathbf S^1$.
\end{theorem}

Even though we establish the class of positive mass theorems with an incompressible condition, it does not rule out any possibility for a positive mass theorem without incompressible conditions. For example, the Eguchi-Hason gravitational instanton \cite{EH1979} provides an example of complete manifolds with zero mass but it has a compressible end diffeomorphic to $\mathbf R^4/\mathbf Z_2$. So it might be possible to prove a positive mass theorem for ALE manifolds with its end diffeomorphic to $\mathbf R^4/\mathbf Z_2$ without any additional conditions.

It follows from Bartnik's work \cite{Bartnik1986} that the total mass is actually a modified integral of scalar curvature. From the Gauss-Bonnet formula, the counterpart for a complete surface turns out to be the angle at infinity. This leads us to consider the following class of manifolds.
\begin{definition}
Let $(F,g_F)$ be a closed flat manifold. A triple $(M,g,\mathcal E)$ is called an asymptotically conical manifold with fiber $F$ and (arbitrary ends) if 
\begin{itemize}
\item $(M,g)$ is a complete Riemannian manifold with $\dim M-\dim F= 2$;
\item $\mathcal E$ is an end of $M$ diffeomorphic to $(\mathbf R^2-B)\times F$;
\item there is a positive constant $\beta$ such that the metric $g$ on $\mathcal E$ satisfies
\begin{equation}\label{Eq: decay 2}
|g-g_\beta|_{g_\beta}+r|\nabla_{g_\beta}(g-g_\beta)|_{g_\beta}+r^2|\nabla_{g_\beta}^2(g-g_\beta)|_{g_\beta}=O(r^{-\mu}),\quad \mu>0,
\end{equation}
where $g_\beta=\mathrm dr^2+\beta^2r^2\mathrm d\theta^2+g_F$, $\nabla_{g_\beta}$ is its corresponding covariant derivative, and $r$ is the distance function on $\mathbf R^2$;
\end{itemize}
For any asymptotically conical manifold $(M,g,\mathcal E)$, the constant $2\pi\beta$ will be called the angle of $(M,g,\mathcal E)$ at infinity.
\end{definition}

Denote $i_F:F\to M$ to be the natural inclusion map from the composition of  the inclusions $i_1:F\to \mathcal E$ and $i_2:\mathcal E\to M$. Similarly we have
\begin{theorem}\label{Thm: main 5}
If $(M,g,\mathcal E)$, $\dim M\leq 7$, is an asymptotically conical manifold with fiber $T^{n-2}$ such that its scalar curvature $R(g)\geq 0$ and one of the following holds:
\begin{itemize}
\item $F=T^{n-2}$ and $(i_F)_*:\pi_1(F)\to \pi_1(M)$ is injective;
\item or $F=\mathbf S^1$ and $(i_F)_*:\pi_1(F)\to \pi_1(M)$ is non-zero,
\end{itemize} 
then the angle of $(M,g,\mathcal E)$ at infinity is no greater than $2\pi$, where the equality implies that $(M,g)$ is flat.
\end{theorem}

 We point out that Theorem \ref{Thm: main 4} is a generalization of the classical positive mass theorem for asymptotically flat manifolds and the idea of the proof is similar to the original one. However, there are still additional works to be done.

First, the analysis involved turns out to be much more delicate due to the existence of the fiber $F$ and arbitrary ends. In order to solve PDEs on asymptotically flat manifolds with fiber $F$, we have to establish a weighted Sobolev inequality (see Proposition \ref{Prop: Sobolev}) on $(\mathbf R^d-B)\times F$, which allows tests function to be non-zero on the inner boundary. Also, some extra efforts need to be devoted to our construction of appropriate conformal factors due to the existence of arbitrary ends other than $\mathcal E$ (compare Proposition \ref{3.2likeSY} with \cite[Lemma 3.2]{SY79PMT}).

Second, the volume comparison theorem cannot be used to prove the flatness of $(M,g)$ from its Ricci-flatness if $F=\mathbf S^1$ and $(i_F)_*:\pi_1(F)\to \pi_1(M)$ is only non-zero. As a brand new observation, we find out that the fast decay Ricci curvature implies the following 
result. 
\begin{proposition}
If $(M,g,\mathcal E)$ is an asymptotically flat manifold with fiber $F=\mathbf S^{1}$ and the unique end $\mathcal E$ such that 
$$\sum\limits^3_{k=0}r^k|\nabla_{g_0}^k (g-g_{0})|_{g_0}=O(r^{-\mu}), \quad\mu>\frac{d-2}{2},$$
and
$$|\nabla_{g_0} Ric|_{g_0}+r|\nabla_{g_0} Ric|_{g_0}=O(r^{-n-\epsilon})\quad \text{for some}\quad\epsilon>0 .$$
Let $\{x^{1},\cdots,x^{d},\theta\}$ be the given coordinate system of $(M,g,\mathcal E)$ at infinity, where $\theta$ is the parameter of $S^{1}$. Then there exists a coordinate system $\{y^{1},\cdots,y^{d},\theta\}$ outside a compact subset such that
$\Delta_{g}y^{i}=0,\,i=1,\cdots,d$. 
Moreover, if we denote $\alpha^{i}=dy^{i}$, then we have
\begin{equation}\label{Eq: GB mass}
\sum\limits^{d}_{i=1}\int_{M}|\nabla_g\alpha^{i}|_g^{2}+\Ric(\alpha^{i},\alpha^{i})\mathrm{d}V_{g}=c(d,S^{1})m(M,g,\mathcal{E}),
\end{equation}
where $c(d,S^{1})=d|\mathbf S^{d-1}|\vol(S^{1},g_{S^{1}})$.
\end{proposition}
\begin{remark}
The left hand side of \eqref{Eq: GB mass} is called the Gauss-Bonnet mass by Minerbe in his work \cite{Minerbe2008}. This proposition says that if the Ricci curvature decay is fast enough, then the Gauss-Bonnet mass is equal to the total mass up to a positive scale (independent of the metric $g$).
\end{remark}
In the proof of Theorem \ref{Thm: main 5}, the Lohkamp compactification is no longer valid since the density theorem like Proposition \ref{conformal end} is hard to prove due to the possible blow-up of the conformal factors on asymptotically conical manifolds with fiber $F$. As an alternative, we develop a new compactification method based on the quasi-spherical metric (see Proposition \ref{Prop: PMT to PSC 2}).

\subsection{Brown-York mass and the fill-in problem} Other than the total mass, another central topic in general relativity is the definition of a quantity measuring the mass contained in a bounded region, which is called {\it quasi-local mass}. Brown and York \cite{BY1991} \cite{BY1993} raised the so called Brown-York mass as a candidate for quasi-local mass.
Given a compact
spacelike hypersurface $\Omega$ in a spacetime, assuming its boundary $\partial\Omega$ is a $2$-sphere with
positive Gauss curvature, the Brown–York mass of $\partial\Omega$ is given by
$$
m_{BY}(\partial\Omega,\Omega)=\frac{1}{8\pi}\int_{\partial\Omega}(H_0-H)\,\mathrm d\sigma.
$$
Here $\mathrm d\sigma$ is the induced area element on $\partial\Omega$, $H$ is the mean curvature of $\partial\Omega$ in $\Omega$, and $H_0$ is the mean curvature of $\partial\Omega$ after isometrically embedded into the Euclidean $3$-space. This definition makes sense since the solution of Weyl's embedding problem \cite{Nirenberg1953} \cite{Pogorelov1964} guarantees the existence and uniqueness of an isometric embedding from $2$-spheres with positive Gaussian curvature to the Euclidean space. It is also reasonable since the third named author and Tam \cite{ST2002} proved the nonnegativity of the Brown-York mass when $\Omega$ has nonnegative scalar curvature and mean convex boundary. But without doubt the definition is too restrictive if the boundary $\partial\Omega$ of the region $\Omega$ has to be $2$-spheres with positive Gaussian curvature.
In order to generalize the notion of Brown-York mass to more general surfaces, Mantoulidis and Miao \cite{MM2017} introduced the $\Lambda$-invariant as follows. Given a connected closed surface $(\Sigma,g_\Sigma)$, one first makes the collection $\mathcal F$ of all {\it admissible fill-ins} of $(\Sigma,g_\Sigma)$, i.e. compact $3$-manifolds $(\Omega,g)$ such that
\begin{itemize}
\item $(\partial\Omega,g|_{\partial\Omega})$ is isometric to $(\Sigma,g_\Sigma)$;
\item $\partial\Omega$ is mean-convex with respect to the outer unit normal;
\item the scalar curvature $R(g)\geq 0$.
\end{itemize}
Then the $\Lambda$-invariant is defined to be
$$
\Lambda(\Sigma,g_\Sigma)=\max_{\Omega\in\mathcal F}\int_{\partial\Omega}H\,\mathrm d\sigma.
$$
The essential consideration for $\Lambda$-invariant is the finiteness since if the $\Lambda$-invariant is proven to be finite all the time, then one can introduce the (generalized) Brown-York mass to be
\begin{equation}\label{Eq: BY}
m_{BY}(\partial\Omega,\Omega)=\Lambda(\partial\Omega,g|_{\partial\Omega})-\int_{\partial\Omega}H\,\mathrm d\sigma.
\end{equation}
By proving the finiteness of $\Lambda$-invariant for $2$-spheres with arbitrary smooth metrics, Mantoulidis and Miao successfully got rid of the positive Gaussian curvature condition in the definition of Brown-York mass. But it still leaves a problem whether the Brown-York mass given by \eqref{Eq: BY} is well-defined for closed surfaces with positive genus.

As another applications of our Proposition \ref{Prop: main 2} and \ref{Prop: main 3}, we can generalize the Brown-York mass to the flat $2$-torus case. Namely, we can show
\begin{proposition}\label{Prop: main 7}
For any flat metric $g_{flat}$ on $T^2$, we have
$$
\Lambda(T^2,g_{flat})<+\infty.
$$
Futhermore, if $(T^2,g_{flat})$ splits as $\mathbf S^1(a)\times \mathbf S^1(b)$, then
$$
\Lambda(T^2,g_{flat})=4\pi^2\max\{a,b\}.
$$
\end{proposition}

This proposition (as well as the more general Proposition \ref{Prop: main 6}) also gives a partial answer to the following fill-in problem raised by Gromov. In his work \cite{Gromov2019b}, Gromov suggested a systematic discussion on the following fill-in problem: {\it given a closed Riemannian $n$-manifold $(\Sigma,g_\Sigma)$ associated with a smooth function $h:\Sigma\to \mathbf R$, can one find a compact Riemannian $(n+1)$-manifold $(\Omega,g)$ such that
\begin{itemize}
\item[(i)] $(\partial\Omega,g|_{\partial\Omega})$ is isometric to $(\Sigma,g_\Sigma)$ through a diffeomorphism $\phi:\partial\Omega\to\Sigma$;
\item[(ii)] the mean curvature function $H$ of $\partial\Omega$ with respect to the unit outer normal equals to $h\circ \phi$;
\item[(iii)] the scalar curvature $R(g)$ is nonnegative?
\end{itemize}  }

Other than those partial results in \cite{Gromov2019b}, \cite{SWWZ2021}, \cite{SWW2020}, we mention two complete results related to the fill-in problem above. Without the prescribed function $h$, Shi, Wang and Wei \cite{SWW2020} proved that $(\Sigma,g_\Sigma)$ always admits a compact Riemannian manifold $(\Omega,g)$ satisfying (i) and (iii) if $\Sigma$ can be realized as the boundary of a compact manifold. Based on this result, Miao \cite{Miao2021} proved that there is a universal constant $C=C(\Sigma,g_\Sigma)$ such that $(\Omega,g)$ satisfying (i), (ii) and (iii) does not exist if the prescribed function $h\geq C$. 

Related to the estimate for $\Lambda$-invariant, Gromov proposed the following conjecture:
\begin{conjecture}
There is a universal constant $C=C(\Sigma,g_\Sigma)$ such that any $(\Omega,g)$ with properties (i), (ii) and (iii) satisfies
$$
\int_{\partial\Omega} H\,\mathrm d\sigma\leq C.
$$
\end{conjecture}
This conjecture is now largely open if no additional condition is imposed. Under the further assumption that $\Omega$ has mean convex boundary, the works in \cite{ST2002}, \cite{MM2017}, \cite{SWWZ2021}, \cite{SWW2020} as well as Proposition \ref{Prop: main 7} give several partial answers.
\subsection{The arrangement of this paper} 
The rest of this paper is organized as follows. In Section 2, we give the proof for Theorem \ref{Thm: main 1}. In Section 3, we present proofs for Proposition \ref{Prop: main 2} and Proposition \ref{Prop: main 3}. Section 4 is devoted to proving Theorem \ref{Thm: main 4} and Theorem \ref{Thm: main 5} with different compactification arguments. As a preparation, we also establish weighted Sobolev inequalities on $(\mathbf R^n-B)\times T^k$ (see Proposition \ref{Prop: Sobolev} and Corollary \ref{Cor: Sobolev Lp}). Finally we include a discussion on the fill-in problem and prove Proposition \ref{Prop: main 7}.

 \medskip
 {\it Acknowledgements.} We would like to thank Professor Shing-Tung Yau for drawing our attention to positive mass theorems on manifolds with general asymptotic structure at the infinity. We are also grateful to Dr. Chao Li for many inspiring discussions with the third named author on relationships between the positive mass theorem for ALF manifolds and the fill-in problem.

\section{Proof for Theorem \ref{Thm: main 1}}
Our proof for Theorem \ref{Thm: main 1} is mainly based on the idea from Gromov and Lawson \cite{GL83} but some modifications are needed in order to use the soap bubble method.
It would be nice if we can just deal with an embedding $i:\Sigma\to M$, and so we start with the following reduction lemma.
\begin{lemma}\label{Lem: reduction}
Under the same assumption of Theorem \ref{Thm: main 1}, we can find an essential closed manifold $\Sigma'$ over $\Sigma_0$ and a manifold $M'$ such that there are an embedding $i':\Sigma'\to M'$, a projection map $p':M'\to \Sigma_0$, and a non-zero degree map $f':\Sigma'\to \Sigma_0$
such that the following diagram
\begin{equation*}
\xymatrix{ \Sigma'\ar[rr]^{i'}\ar[dr]_{f'} & &M'\ar[dl]^{p'}\\
& \Sigma_0&}
\end{equation*}
is commutative.
\end{lemma}
\begin{proof}
Fix a marked point $\sigma$ in $\Sigma$ and denote $m=i(\sigma)$ in $M$. First let us take the covering $p:(\hat M,\hat m)\to (M,m)$ such that
$$p_*(\pi_1(\hat M,\hat m))=i_*(\pi_1(\Sigma,\sigma))\subset \pi_1(M,m).$$
From lifting we can find a continuous map $\hat i:(\Sigma,\sigma)\to (\hat M,\hat m)$ such that the following diagram
\begin{equation*}
\xymatrix{
& (\hat M,\hat m)\ar[d]^{p}\\
(\Sigma,\sigma)\ar[ur]^{\hat i}\ar[r]_{i}& (M,m)}
\end{equation*}
is commutative.
Since the induced map $p_*:\pi_1(\hat M,\hat m)\to (M,m)$ is injective, we have
$
\ker \hat i_*=\ker i_*\subset \ker f_*
$ in $\pi_1(\Sigma)$. As a result, we can find a homomorphism $\hat J:\pi_1(\hat M,\hat m)\to \pi_1(\Sigma_0,\sigma_0)$ with $\sigma_0=f(\sigma)$ such that $f_*=\hat J\circ \hat i_*$. Since $\Sigma_0$ is aspherical, from the algebraic topology (see \cite[Proposition 1B.9]{Hatcher2002} for instance) we can construct a continuous map $\hat j:(\hat M,\hat m)\to (\Sigma_0,\sigma_0)$ such that the induced map $\hat j_*:\pi_1(\hat M,\hat m)\to \pi_1(\Sigma_0,\sigma_0)$ equals to $\hat J$. Furthermore, the map $\hat j\circ \hat i$ is homotopic to $f$ and so they induce the same correspondance between homology groups of $\Sigma$ and $\Sigma_0$. In particular, the fundamental class $[\Sigma]$ is mapped to a non-trivial homology class $\beta$ in $H_{n-1}(\hat M)$ under the map $\hat i$. From Lemma \ref{Lem: realize} we can find an embedded orientable hypersurface $\Sigma'$ to represent the homology class $\beta$. Notice that we have
$$\hat j_*(\beta)=f_*([\Sigma])=(\deg f)\cdot[\Sigma_0]\neq 0\in H_{n-1}(\Sigma_0,\mathbf Z).$$
Therefore, we can find one component of $\Sigma'$ such that the map $\hat j$ restricted to this component gives a non-zero degree map to $\Sigma_0$. For simplicity we still denote it by $\Sigma'$. Let $i':\Sigma'\to \hat M$ be the embedding of $\Sigma'$ and $f'=\hat j\circ i'$. Then we have the following commutative diagram \begin{equation*}
\xymatrix{ \Sigma'\ar[rr]^{i'}\ar[dr]_{f'} & &\hat M \ar[dl]^{\hat j}\\
& \Sigma_0&}.
\end{equation*}
The proof is now completed by taking $M'=\hat M$ and $p'=\hat p$.
\end{proof}

Now we are ready to prove Theorem \ref{Thm: main 1}.

\begin{proof}[Proof of Theorem \ref{Thm: main 1}]
From Lemma \ref{Lem: reduction} we can assume that $i:\Sigma\to M$ is an embedding and that there is a map $j:M\to \Sigma_0$ such that the following diagram
\begin{equation*}
\xymatrix{ \Sigma\ar[rr]^{i}\ar[dr]_{f} & &M \ar[dl]^{j}\\
& \Sigma_0&}
\end{equation*}
is commutative. Possibly passing to a two-sheeted covering, we can assume that both $\Sigma$ and $M$ are orientable and so it makes sense to consider homology groups of $\Sigma$ and $M$ with integer coefficients. Fix a marked point $\sigma$ in $\Sigma$ and we denote $m=i(\sigma)$ and $\sigma_0=f(\sigma)$.
In the following, we consider the covering $\tilde p:(\tilde M,\tilde m)\to (M,m)$ such that
\begin{equation}\label{Eq: lift fundamental group}
\tilde p_*(\pi_1(\tilde M,\tilde m))=i_*(\pi_1(\Sigma,\sigma))\subset \pi_1(M,m).
\end{equation}
By lifting we can construct an embedding $\tilde i:(\Sigma,\sigma)\to (\tilde M,\tilde m)$ such that the diagram
\begin{equation}\label{Diagram: 1}
\xymatrix{
& (\tilde M,\tilde m)\ar[d]^{\tilde p}\ar[dr]^{\tilde j=j\circ \tilde p}&\\
(\Sigma,\sigma)\ar[ur]^{\tilde i}\ar[r]_{i}& (M,m)\ar[r]^{j}&(\Sigma_0,\sigma_0)}.
\end{equation}
is commutative.

In order to apply the soap bubble method, we need to show that the embedded hypersurface $\tilde \Sigma=\tilde i(\Sigma)$ separates $\tilde M$ into two unbounded components. This can be seen from the following contradiction argument. Assume that $\tilde M-\tilde \Sigma$ has only one component. Then we can find a closed curve $\tilde\gamma$ which intersects $\tilde \Sigma$ once. From \eqref{Eq: lift fundamental group} and the injectivity of $\tilde p_*$, it is clear that the map $\tilde i_*:\pi_1(\Sigma,\sigma)\to \pi_1(\tilde M,\tilde m)$ is surjective. In particular, we can find a closed curve $\gamma$ in $\Sigma$ such that $\tilde i\circ \gamma$ is a closed curve in $\tilde\Sigma$ which is homotopic to $\tilde \gamma$. Notice that $\tilde \Sigma$ has trivial normal bundle and so we can push $\tilde i\circ \gamma$ to one side of $\tilde \Sigma$ such that the new closed curve has no intersection with $\tilde \Sigma$. This is impossible since the intersection number keeps invariant under homotopic. Since $\tilde\Sigma$ is homologically non-trivial, the components of $\tilde M-\tilde\Sigma$ are both unbounded.

Now we are going to deduce a contradiction if there is a complete metric $g$ on $M$ with positive scalar curvature. Lift this metric to $\tilde g$ on $\tilde M$. The idea is to find an orientable hypersurface in $\tilde M$ homologous to $\tilde\Sigma$, which admits a smooth metric with positive scalar curvature, based on the soap bubble method. Since $\tilde\Sigma$ separates $\tilde M$ into two unbounded components, the signed distance function to $\tilde\Sigma$ is well-defined and surjective onto $\mathbf R$. Through mollification we can construct a smooth proper function $\rho:\tilde M\to\mathbf R$ with $\Lip \rho <1$ and $\rho^{-1}(0)=\tilde\Sigma$. Given any small positive constant $\epsilon$, from a simple construction (see \cite{Zhu2020} for instance) we can find a smooth function
$$h_\epsilon:\left(-\frac{1}{n\epsilon},\frac{1}{n\epsilon} \right)\to\mathbf R$$
satisfying $h_\epsilon'<0$,
\begin{equation}\label{Eq: positive outside}
\frac{n}{n-1}h_\epsilon^2+2h_\epsilon'=n(n-1)\epsilon^2\quad \text{in}\quad \left(-\frac{1}{n\epsilon},-\frac{1}{2n}\right]\cup\left[\frac{1}{2n}, \frac{1}{n\epsilon}\right),
\end{equation}
\begin{equation}\label{Eq: error inside}
\sup_{-\frac{1}{2n}\leq t\leq \frac{1}{2n}} \left|\frac{n}{n-1}h_\epsilon^2+2h_\epsilon'\right|\leq C(n)\epsilon,
\end{equation}
and
\begin{equation}\label{Eq: barrier}
\lim_{t\to\mp \frac{1}{n\epsilon}} h_\epsilon(t)=\pm\infty.
\end{equation}
Denote $\Omega_0=\{\rho <0\}$. We consider the functional
$$
\mathcal B_\epsilon(\Omega)=\mathcal H^{n-1}(\partial^*\Omega)-\int_{ M}(\chi_\Omega-\chi_{\Omega_0})h_\epsilon\circ\rho\,\mathrm d\mathcal H^n
$$
defined on
$$
\mathcal C_\epsilon=\left\{\text{Caccipoli set}\,\Omega\subset \tilde M:\Omega\Delta\Omega_0\Subset\phi^{-1}\left(- \frac{1}{n\epsilon},\frac{1}{n\epsilon}\right)\right\}.
$$
Notice that the subset $K=\rho^{-1}\left([-\frac{1}{2n},\frac{1}{2n}]\right)$ is compact and so the scalar curvature $R(\tilde g)$ in $K$ is bounded below by a positive constant. Combined with \eqref{Eq: positive outside} and \eqref{Eq: error inside}, we can take $\epsilon$ to be small enough such that the quantity
$$
R(\tilde g)+\left(\frac{n}{n-1}h_\epsilon^2+2h_\epsilon'\right)\circ \rho
$$
is positive everywhere. Since the dimension of $\tilde M$ is no greater than $7$, we can find a smooth minimizer $\Omega$ in $\mathcal C_\epsilon$ for the functional $\mathcal B_\epsilon$ with \eqref{Eq: barrier} as a barrier condition (refer to \cite{Zhu2020} for details). It is not difficult to compute that $\partial\Omega$ satisfies the following stability inequality
\begin{equation}\label{Eq: stability}
\int_{\partial\Omega}|\nabla\phi|^2\mathrm d\sigma\geq \frac{1}{2}\int_{\partial\Omega}\left(R(\tilde g)-R_{\partial\Omega}+\left(\frac{n}{n-1}h_\epsilon^2+2h_\epsilon'\right)\circ \rho\right)\phi^2\,\mathrm d\sigma
\end{equation}
for all $\phi$ in $C^\infty(M)$, where $R_{\partial\Omega}$ is the scalar curvature of $\partial\Omega$ with the induced metric from $(\tilde M,\tilde g)$. From \eqref{Eq: stability} one can deduce that the conformal Laplace operator of $\partial\Omega$ is positive and so $\partial\Omega$ admits a smooth conformal metric with positive scalar curvature. From the definition of class $\mathcal C_\epsilon$, we have
$$\partial\Omega=\tilde\Sigma+\partial(\Omega-\Omega_0)$$
in the chain level, which implies $\partial\Omega$ is homologous to $\tilde\Sigma$. At this stage, the desired contradiction is close at hand. Since $\partial\Omega$ is homologous to $\tilde\Sigma$, the map $\tilde j:\tilde M\to \Sigma_0$ in diagram \eqref{Diagram: 1} restricted to $\partial\Omega$ gives a non-zero degree map from $\partial\Omega$ to $\Sigma_0$. This yields that $\partial\Omega$ is an essential closed manifold over $\Sigma_0$. From our assumption $\Sigma_0\in\mathcal C_{deg}$ we conclude that $\partial\Omega$ cannot admit any smooth metric with positive scalar curvature and we obtain a contradiction.

{Finally let us prove the rigidity part. Given a complete metric $g$ on $M$ with nonnegative scalar curvature, the obstruction for positive scalar curvature metric combined with Kazdan's deformation theorem \cite{Kazdan82} yields that the metric $g$ is Ricci flat. From previous proof the covering $\tilde M$ has two ends and the splitting theorem \cite{CG1971} yields that $\tilde M$ splits as the Riemannian product $\tilde\Sigma'\times \mathbf R$, where $\tilde \Sigma'$ is a Ricci-flat essential closed manifold over $\Sigma_0$.  From \cite[Theorem 1.4]{FW1974}, we can assume $\tilde\Sigma'=\tilde\Sigma_1'\times T^k$ after passing to a covering without loss of generality, where $\tilde\Sigma_1'$ is closed and simply connected. All we need to show is $k=n-1$. Denote $f:(\tilde\Sigma',\tilde\sigma')\to (\Sigma_0,\sigma_0)$ to be the non-zero degree map associated to $\tilde\Sigma$ and consider the covering $\tilde p:(\tilde\Sigma_0,\tilde\sigma_0)\to (\Sigma_0,\sigma_0)$ such that
$$f_*(\pi_1(\tilde\Sigma',\tilde\sigma'))=\tilde p_*(\pi_1(\tilde\Sigma_0,\tilde\sigma_0)).$$
Since we can lift $f$ to a non-zero degree map $\tilde f:\tilde\Sigma'\to \tilde\Sigma_0$, we see that $\tilde\Sigma_0$ is closed and so the top dimensional homology group $H_{n-1}(\tilde\Sigma_0)\neq 0$. On the other hand, since aspherical manifolds have no torsion elements in their fundamental groups, the fundamental group of $\tilde\Sigma_0$ has to be $\mathbf Z^l$ with $l\leq k$. It is a basic topological fact that the homotopy type of an aspherical manifolds is uniquely determined by its fundamental group. Therefore, $\tilde\Sigma_0$ must be homotopic to the $l$-torus $T^l$ with $l=n-1$ and this implies $k=n-1$.}
\end{proof}

\section{Proof for Proposition \ref{Prop: main 2} and \ref{Prop: main 3}}\label{Sec: 3}

This section devotes to a detailed proof for Proposition \ref{Prop: main 2} and \ref{Prop: main 3}.
\begin{proof}[Proof of Proposition \ref{Prop: main 2}]
From the definition of the generalized connected sum, we have
$$(T^n,i)\#_{T^k}(M_2,i_2)=T^k\times (T^{n-k}-B)\sqcup_{\Phi} \left(M_2-U_2\right),$$
where $U_2$ is the tubular neighborhood of $i_2(\Sigma)$ in $M_2$ and the gluing map $\Phi$ is a fiber preserving diffeomorphism between $\partial(T^k\times B)$ and $\partial U_2$.
Notice that each $\mathbf S^1$ is in good pairing with some $T^{n-k-1}$ in $T^{n-k}-B$ as shown in Figure \ref{Fig: 1}.

\begin{figure}[htbp]
\centering
\includegraphics[width=5cm]{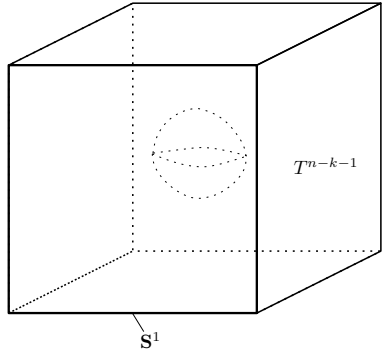}
\caption{$T^{n-k}-B$ can be viewed as the quotient space of punctured cube $I^{n-k}-B$ by identifying opposite faces. Faces and edges has nice intersection property. For example, the right face represents a linear $T^{n-k-1}$ in $T^k-B$, which has intersection number one with the circle $\mathbf S^1$ represented by the horizon edge.}
\label{Fig: 1}
\end{figure}

Fix one such $T^{n-k-1}$ and denote it by $T^{n-k-1}_0$. We would like to show the incompressiblity of the hypersurface $\Sigma =T^k\times T^{n-k-1}_0$ in $(T^n,i)\#_{T^k}(M_2,i_2)$. Fix a point $p$ in $\Sigma$. Then the fundamental group $\pi_1(\Sigma,p)$ is isometric to $\mathbf Z^{n-1}$ in the sense that every closed curve $\gamma:(\mathbf S^1,1)\to (\Sigma,p)$ is homotopic to a closed curve in the form of
$$
\gamma_1^{c_1}\ast \gamma_2^{c_2}\ast \cdots\ast \gamma_{n-1}^{c_{n-1}},\quad c_i \in \mathbf Z,
$$
where $\gamma_i$ are coordinate circles in $\Sigma$ passing through $p$.

In the following, we are going to show that $\gamma$ is homotopically trivial in $\Sigma$ under the assumption that $\gamma$ is homotopic to a point in $(T^n,i)\#_{T^k}(M_2,i_2)$. Based on the nice pairing property above, for each $i\geq k+1$ we can find $T^{n-k-1}_i$ in $T^{n-k}-B$ such that
$$
c_i=[\gamma]\cdot [T^k\times T^{n-k-1}_i]=0.
$$
This means that the closed curve $\gamma$ can be homotopic to a closed curve in the $T^k$-component of $\Sigma$, still denoted by $\gamma$. Notice that $\gamma$ can be homotopic to a closed curve $\bar \gamma$ in the boundary $\partial\left(T^k\times B\right)$. Denote $c_1:T^k\times B\to i(T^k)$ and $c_2:U_2\to i_2(T^k)$ to be the contraction maps associated to tubular neighborhoods. Since the diffeomorphism $\Phi$ is fiber-preserving, we have the following commutative diagram
\begin{equation*}
\xymatrix{\partial(T^k\times B)\ar[r]^{\Phi}\ar[d]^{\partial c_1}&\partial U_2\ar[d]^{\partial c_2}\\
i(T^k)\ar[r]^{i_2\circ i^{-1}}&i_2(T^k).}
\end{equation*}
If the closed curve $\gamma$ cannot be homotopic to a point in the $T^k$-component of $\Sigma$,
then the closed curve $\bar\gamma$ is homotopically non-trivial in the $T^k$-component of $\partial(T^k\times B)$ and the same thing holds for the closed curve $(\partial c_1)\circ \bar \gamma$ in $i(T^k)$. From above commutative diagram, 
we conclude that the closed curve $\Phi\circ \bar\gamma$ is homotopic to a homotopically non-trivial closed curve in $i_2(T^k)$. Now
the incompressible property implies that $\Phi\circ \bar\gamma$ is also homotopically non-trivial in $M_2$, let alone $M_2-U_2$. It then follows from Lemma \ref{Lem: no contraction extension} and Lemma \ref{Lem: lifting property example} that $\Phi\circ\bar \gamma$ cannot be homotopically trivial in $(T^n,i)\#_{T^k}(M_2,i_2)$, but this is impossible since the closed curve $\Phi\circ \bar\gamma$ after gluing is homotopic to $\gamma$, which has been assumed to be null homotopic in $(T^n,i)\#_{T^k}(M_2,i_2)$. As a consequence, the closed curve $\gamma$ is homotopic to a point in $T^k$ as well as in $\Sigma$. Now our main Theorem \ref{Thm: main 1} can be applied to obtain the desired consequence.
\end{proof}

Next let us give a detailed proof for Proposition \ref{Prop: main 3}.

\begin{proof}[Proof of Proposition \ref{Prop: main 3}]
Let us start with the following special case.

{\it Case 1. $M_2$ is closed and orientable.} From Figure \ref{Fig: 1}, we can find $(n-1)$ pairs of $T^{n-2}$ and $\mathbf S^1$ in $T^{n-1}-B$ with intersection number one. Let us label them as $T^{n-2}_l$ and $\mathbf S^1_l$ for $l=1,2,\ldots,n-1$ such that $\mathbf S^1_l$ is exactly the $l$-th circle component of $T^{n-1}-B$. Denote $\beta_l$ to be the cohomology class coming from the Poincar\'e dual of $\mathbf S^1\times T^{n-2}_l$ in the generalized connected sum
$$(T^n,i)\#_{\mathbf S^1}(M_2,i_2)=\mathbf S^1\times (T^{n-1}-B)\sqcup_\Phi (M_2-U_2).$$
Notice that $\mathbf S^1\times T^{n-2}_l$ has intersection number one with $\mathbf S^1_l$. So every class $\beta_l$ is non-trivial in $H^1(\bar M,\mathbf Z)$. Here and in the sequel, we will write $\bar M=(T^n,i)\#_{\mathbf S^1}(M_2,i_2)$ for short.

Now we would like to show that the homology class
$$
\tau:=[M]\frown(\beta_1\smile\beta_2\smile\cdots\smile\beta_{n-2})
$$
in $H_2(\bar M,\mathbf Z)$ does not lie in the image of Hurewicz homomorphism $\pi_2(\bar M)\to H_2(\bar M,\mathbf Z)$. Otherwise, we can find a cycle $C$ in $\bar M$ consisting of $2$-spherical components such that it represents the class $\tau$. Now let us investigate the intersection of the cycle $C$ and the hypersurface $\Sigma:=\mathbf S^1\times T^{n-2}_{n-1}$. From basic algebraic topology, the intersection $C\cap \Sigma$ is a union of closed curves in $\Sigma$, which is homologous to the $\mathbf S^1$-component of $\Sigma$. Clearly, each of these curves is homotopic to a point through a deformation in some $2$-spherical component of $\Sigma$. Notice that $\Sigma$ is a $(n-1)$-torus and so the Hurewicz homomorphism $\pi_1(\Sigma)\to H_1(\Sigma,\mathbf Z)$ is an isomorphism. As a result, the $\mathbf S^1$-component of $\Sigma$ is null homotopic in $\bar M$, and the same thing holds for the $\mathbf S^1$-component of $\partial \left(\mathbf S^1\times B\right)$. Denote $\gamma$ to be a closed curve in $\partial \left(\mathbf S^1\times B\right)$ homotopic to the $\mathbf S^1$-component. It follows from Lemma \ref{Lem: no contraction extension} and Lemma \ref{Lem: lifting property example} that the closed curve $\Phi\circ\gamma$ must shrink to a point in $M_2-U_2$. 
From the commutative diagram
\begin{equation*}
\xymatrix{\partial(\mathbf S^1\times B)\ar[r]^{\Phi}\ar[d]^{\partial c_1}&\partial U_2\ar[d]^{\partial c_2}\\
i(\mathbf S^1)\ar[r]^{i_2\circ i^{-1}}&i_2(\mathbf S^1).}
\end{equation*}
we see that $\Phi\circ \gamma$ is homotopic to $i_2:\mathbf S^1\to M_2$ and so the map $i_2:\mathbf S^1\to M_2$ is homotopically trivial. This leads to a contradiction.

The above discussion yields that $\bar M$ is a closed Schoen-Yau-Schick manifold and then the desired consequence comes from Schoen-Yau's work in \cite{SY2017}.

{\it General case when $M_2$ may be non-compact.} We still use the soap bubble method to overcome the difficulty from non-compactness. As before, we can find $(n-1)$ pairs of $T^{n-2}$ and $\mathbf S^1$ in $T^{n-1}-B$ with intersection number one, labelled by $T^{n-2}_l$ and $\mathbf S^l_l$ with $l=1,2,\ldots,n-1$. Let $\bar M$ be the covering of $(T^n,i)\#_{\mathbf S^1}(M_2,i_2)$ from an unwinding along $\mathbf S^1_1$. Intuitively, one can imagine $\bar M$ as follows. First we take the covering $\mathbf R\times T^{n-2}$ of $T^{n-1}$. After taking away infinitely many disjoint balls $\{B_i\}_{i\in\mathbf Z}$ and production with the circle, the manifold
$$
\mathbf S^1\times\left(\mathbf R\times T^{n-2}-\bigcup_{i\in\mathbf Z} B_i\right)
$$
provides a covering of $\mathbf S^1\times (T^{n-1}-B)$. Next we glue infinitely many copies of $M_2-U_2$ to this manifold along each boundary component $\mathbf S^1\times \partial B_i$ and then the resulting manifold is our desired covering $\bar M$. If we lift the $(n-1)$-torus $\mathbf S^1\times T^{n-2}_1$ to $\bar M$, we can obtain an embedded hypersurface, denoted by $\bar T^{n-1}_1$, in $\bar M$ that seperates $\bar M$ into two unbounded components.

Now let us deduce a contradiction under the assumption that the generalized connected sum $(T^n,i)\#_{\mathbf S^1}(M_2,i_2)$ admits a complete metric with positive metric. By lifting, we obtain a complete metric with positive scalar curvature on $\bar M$. With the exactly same argument as in the proof of Theorem \ref{Thm: main 1}, we can find an embedded hypersurface $\bar\Sigma$ homologous to $\bar T^{n-1}_1$ in $\bar M$ such that $\bar\Sigma$ admits a smooth metric with positive scalar curvature. The image of $\bar\Sigma$ under the covering map is now a hypersurface $\Sigma$ in $(T^n,i)\#_{\mathbf S^1}(M_2,i_2)$ homologous to $\mathbf S^1\times T^{n-2}_1$ with the same property. Since $\Sigma$ is compact, we can proceed our discussion in a large compact subset $K$ of $(T^n,i)\#_{\mathbf S^1}(M_2,i_2)$ containing $\Sigma$. In this compact subset $K$, it still makes sense to take the Poincar\'e dual $\beta_l$ of $\mathbf S^1\times T^{n-2}_l$. Through the same argument as in Case 1, one can conclude that $\Sigma$ is a closed orientable Schoen-Yau-Schick manifold. This leads to a contradiction since $\Sigma$ admits a smooth metric with positive scalar curvature.
\end{proof}

\section{Positive mass theorems with an incompressible condition}
\subsection{A weighted Sobolev Inequality}\label{Sec: Sobolev}
In this subsection, we denote $\mathbf R^n$ to be the Euclidean $n$-space with $n\geq 2$ and $B$ to be the {\it open} unit ball in $\mathbf R^n$. Also we denote $r$ to be the radical distance function on $\mathbf R^n$. Given positive constants $d_i$, $i=1,2,\ldots,k$, we use $T^k$ to denote the product $k$-torus
$$
\mathbf S^1(d_1)\times \mathbf S^1(d_2)\times \cdots \times \mathbf S^1(d_k).
$$

The purpose of this section is to prove the following weighted Sobolev inequality.
\begin{proposition}\label{Prop: Sobolev}
There is a universal constant $C$ depending only on $n$, $k$ and $d_i$, $i=1,2,\ldots,k$, such that the inequality
\begin{equation}\label{Eq: sobolev weighted}
    \left(\int_{(\mathbf R^n-B)\times T^k}|f|^{\frac{n+k}{n+k-1}}r^{\frac{-k}{n+k-1}}\,\mathrm dx\mathrm dt\right)^{\frac{n+k-1}{n+k}}\leq C\int_{(\mathbf R^n-B)\times T^k}|\nabla f|\,\mathrm dx\mathrm dt
\end{equation}
holds for any $f$ in $C_c^1((\mathbf R^n-B)\times T^k)$. Here the function $f$ is allowed to take non-zero values on the inner boundary $\partial B\times T^k$.
\end{proposition}

Before the proof we need the following lemma for preparation.
\begin{lemma}\label{Lem: sobolev average}
If $f$ is a function in $C_c^1((\mathbf R^n-B)\times T^k)$ satisfying
\begin{equation}\label{Eq: zero average}
\int_{\mathbf S^1(d_i)} f(x,t_1,t_2,\ldots,t_k)\,\mathrm dt_i=0,\quad i=1,2,\ldots,k,
\end{equation}
then we have
\begin{equation}\label{Eq: sobolev no weight}
 \left(\int_{(\mathbf R^n-B)\times T^k}|f|^{\frac{n+k}{n+k-1}}\,\mathrm dx\mathrm dt\right)^{\frac{n+k-1}{n+k}}\leq \int_{(\mathbf R^n-B)\times T^k}|\nabla f|\,\mathrm dx\mathrm dt.
\end{equation}
\end{lemma}
\begin{proof}
The proof is standard and it can be found in \cite{Evans2010}. Here we just include the arguments for completeness. First let us do zero extension of $f$ and $\nabla f$ into $B\times T^k$. Although this breaks the continuity of these functions, it doesn't matter in our discussion below. For convenience, we view the extended $f$ and $\nabla f$ as functions over $\mathbf R^n\times [0,d_1]\times\cdots\times [0,d_k]$, and we denote
\begin{equation*}
I_j=\left\{
\begin{array}{cc}
             \mathbf R, & 1\leq j\leq n; \\
             {[0,d_{j-n}]}, & n+1\leq j\leq n+k.
           \end{array}
\right.
\end{equation*}
Since $f$ has compact support and satisfies \eqref{Eq: zero average}, it holds for any $1\leq j\leq n+k$ that
\begin{equation}\label{Eq: basic inequality}
|f(x_1,\ldots, x_{n+k})|\leq \int_{I_j}|\nabla f(x_1,\ldots,y_j,\ldots,x_{n+k})|\mathrm dy_j.
\end{equation}
As a result, we have
$$
|f(x_1,\ldots, x_{n+k})|^{\frac{n+k}{n+k-1}}\leq \prod_{j=1}^{n+k}\left(\int_{I_j}|\nabla f(x_1,\ldots,y_j,\ldots,x_{n+k})|\mathrm dy_j\right)^{\frac{1}{n+k-1}}.
$$
Integrating this inequality over $I_1$ and applying the H\"older inequality, we can obtain
\begin{equation*}
\begin{split}
\int_{I_1}|f(x_1,\ldots, x_{n+k})&|^{\frac{n+k}{n+k-1}}\mathrm dx_1
\leq\left(\int_{I_1}|\nabla f(x_1,\ldots,y_j,\ldots,x_{n+k})|\mathrm dy_1\right)^{\frac{1}{n+k-1}}\\
&\cdot \prod_{j=2}^{n+k}\left(\int_{I_j}\int_{I_1}|\nabla f(x_1,\ldots,y_j,\ldots,x_{n+k})|\mathrm dy_j\mathrm dx_1\right)^{\frac{1}{n+k-1}}.
\end{split}
\end{equation*}
Now we continue to integrate above inequality over the rest $I_j$ step by step with the use of the H\"older inequality. Eventually we will arrive at
\begin{equation*}
\begin{split}
\left(\int_{\mathbf R^n\times [0,d_1]\times\cdots\times [0,d_k]}|f|^{\frac{n+k}{n+k-1}}\,\mathrm dx\right)^{\frac{n+k-1}{n+k}}\leq \int_{\mathbf R^n\times [0,d_1]\times\cdots\times [0,d_k]}|\nabla f|\,\mathrm dx,
\end{split}
\end{equation*}
which is exactly the desired inequality \eqref{Eq: sobolev no weight}.
\end{proof}

Next we are going to prove Proposition \ref{Prop: Sobolev} inductively with an averaging trick. In fact, we can establish the following more general result.
\begin{proposition}
If $f$ is a function in $C_c^1((\mathbf R^n-B)\times T^k)$ satisfying
\begin{equation*}
\int_{\mathbf S^1(d_i)} f(x,t_1,t_2,\ldots,t_k)\,\mathrm dt_i=0,\quad i=1,2,\ldots,l,
\end{equation*}
with $0\leq l\leq k$, then there is a universal constant $C$ depending only on $n$, $k$, $l$ and $d_i$, $i=1,2,\ldots,k$, such that
\begin{equation}
 \left(\int_{(\mathbf R^n-B)\times T^k}|f|^{\frac{n+k}{n+k-1}}r^{\frac{l-k}{n+k-1}}\,\mathrm dx\mathrm dt\right)^{\frac{n+k-1}{n+k}}\leq C\int_{(\mathbf R^n-B)\times T^k}|\nabla f|\,\mathrm dx\mathrm dt.
\end{equation}
\end{proposition}
\begin{proof}
The proof will be completed from an induction on the difference $k-l$. It follows from Lemma \ref{Lem: sobolev average} that above proposition holds when $k-l=0$. From induction we can assume that above proposition holds when $k-l=m$ and then deal with the case $k-l=m+1$. For any $f$ in $C_c^1((\mathbf R^n-B)\times T^k)$ satisfying
\begin{equation*}
\int_{\mathbf S^1(d_i)} f(x,t_1,t_2,\ldots,t_k)\,\mathrm dt_i=0,\quad i=1,2,\ldots,l,
\end{equation*}
we write it as
$
f=f_1+f_2
$
with
$$
f_1=\frac{1}{d_{l+1}}\int_{\mathbf S^1(d_{l+1})}f(x,t_1,t_2,\ldots,t_k)\,\mathrm dt_{l+1}.
$$
Clearly, $f_2$ is a function in $C_c^1((\mathbf R^n-B)\times T^{k-1})$ satisfying
\begin{equation*}
\int_{\mathbf S^1(d_i)} f_2(x,t_1,t_2,\ldots,t_k)\,\mathrm dt_i=0,\quad i=1,2,\ldots,l+1.
\end{equation*}
Notice that the gradients of $f_1$ and $f_2$ are both controlled by the gradient of $f$. So we have
\begin{equation}\label{Eq: estimate f2}
\begin{split}
& \left(\int_{(\mathbf R^n-B)\times T^k}|f_2|^{\frac{n+k}{n+k-1}}r^{-\frac{m+1}{n+k-1}}\,\mathrm dx\mathrm dt\right)^{\frac{n+k-1}{n+k}}\\\leq& \left(\int_{(\mathbf R^n-B)\times T^k}|f_2|^{\frac{n+k}{n+k-1}}r^{-\frac{m}{n+k-1}}\,\mathrm dx\mathrm dt\right)^{\frac{n+k-1}{n+k}} \\
\leq&C\int_{(\mathbf R^n-B)\times T^k}|\nabla f_2|\,\mathrm dx\mathrm dt
\leq  C\int_{(\mathbf R^n-B)\times T^k}|\nabla f|\,\mathrm dx\mathrm dt.
\end{split}
\end{equation}
Here and in the sequel, the symbol $C$ is always denoted to be a universal constant depending only on $n$, $k$, $l$ and $d_i$, $i=1,2,\ldots,k$.
On the other hand, the function $f_1$ can be viewed as a function in $C_c^1((\mathbf R^n-B)\times T^{k-1})$ and it satisfies
\begin{equation*}
\int_{\mathbf S^1(d_i)} f_1(x,t_1,t_2,\ldots,t_{k-1})\,\mathrm dt_i=0,\quad i=1,2,\ldots,l.
\end{equation*}
From inductive assumption we see
\begin{equation*}
\int_{(\mathbf R^n-B)\times T^k}|f_1|^{\frac{n+k}{n+k-1}}r^{-\frac{m+1}{n+k-1}}\mathrm dx\mathrm dt
= d_{l+1}\int_{(\mathbf R^n-B)\times T^{k-1}}|f_1|^{\frac{n+k}{n+k-1}}r^{-\frac{m+1}{n+k-1}}\mathrm dx\mathrm dt.
\end{equation*}

Now we have to divide the discussion into two cases:

{\it Case 1.} $l>0$. Since we have $m<k-1$ in this case, it follows from H\"older inequality that
\begin{equation*}
\begin{split}
&\left(\int_{(\mathbf R^n-B)\times T^{k-1}}|f_1|^{\frac{n+k}{n+k-1}}r^{-\frac{m+1}{n+k-1}}\mathrm dx\mathrm dt\right)^{\frac{n+k-1}{n+k}}\\
\leq&\left(\int_{(\mathbf R^n-B)\times T^{k-1}}|f_1|^{\frac{n+k-1}{n+k-2}}r^{-\frac{m}{n-k-2}}\mathrm dx\mathrm d t\right)^{\frac{n+k-2}{n+k-1}}\\
&\qquad\qquad\cdot\left(\int_{(\mathbf R^n-B)\times T^{k-1}}r^{-n-k+m+1}\mathrm dx\mathrm dt\right)^{\frac{1}{(n+k)(n+k-1)}}\\
\leq & C\int_{(\mathbf R^n-B)\times T^{k-1}}|\nabla f_1|\mathrm dx\mathrm dt
\leq  C\int_{(\mathbf R^n-B)\times T^{k}}|\nabla f_1|\mathrm dx\mathrm dt.
\end{split}
\end{equation*}
This yields
$$
\left(\int_{(\mathbf R^n-B)\times T^k}|f_1|^{\frac{n+k}{n+k-1}}r^{-\frac{m+1}{n+k-1}}\mathrm dx\mathrm dt\right)^{\frac{n+k-1}{n+k}}\leq C\int_{(\mathbf R^n-B)\times T^k}|\nabla f|\mathrm dx\mathrm dt.
$$

{\it Case 2.} $l=0$. A more delicate analysis is involved in this case since $r^{-n}$ is not integrable in $\mathbf R^n-B$. The basic idea is to do further averaging procedure to $f_1$. Write $f_1=f_{11}+f_{12}$ with
$$
f_{11}=\frac{1}{d_{l+2}}\int_{\mathbf S^1(d_{l+2})}f_1(x,t_1,\cdots,t_k)\mathrm dt_{l+2}.
$$
It is easy to check that we can obtain
$$
\left(\int_{(\mathbf R^n-B)\times T^k}|f_{12}|^{\frac{n+k}{n+k-1}}r^{-\frac{k}{n+k-1}}\mathrm dx\mathrm dt\right)^{\frac{n+k-1}{n+k}}\leq C\int_{(\mathbf R^n-B)\times T^k}|\nabla f_{12}|\mathrm dx\mathrm dt
$$
with the help of the inductive assumption and the H\"older inequality. However, the critical exponent appears when we try to deal with the function $f_{11}$. So we continue to do averaging procedure to $f_{11}$, and finally we reduce the desired estimate to a function $\bar f$ in $C_c^1((\mathbf R^n-B)\times T^k)$ satisfying
$$
\bar f(x,t_1,\ldots,t_k)=g(x).
$$
It suffices to show
$$
\left(\int_{\mathbf R^n-B}|g|^{\frac{n+k}{n+k-1}}r^{-\frac{k}{n+k-1}}\mathrm dx\right)^{\frac{n+k-1}{n+k}}\leq C\int_{\mathbf R^n-B}|\nabla g|\,\mathrm dx.
$$
Actually we have the following Hardy inequality
\begin{equation}\label{Eq: hardy}
\int_{\mathbf R^n-B}|g|r^{-1}\,\mathrm dx\leq \frac{1}{n-1}\int_{\mathbf R^n-B}|\nabla g|\,\mathrm dx.
\end{equation}
To see this, let us do some calculation.
\begin{equation*}
\begin{split}
\int_{\mathbf R^n-B}|g|r^{-1}\,\mathrm dx&=\int_{\mathbf S^{n-1}}\mathrm d\sigma \int_1^{+\infty}|g|r^{n-2}\,\mathrm dr\\
&\leq\int_{\mathbf S^{n-1}}\mathrm d\sigma \int_1^{+\infty}\mathrm dr\int_{r}^{+\infty}|\nabla g|(s,\theta)r^{n-2}\,\mathrm ds\\
&=\int_{\mathbf S^{n-1}}\mathrm d\sigma \int_1^{+\infty}\mathrm ds\int_{1}^s|\nabla g|(s,\theta)r^{n-2}\,\mathrm dr\\
&\leq\frac{1}{n-1}\int_{\mathbf S^{n-1}}\mathrm d\sigma\int_1^{+\infty}|\nabla g|(s,\theta)s^{n-1}\,\mathrm ds\\
&=
\frac{1}{n-1}\int_{\mathbf R^n-B}|\nabla g|\,\mathrm dx.
\end{split}
\end{equation*}
Now the desired inequality comes from a simple interpolation of \eqref{Eq: sobolev no weight} and \eqref{Eq: hardy}. Namely, we have
\begin{equation*}
\begin{split}
&\left(\int_{\mathbf R^n-B}|g|^{\frac{n+k}{n+k-1}}r^{-\frac{k}{n+k-1}}\mathrm dx\right)^{\frac{n+k-1}{n+k}}\\
\leq &\left(\int_{\mathbf R^n-B}|g|^{\frac{n}{n-1}}\mathrm dx\right)^{\frac{n}{n+k}}\left(\int_{\mathbf R^n-B}|g|r^{-1}\,\mathrm dx\right)^{\frac{k}{n+k}}\\
\leq &C\int_{\mathbf R^n-B}|\nabla g|\,\mathrm dx.
\end{split}
\end{equation*}
Now we complete the proof.
\end{proof}

For further generalization of Proposition \ref{Prop: Sobolev}, we point out that the validity of Proposition \ref{Prop: Sobolev} essentially relies on
\begin{itemize}
\item[(i)] the no-weighted Sobolev inequality
\begin{equation*}
 \left(\int_{(\mathbf R^n-B)\times T^k}|f|^{\frac{n+k}{n+k-1}}\,\mathrm dx\mathrm dt\right)^{\frac{n+k-1}{n+k}}\leq \int_{(\mathbf R^n-B)\times T^k}|\nabla f|\,\mathrm dx\mathrm dt
\end{equation*}
for all $f$ in $C_c^1((\mathbf R^n-B)\times T^k)$ with
$$
\int_{\mathbf S^1(d_i)} f(x,t_1,t_2,\ldots,t_k)\,\mathrm dt_i=0,\quad i=1,2,\ldots,k.
$$
\item[(ii)] the Hardy inequality
$$
\int_{\mathbf R^n-B}|f|r^{-1}\mathrm dx\leq \frac{1}{n-1}\int_{\mathbf R^n-B}|\nabla f|\,\mathrm dx
$$
for all $f$ in $C_c^1(\mathbf R^n-B)$.
\end{itemize}

It is easy to generalize above inequalities to $L^p$-version. Namely we have
\begin{lemma}
If $f$ is a function in $C_c^1((\mathbf R^n-B)\times T^k)$ satisfying
\begin{equation*}
\int_{\mathbf S^1(d_i)} f(x,t_1,t_2,\ldots,t_k)\,\mathrm dt_i=0,\quad i=1,2,\ldots,k,
\end{equation*}
then for any $1\leq p<n+k$ there is a universal constant $C$ depending only on $n$, $k$, $p$ and $d_i$, $i=1,2,\ldots,k$, such that
\begin{equation*}
 \left(\int_{(\mathbf R^n-B)\times T^k}|f|^{\frac{(n+k)p}{n+k-p}}\,\mathrm dx\mathrm dt\right)^{\frac{n+k-p}{(n+k)p}}\leq C\left(\int_{(\mathbf R^n-B)\times T^k}|\nabla f|^p\,\mathrm dx\mathrm dt\right)^{\frac{1}{p}}.
\end{equation*}
\end{lemma}
\begin{proof}
The proof is standard. First we point out that the key to inequality \eqref{Eq: sobolev no weight} is the validity of \eqref{Eq: basic inequality}, which only requires that $f$ has zero in all $\mathbf S^1$ slices. Therefore, if $f$ is a function in $C_c^1((\mathbf R^n-B)\times T^k)$ satisfying
\begin{equation*}
\int_{\mathbf S^1(d_i)} f(x,t_1,t_2,\ldots,t_k)\,\mathrm dt_i=0,\quad i=1,2,\ldots,k,
\end{equation*}
then \eqref{Eq: sobolev no weight} can be applied to the function $|f|^{\frac{p(n+k-1)}{n+k-p}}$. Then we have
\begin{equation*}
\begin{split}
&\left(\int_{(\mathbf R^n-B)\times T^k}|f|^{\frac{p(n+k)}{n+k-p}}\,\mathrm dx\mathrm dt\right)^{\frac{n+k-1}{n+k}}\\
\leq &\frac{p(n+k-1)}{n+k-p}\int_{(\mathbf R^n-B)\times T^k}|f|^{\frac{(n+k)(p-1)}{n+k-p}}|\nabla f|\,\mathrm dx\mathrm dt\\
\leq &\frac{p(n+k-1)}{n+k-p}\left(\int_{(\mathbf R^n-B)\times T^k}|f|^{\frac{p(n+k)}{n+k-p}}\,\mathrm dx\mathrm dt\right)^{\frac{p-1}{p}}\left(\int_{(\mathbf R^n-B)\times T^k}|\nabla f|^p\,\mathrm dx\mathrm dt\right)^{\frac{1}{p}}.
\end{split}
\end{equation*}
This gives
\begin{equation*}
 \left(\int_{(\mathbf R^n-B)\times T^k}|f|^{\frac{(n+k)p}{n+k-p}}\,\mathrm dx\mathrm dt\right)^{\frac{n+k-p}{(n+k)p}}\leq \frac{p(n+k-1)}{n+k-p}\left(\int_{(\mathbf R^n-B)\times T^k}|\nabla f|^p\,\mathrm dx\mathrm dt\right)^{\frac{1}{p}}.
\end{equation*}
This is just the desired inequality.
\end{proof}
\begin{lemma}
Let $f$ be a function in $C_c^1(\mathbf R^n-B)$. Then for any $1\leq p<n$ we have
$$
\int_{\mathbf R^n-B}|f|^pr^{-p}\,\mathrm dx\leq\left(\frac{p}{n-p}\right)^{p}\int_{\mathbf R^n-B}|\nabla f|^p.
$$
\end{lemma}
\begin{proof}
The proof comes from a direct calculation. We have
\begin{equation*}
\begin{split}
\int_{\mathbf R^n-B}|f|^pr^{-p}\,\mathrm dx&=\int_{\mathbf S^{n-1}}\mathrm d\sigma \int_1^{+\infty}|f|^pr^{n-1-p}\,\mathrm dr\\
&\leq p\int_{\mathbf S^{n-1}}\mathrm d\sigma \int_1^{+\infty}\mathrm dr\int_{r}^{+\infty}|f|^{p-1}(s,\theta)|\nabla f|(s,\theta)r^{n-1-p}\,\mathrm ds\\
&=p\int_{\mathbf S^{n-1}}\mathrm d\sigma \int_1^{+\infty}\mathrm ds\int_{1}^s|f|^{p-1}(s,\theta)|\nabla f|(s,\theta)r^{n-1-p}\,\mathrm dr\\
&\leq\frac{p}{n-p}\int_{\mathbf S^{n-1}}\mathrm d\sigma\int_1^{+\infty}|f|^{p-1}(s,\theta)|\nabla f|(s,\theta)s^{n-p}\,\mathrm ds\\
&=
\frac{p}{n-p}\left(\int_{\mathbf S^{n-1}}\mathrm d\sigma\int_1^{+\infty}|f|^{p}(s,\theta)s^{n-1-p}\,\mathrm ds\right)^{\frac{p-1}{p}}\\
&\quad\cdot\left(\int_{\mathbf S^{n-1}}\mathrm d\sigma\int_1^{+\infty}|\nabla f|^{p}(s,\theta)s^{n-1}\,\mathrm ds\right)^{\frac{1}{p}}\\
&=\frac{p}{n-p}\left(\int_{\mathbf R^n-B}|f|^pr^{-p}\,\mathrm dx\right)^{\frac{p-1}{p}}\left(\int_{\mathbf R^n-B}|\nabla f|^p\,\mathrm dx\right)^{\frac{1}{p}}.
\end{split}
\end{equation*}
This gives the desired inequality.
\end{proof}
Combining above two lemmas with the averaging trick, we can obtain the following
\begin{proposition}\label{Prop: Sobolev Lp}
For any $1\leq p<n$ there is a universal constant $C$ depending only on $n$, $k$, $p$ and $d_i$, $i=1,2,\ldots,k$, such that the inequality
\begin{equation*}
    \left(\int_{(\mathbf R^n-B)\times T^k}|f|^{\frac{(n+k)p}{n+k-p}}r^{\frac{-kp}{n+k-p}}\,\mathrm dx\mathrm dt\right)^{\frac{n+k-p}{(n+k)p}}\leq C\left(\int_{(\mathbf R^n-B)\times T^k}|\nabla f|^p\,\mathrm dx\mathrm dt\right)^{\frac{1}{p}}
\end{equation*}
holds for any $f$ in $C_c^1((\mathbf R^n-B)\times T^k)$. Here the function $f$ is allowed to take non-zero values on the inner boundary $\partial B\times T^k$.
\end{proposition}
Clearly the Sobolev inequality still holds if one passes to an equivalent metric or a quotient space. Since every closed flat manifold is covered by a flat torus \cite{B1911} \cite{B1912}, we conclude that
\begin{corollary}\label{Cor: Sobolev Lp}
Let $(F,g_{flat})$ be a flat closed $k$-manifold and $g$ be a smooth metric on $(\mathbf R^n-B)\times F$ equivalent to $g_{euc}\oplus g_{flat}$. For any $1\leq p<n$ there is a universal constant $C$ depending only on $n$, $k$, $p$ and $F$, such that the inequality
\begin{equation*}
    \left(\int_{(\mathbf R^n-B)\times F}|f|^{\frac{(n+k)p}{n+k-p}}r^{\frac{-kp}{n+k-p}}\,\mathrm d\mu_{ g}\right)^{\frac{n+k-p}{(n+k)p}}\leq C\left(\int_{(\mathbf R^n-B)\times F}|\nabla f|^p\,\mathrm d\mu_{ g}\right)^{\frac{1}{p}}
\end{equation*}
holds for any $f$ in $C_c^1((\mathbf R^n-B)\times F)$.
\end{corollary}

\subsection{Lohkamp compactification and Proof for Theorem \ref{Thm: main 4}} In this subsection, we are going to prove our main Theorem \ref{Thm: main 4} for asymptotically flat manifolds with fiber $F$ (recall from Definition \ref{Defn: AF with fiber F}). First we point out that the fiber $F$ can be assumed to be $T^k$, $k=\dim F$, from a lifting argument without loss of generality, and this is always the case in this subsection. In the following discussion, $n$ is denoted to be $\dim M$, $k$ is denoted to be $\dim F$ and $d=n-k$.

Let us start with the following lemma, which provides a convenient way to compute the mass $m(M,g,\mathcal E)$.
\begin{lemma}\label{defn of mass}
 Let $(M, g, \mathcal{E})$ be an asymptotically flat manifold with fiber $F$ and arbitrary ends. Fix an orthornormal basis $\{e_{a}\}$ with respect to $g_{0}=g_{euc}\oplus g_F$ consisting of $\{\frac{\partial }{\partial x_{i}},f_{\alpha}\}$ such that $g$ has the expression $g=g_{ab}e_a\otimes e_b$. Then the mass of $(M, g, \mathcal{E})$ is given by
$$
\begin{aligned}
m(M, g, \mathcal{E})
&=\frac{1}{2\left|\mathbf{S}^{d-1}\right| \operatorname{Vol}\left(F, g_{F}\right)} \lim _{\rho \rightarrow+\infty}\int_{S_{\rho}\times F}(\partial_{i}g_{ij}-\partial_{j}g_{aa})\ast\mathrm{d}x^{j}\mathrm{d}\gamma.
\end{aligned}
$$
In above expression, the symbol $a$ runs over all index, $i$ runs over the index on Euclidean space $\mathbf R^d$, $\alpha$ runs over index on $F$, $d\gamma$ is volume form of $g_{F}$ and  $\ast$ is the Hodge star operator corresponding to the Euclidean metric $\mathrm{d}x^{2}$ in $\mathbf R^{d}$.
\end{lemma}
\begin{proof}
Let $\{e_a^*\}$ be the dual frame of $\{e_a\}$. From the definition of the total mass we have
\[
\begin{split}
m(M, g, \mathcal{E})
&=\frac{1}{2\left|\mathbf{S}^{d-1}\right| \operatorname{Vol}\left(F, g_{F}\right)} \lim _{\rho \rightarrow+\infty}\int_{S_{\rho}\times F}(\partial_{a}g_{ab}-\partial_{b}g_{aa})\ast e_b^*\\
&=\frac{1}{2\left|\mathbf{S}^{d-1}\right| \operatorname{Vol}\left(F, g_{F}\right)} \lim _{\rho \rightarrow+\infty}\int_{S_{\rho}\times F}(\partial_{a}g_{aj}-\partial_{j}g_{aa})\ast\mathrm{d}x^{j}\mathrm{d}\gamma.
\end{split}
\]
We obtain the desired formula after applying the divergence theorem to the integral of the first term along $F$.
\end{proof}

Denote $U$ to be an open set such that $\mathcal{E}\subset U $ and $\overline{U-\mathcal{E}}$ is compact.  We have 

\begin{lemma}\label{3}For $p=2$, there exist a constant $C_{U}$ depending only on $g$, $U$, $n$ and $k$ such that for $f \in C_{c}^{1}\left( \overline{U}\right)$ ,
\begin{equation}\label{sobelov on U}
\left(\int_{U}|f|^{\frac{2(d+k)}{d+k-2}} r^{\frac{-2k}{d+k-2}} \mathrm{~d} \mu_{g}\right)^{\frac{d+k-2}{2(d+k)}} \leq C_{U}\left(\int_{U}|\nabla f|^{2} \mathrm{~d} \mu_{g}\right)^{\frac{1}{2}}
\end{equation}
\end{lemma}
\begin{proof} 
Since the metric $g$ is euivalent to $g_0=g_{euc}\oplus g_F$, Corollary \ref{Prop: Sobolev Lp} also holds for $(\mathcal E,g)$. To extend the Sobolev inequality on $(U,g)$, the argument is the same as in \cite[Lemma 3.1]{SY1979}.
\end{proof}

For a function, it can be written as the difference of its positive part and negetive part, i.e. 
$$f=f_{+}-f_{-},f_{+}=\max\{f,0\},f_{-}=\max\{-f,0\}.$$

\begin{proposition}\label{3.2likeSY}
Assume that $f$ is a smooth function on $M$ with compact support in $U$ such that
the negative part $f_{-}$ of $f$ satisfies
\begin{equation}\label{f_{-}control}
\left(\int_{U}\left|f_{-}\right|^{\frac{d+k}{2}} r^{k}\,\mathrm{d} \mu_{g}\right)^{\frac{2}{d+k}}\leq \frac{1}{2C_{U}^{2}},
\end{equation}
here $C_{U}$ is the sobelov constant comes from Lemma $\ref{3}$, and 
\begin{equation}\label{f control}
\left(\int_{U}|f|^{\frac{2 (d+k)}{d+k+2}}  r^{\frac{2k(d+k-2)}{(d+k+2)^{2}}} \mathrm{d} \mu_{g}\right)^{\frac{d+k+2}{2(d+k)}} \leq C_{0}
\end{equation}
holds for some constants $C_{0}$.
Then the equation
\begin{equation}\label{main equation}
\Delta_{g} u-f u=0
\end{equation}
has a positive solution $u$ on $M$ such that $u$ has the expansion
\begin{equation}\label{expansion}
u=1+\frac{A}{ r^{d-2}}+\omega
\end{equation}
with
\begin{equation}\label{A}
A=-\frac{1}{(d-2)\left|\mathbf{S}^{d-1}\right||T^{k}|} \int_{U} f u \,\mathrm{d} \mu_{g}
\end{equation}
and for any $\epsilon>0$ small, $|\omega|+r|\partial \omega|+r^{2}\left|\partial^{2} \omega\right| = O(r^{-(d-1-\epsilon)})$ in $\mathcal{E}$ as $r\rightarrow \infty$. Moreover, we have
\begin{equation}\label{integral on boundary}
\int_{\partial U} \frac{\partial u}{\partial \vec{n}} \mathrm{~d} \sigma_{g}=0 \quad \text { and } \quad \int_{\partial U} u \frac{\partial u}{\partial \vec{n}} \mathrm{~d} \sigma_{g} \leq 0
\end{equation}
where $\vec{n}$ is the outward unit normal vector of $\partial U$.
\end{proposition}
\begin{proof}
The proof is divided into two steps.

\textbf{Step 1: the existence of $u$.}
Let us take a smooth exhaustion
$$
U=U_{0} \subset U_{1} \subset U_{2} \subset \cdots
$$
such that $U_{i}-\mathcal{E}$ has a compact closure for all $i$ and
$$
M=\bigcup_{i=0}^{\infty} U_{i}
$$
such exhaustion can be construucted from collecting points no greater than a certain distance to $U$. Fix some $U_{i}$ and consider the following equation
\begin{equation}\label{v_{i}in U_{i}}
\Delta_{g} v_{i}-f v_{i}=f \quad\text { in } \quad U_{i}; \quad \frac{\partial v_{i}}{\partial \vec{n}}=0 \quad\text { on } \quad\partial U_{i}
\end{equation}
where $\vec{n}$ is denoted to be the outward unit normal of $\partial U_{i}$. It is standard that equation \eqref{v_{i}in U_{i}} can be solved by exhaustion method. Let us extend the radical function $r$ by defining $r=1$ outside $\mathcal{E}$. For any $R>1$, we take
$$
U_{i, R}=\left\{x \in U_{i}: r(x) \leq R\right\}
$$
We have $\partial U_{i, R}=\partial U_{i} \cup \partial B_{R}$, where $\partial B_{R}=\{x \in M: r(x)=R\} \subset \mathcal{E}$. With this setting, we consider the following equation
\begin{equation}\label{v_{i,R}in U_{i,R}}
\left\{\begin{array}{clc}
\Delta_{g} v_{i, R}-f v_{i, R}=f & \text { in } & U_{i, R} \\
\frac{\partial v_{i, R}}{\partial n}=0 & \text { on } & \partial U_{i} \\
v_{i, R}=0 & \text { on } & \partial B_{R}
\end{array}\right.
\end{equation}

From Fredholm alternative \cite[Theorem 5.3]{GT2001}, equation \eqref{v_{i,R}in U_{i,R}} has a smooth solution $v_{i, R}$ once the corresponding homogeneous equation to \eqref{v_{i,R}in U_{i,R}} has only zero solution, so we verify this shortly. Let $\zeta$ be a solution of the homogeneous equation. From integral by parts we see
$$
\int_{U_{i, R}}\left|\nabla_{g} \zeta\right|^{2} \mathrm{~d} \mu_{g}=-\int_{U_{i, R}} f \zeta^{2} \mathrm{~d} \mu_{g}
$$
From inequality \eqref{sobelov on U}, the left hand side is greater than or equal to
$$
\frac{1}{C_{U}^{2}}\left(\int_{U_{R}}|\zeta|^{\frac{2(d+k) }{d+k-2}} r^{\frac{-2k}{d+k-2}} \mathrm{~d} \mu_{g}\right)^{\frac{d+k-2}{d+k }}
$$
however, the right hand side is less than or equal to
$$
\left(\int_{U_{R}}\left|f_{-}\right|^{\frac{d+k}{2}} r^{k}\,\mathrm{d} \mu_{g}\right)^{\frac{2}{d+k}}\left(\int_{U_{R}}|\zeta|^{\frac{2(d+k) }{d+k-2}} r^{\frac{-2k}{d+k-2}} \,\mathrm{d} \mu_{g}\right)^{\frac{d+k-2}{d+k}}
$$
where we used H\"{o}lder inequality and fact that $f$ has compact support in $U$. Combined with \eqref{f_{-}control} we obtain
$$
\int_{U_{R}}|\zeta|^{\frac{2(d+k) }{d+k-2}} r^{\frac{-2k}{d+k-2}} \,\mathrm{d} \mu_{g}=0
$$
hence $\zeta$ vanishes in $U_{R}$. Notice that $\zeta$ is a harmonic function in $U_{i, R}-U_{R}$ and a small neighbourhood of $\partial U$, then the unique continuation yields $\zeta \equiv 0$ on $U_{i,R}$, so \eqref{v_{i,R}in U_{i,R}} has a smooth solution $v_{i, R}$. Multiplying $v_{i, R}$ to both side of \eqref{v_{i,R}in U_{i,R}}, then integrating and using H\"{o}lder inequality, we have
$$
\begin{aligned}
\int_{U_{i, R}}\left|\nabla_{g} v_{i, R}\right|^{2} \mathrm{~d} \mu_{g} \leq
&\left(\int_{U_{R}}\left|f_{-}\right|^{\frac{d+k}{2}} r^{k}\,\mathrm{d} \mu_{g}\right)^{\frac{2}{d+k}}\left(\int_{U_{R}}|v_{i, R}|^{\frac{2(d+k) }{d+k-2}} r^{\frac{-2k}{d+k-2}}\, \mathrm{d} \mu_{g}\right)^{\frac{d+k-2}{d+k}}\\
+&\left(\int_{U_{R}}|f|^{\frac{2 (d+k)}{d+k+2}} r^{\frac{2k(d+k-2)}{(d+k+2)^{2}}}\,\mathrm{d} \mu_{g}\right)^{\frac{d+k+2}{2(d+k)}}\left(\int_{U_{R}}\left|v_{i, R}\right|^{\frac{2(d+k)}{d+k-2}} r^{\frac{-2k}{d+k-2}}\, \mathrm{d} \mu_{g}\right)^{\frac{d+k-2}{2 (d+k)}}.
\end{aligned}
$$
Combined with \eqref{sobelov on U}, \eqref{f_{-}control} and \eqref{f control}, it follows
\begin{equation}\label{estimate of v_{i,R}}
\left(\int_{U_{R}}\left|v_{i, R}\right|^{\frac{2(d+k)}{d+k-2}} r^{\frac{-2k}{d+k-2}} \,\mathrm{d} \mu_{g}\right)^{\frac{d+k-2}{2 (d+k)}} \leq 2 C_{U}^{2} C_{0}.
\end{equation}

Take an increasing sequence $R_{j}$ such that $R_{j} \rightarrow+\infty$ as $j \rightarrow \infty$. Since the weight function $r^{\frac{-2k}{d+k-2}}$ is bounded from below in any compact set, from the $L^{p}$-estimate \cite[Theorem 9.11]{GT2001} 
and the Schauder estimate \cite[Theorem 6.2]{GT2001},  we have
\begin{equation}\label{K in U}
\sup _{K}\left|\nabla_{g}^{k} v_{i, R_{j}}\right| \leq C(k, K)
\end{equation}
for any compact subset $K$ of $U$, where $C(k, K)$ is a constant depending on $k$ and $K$, but independent of $i$ and $R_{j}$. Fix a positive constant $r_{0}>1$.

Now we claim that 
\begin{equation}\label{C(i,r_{0})}
\sup _{U_{i,r_{0}}}\left|\nabla_{g}v_{i, R_{j}}\right| \leq C(i, r_{0})
\end{equation}
holds for some constant $C(i,r_{0})$ independent of $R_{j}$. Otherwise, there is a subsequence (still denoted by $v_{i, R_{j}}$) such that
$$
\sup _{U_{i, r_{0}}}\left|v_{i, R_{j}}\right| \rightarrow+\infty \quad \text { but } \quad\left(\int_{U_{r_{0}}}\left|v_{i, R_{j}}\right|^{\frac{2 (d+k)}{d+k-2}} \,\mathrm{d} \mu_{g}\right)^{\frac{d+k-2}{2 (d+k)}} \leq 2 C_{U}^{2} C_{0}C_{r_{0}},
$$
where $C_{r_{0}}$ is a constant depending on $r_{0}$, which comes from the weight $r^{\frac{-2k}{d+k-2}}$.

Let
$$
w_{i, R_{j}}=\left(\sup _{U_{i, r_{0}}}\left|v_{i, R_{j}}\right|\right)^{-1} v_{i, R_{j}}.
$$
Clearly $\left|w_{i, R_{j}}\right| \leq 1$ in $U_{i, r_{0}}$ and so it follows from the Schauder estimates \cite[ Theorem 6.2 and 6.30]{GT2001} that up to a subsequence $w_{i, R_{j}}$ converges uniformly to a limit function $w_{i}$ solving
$$
\Delta_{g} w_{i}-f w_{i}=0 \quad \text { in } \quad U_{i, r_{0}} ; \quad \frac{\partial w_{i}}{\partial \vec{n}}=0 \quad \text { on } \quad \partial U_{i},
$$
and satisfying
$$
\sup _{U_{i, r_{0}}}\left|w_{i}\right|=1 \quad \text { and } \quad w_{i}=0 \quad \text { in } \quad U_{r_{0}},
$$
where the latter comes from \eqref{K in U}. As before, $w_{i}$ is a harmonic function in $U_{i, r_{0}}-U_{r_{0}}$ and  a small neighbourhood of $\partial U$, then the unique continuation yields $w_{i} \equiv 0$ in $U_{i, r_{0}}$, which leads to a contradiction. 

Now the uniform bound for $v_{i, R_{j}}$ in $U_{i, r_{0}}$ combined with Schauder estimate as well as
estimate \eqref{C(i,r_{0})} implies that a subsequence of $v_{i, R_{j}}$ converges to a smooth solution $v_{i}$ of \eqref{v_{i}in U_{i}}. Clearly we have
\begin{equation}\label{v_{i}estimate}
\left(\int_{U}\left|v_{i}\right|^{\frac{2 (d+k)}{d+k-2}} r^{\frac{-2k}{d+k-2}}\,\mathrm{d} \mu_{g}\right)^{\frac{d+k-2}{2 (d+k)}} \leq 2 C_{U}^{2} C_{0}.
\end{equation}

Next we claim that for any small $\epsilon>0$ there are constants $r_0$ and $C'$ independent of $i$ such that 
\begin{equation}\label{rough decay}
|v_{i}(x)|\leq C'r^{-(d-2-\epsilon)}, \quad r(x)\geq r_{0}.
\end{equation}
To prove this, let us construct a comparison function $z(r)$ in the form of $Cr^{-(d-2-\epsilon)}$ with $C>0$ on the region $\{r\geq r_0\}$. From simple calculation we see
$$\Delta_{g} (r^{-(d-2-\epsilon)})=-\epsilon(d-2-\epsilon)r^{-(d-\epsilon)}+O(r^{-(d-\epsilon+\mu)}).$$
So we can take $r_0$ large enough such that $z(r)$ is superharmonic in $\{r\geq r_0\}$. Note that from \eqref{estimate of v_{i,R}} we have the uniform estimate
$$\sup _{r(x)=r_{0}} |v_{i,R_{j}}(x)|\leq \Lambda$$
for a universal constant $\Lambda$ independent of $i$ and $j$. Also, the function $v_{i, R_{j}}$ vanishes on $\partial B_{R_{j}}$.
After
choosing a large $C$, then we have $z(r)\geq v_{i,R_{j}}$ on boundary of $B_{R_{j}}-B_{r_{0}}$. Since $v_{i,R_{j}}$ is harmonic in $\{r\geq r_0\}$, we have $\Delta_{g} (z(r)-v_{i,R_{j}})\leq 0$. It follows from maximum principle that $v_{i,R_{j}}\leq z(r)$ in $\{r\geq r_0\}$. By letting $j\to +\infty$, we conclude that $v_i\leq z(r)$ in $\{r\geq r_0\}$. Similarly we also have $v_{i}\geq -z(r)$ and this completes the proof of \eqref{rough decay}.

Notice that $v_{i}$ is a harmonic function in $U_{i}-U$. Then the divergence theorem yields
$$
\int_{\partial U} \frac{\partial v_{i}}{\partial \vec{n}} \mathrm{~d} \sigma_{g}=\int_{\partial U_{i}} \frac{\partial v_{i}}{\partial \vec{n}} \mathrm{~d} \sigma_{g}=0
$$
and
$$
\int_{\partial U} v_{i} \frac{\partial v_{i}}{\partial \vec{n}} \mathrm{~d} \sigma_{g}=-\int_{U_{i}-U}\left|\nabla_{g} v_{i}\right|^{2} \mathrm{~d} \sigma_{g} \leq 0.
$$
Next we consider the function $u_{i}=1+v_{i}$ instead. Clearly, $u_{i}$ solves
$$
\Delta_{g} u_{i}-f u_{i}=0 \quad \text { in } \quad U_{i}; \quad \frac{\partial u_{i}}{\partial \vec{n}}=0 \quad \text { on } \quad \partial U_{i}.
$$
To estimete $u_{i}$, we first claim that $u_{i}$ is positive everywhere in $U_{i}$. Otherwise, the set $\Omega_{i,-}=$ $\left\{u_{i}<0\right\}$ is non-empty. Observe that $u_{i}$ is positive at the infinity of $\mathcal{E}$ due to \eqref{rough decay} and so $\Omega_{i,-}$ is compact. From integral by parts as well as the boundary condition of $u_{i}$, we have
$$
\int_{\Omega_{i,-}}\left|\nabla_{g} u_{i}\right|^{2} \,\mathrm{d} \mu_{g}=-\int_{\Omega_{i,-}} f u_{i}^{2}\, \mathrm{d} \mu_{g}.
$$
As before, we conclude from \eqref{sobelov on U} and \eqref{f_{-}control} that $u_{i}$ vanishes in $\Omega_{i,-}$, which is obviously impossible since the unique continuation leads to a contradiction. Now, we are able to apply the  Harnack inequality \cite[Theorem 8.20]{GT2001} to obtain local smooth convergence of $u_{i}$. In fact, functions $u_{i}$ have a uniform bound in any compact subset of $M$ from \eqref{v_{i}estimate}. Combined with the Schauder estimates, it implies that $u_{i}$ converges smoothly to a nonnegative limit function $u$ up to a subsequence. Clearly $u$ solves the equation $\Delta_{g} u-f u=0$ in $M .$ 

\textbf{Step 2: the expansion of $u$ at the infinity of $\mathcal{E}$.}
Denote $v=u-1$. Then $v$ is harmonic when $r\geq r_0$. From \eqref{rough decay} we also have
\begin{equation*}
|v(x)|\leq C'r^{-(d-2-\epsilon)}, \quad r(x)\geq r_0.
\end{equation*}
Note that $\mathcal E$ is covered by $(\mathbf R^d-B)\times\mathbf R^k$. Lift $g$ to a smooth metric $\tilde{g}$ on $(\mathbf R^d-B)\times\mathbf R^k$ and $v$ to a function $\tilde v$ on $(\mathbf R^d-B)\times\mathbf R^k$. After applying \cite[Theorem 6.2]{GT2001} to the function $\tilde v$ in balls $B_r^{n+k}(x_0)\subset (\mathbf R^d-B)\times\mathbf R^k$ with $x_0=(a_0,b_0)$ and $|a_0|=r$, we see
$$
|\partial \tilde v|(x_{0})+r|\partial^{2} \tilde v|(x_{0})+r^2|\partial^{3} \tilde v|(x_{0})\leq Cr^{-(d-1-\epsilon)}
$$
for some constant $C$. Hence we have
\begin{equation}\label{decay of v}
|v(x)|+r|\partial v(x)|+r^{2}|\partial^{2} v(x)|+r^3|\partial^3 v(x)|=O(r^{-(d-2-\epsilon)}).
\end{equation}
Next we work with the metric $g_{0}=g_{\text {euc }}+g_{T^{k}}$ on $\mathcal{E}$. Let us consider the decomposition $v=\Pi_{0}v+\Pi_{1}v$ with
$$
\Pi_{0}v(a_{0},\cdot)=\frac{1}{|\vol(T^{k},g_{T^k})|}\int_{a_{0}\times T^{k}}v(a_{0},\theta)\,\mathrm{d}\theta\quad \text{and}\quad\Pi_{1}v=v-\Pi_{0}v,
$$
where $a_{0}\in \mathbf R^{d}-B$ and $\theta\in T^k$.  Note that $\Pi_{0}v$ can be viewed as a function on $\mathbf R^{d}-B$. So we can compute
$$
\begin{aligned}
\Delta_{g_{\text {euc }}}\Pi_{0}v&=\Delta_{g_{0}}\frac{1}{|T^{k}|}\int_{a_{0}\times T^{k}}v(a_{0},\theta)d\theta\\&=\frac{1}{|T^{k}|}\int_{a\times T^{k}}\{(\delta^{ij}-g^{ij})\partial_{ij}v(a_{0},\theta)+g^{ij}\Gamma_{ij}^{k}\partial_{k}v(a_{0},\theta)\}d\theta\\
&=O(r^{-(d-\epsilon+\mu)})
\end{aligned}$$
and $\partial(\Delta_{g_{\text {euc}}}\Pi_{0}v)=O(r^{-(d-1-\epsilon+\mu)})$.
From \cite[Lemma 3.2]{SY79PMT} there exists a constant $A$ such that
$$\Pi_{0}v=\frac{A}{r^{d-2}}+\omega_{0},\quad\text{where}\quad|\omega_{0}|+r|\partial\omega_{0}|+r^{2}|\partial^{2}\omega_{0}|=O(r^{-(d-1)}).$$
From the definition of $\Pi_{1}v$, for each $a_0\in \mathbf R^d-B$ there always exists a $\theta_{0}\in T^k$ such that $\Pi_{1}v(a_{0},\theta_{0})=0$. For any $\theta\in T^k$ we take a geodesic curve $\gamma$ from $\theta_{0}$ to $\theta$ with respect to the given flat metric on $T^{k}$. From \eqref{decay of v} we have
$$|\Pi_{1}v(a_{0},\theta)|\leq\int_{\{a_{0}\}\times \gamma}\left|\frac{\partial v}{\partial s}\right|\,\mathrm{d}s=O(r^{-(d-1-\epsilon)}).$$
Notice that there always exists a vanishing point on $T^k$ for derivatives $\partial_{\theta_\alpha}v$ and $\partial^2_{\theta_\alpha\theta_\beta}v$. The same argument leads to
 $$
 r|\partial(\Pi_{1}v)(a_{0},\theta)|+r^2|\partial^2(\Pi_{1}v)(a_{0},\theta)|=O(r^{-(d-1-\epsilon)}).
 $$
Denote $\omega=\omega_{0}+\Pi_{1}v$. Then we have
$$
u=1+\frac{A}{r^{d-2}}+\omega\quad,\quad|\omega|+r|\partial\omega|+r^{2}|\partial^2\omega|=O(r^{-(d-1-\epsilon)}).
$$
The expression \eqref{A} of $A$ follows from
integrating \eqref{main equation} on $U_{R}$ and taking $R\rightarrow\infty$.
\end{proof}

In the following, we will call $(M, \bar{g}, \mathcal{E})$ an {\it asymptotically Schwarzschild-like manifold with fiber $F$}, if it is an asymptotically flat manifold with fiber $F$ and there exist constants $A$ and $q >n-2$ such that 
$$\bar g=\left(1+\frac{A}{r^{d-2}}\right)^{\frac{4}{d+k-2}}(g_{\text {euc }}+g_{F})+\omega,$$
where
$$|w|+r|\partial w|+r^{2}|\partial^{2}w|=O(r^{-q})$$
on the end $\mathcal{E}$. In this case, the total mass is expressed by
$$
m(M,\bar g,\mathcal E)=\frac{2(d+k-1)(d-2)}{d+k-2}A.
$$

\begin{proposition}\label{conformal end} Assume that $(M, g, \mathcal{E})$ is an asymptotically flat manifold with fiber $T^{k}$ with nonnegative scalar curvature and total mass $m$. Then for any $\tilde\epsilon>0$, we can construct a new complete metric $\bar{g}$ on $M$ such that $(M, \bar{g}, \mathcal{E})$ is an asymptotically Schwarzschild-like manifold with fiber $T^{k}$, which has nonnegative scalar curvature and mass $\bar{m}$ satisfying $|\bar{m}-m| \leq \tilde\epsilon$.
\end{proposition}
\begin{proof}First we can write the metric $g$ as
$$
g=\left(1+\frac{m_{1}}{r^{d-2}}\right)^{\frac{4}{d+k-2}} (g_{euc}+g_{T^{k}})+\tilde{g}$$
with
$$ m_1=\frac{d+k-2}{2(d+k-1)(d-2)}m.
$$
From Lemma \ref{defn of mass} we have
\begin{equation}\label{zeromasslimit}
\lim _{\rho \rightarrow+\infty}\int_{S_{\rho}\times T^{k}}(\partial_{i}\tilde{g}_{ij}-\partial_{j}\tilde{g}_{aa})\ast\mathrm{d}x^{j}\mathrm{d}\gamma=0.
\end{equation}
Take a fixed nonnegative cutoff function $\zeta: \mathbf{R} \rightarrow[0,1]$ such that $\zeta \equiv 0$ in $(-\infty, 2], \zeta \equiv 1$ in $[3,+\infty)$. With $s>1$ is a constant to be determined later, we define
$$
\hat{g}^{s}=\left(1+\frac{m_{1}}{r^{d-2}}\right)^{\frac{4}{d+k-2}} (g_{euc}+g_{T^{k}})+\left(1-\zeta\left(\frac{r}{s}\right)\right) \tilde{g}.
$$
Without loss of generality, we can assume that the asymptotic order in \eqref{Eq: decay 1} satisfies $\mu\leq d-2$. Clearly we have
$$
\left|\hat{g}^{s}-(g_{euc}+g_{T^{k}})\right|+r\left|\partial \hat{g}^{s}\right|+r^{2}\left|\partial ^{2}\hat{g}^{s}\right| \leq C r^{-\mu},
$$
where $C$ is always denoted to be a universal constant independent of $s$ here and in the sequel. In particular, the metrics $\hat{g}^{s}$ are uniformly equivalent to $g$ in some neighborhood $U$ of $\mathcal{E}$, and so the Sobolev inequality \eqref{sobelov on U} holds for all metrics $\hat{g}^{s}$ with a unifrom constant $C_{U}$ independent of $s$. It is easy to see that
$$R\left(\hat{g}^{s}\right) \geq 0 \quad \text{in}\quad \{r \leq 2 s\}; \quad R\left(\hat{g}^{s}\right) \equiv 0 \quad \text{in}\quad \{r \geq 3 s\} $$
and that $\left|R\left(\hat{g}^{s}\right)\right| \leq C s^{-(\mu+2)}$ in $ \{s \leq r \leq 4 s\}$. Then we have
\begin{equation}\label{R integral estimate}
\left(\int_{\{s \leq r \leq 4 s\}}\left|R\left(\hat{g}^{s}\right)\right|^{\frac{2 (d+k)}{d+k+2}}r^{\frac{2k(d+k-2)}{(d+k+2)^{2}}} \mathrm{~d} \mu_{\hat{g}^{s}}\right)^{\frac{d+k+2}{2 (d+k)}} \leq Cs^{-\frac{4k}{(d+k)(d+k+2)}} 
\end{equation}
and
\begin{equation}\label{R_{-}integral}
\left(\int_{U}\left|R\left(\hat{g}^{s}\right)_{-}\right|^{\frac{d+k}{2}} r^{k}\mathrm{~d} \mu_{g}\right)^{\frac{2}{d+k}}\leq C s^{-\frac{d-2}{2}}\leq \frac{1}{4C_{U}^{2}}
\end{equation}
for sufficiently large $s$.
Now we take another nonnegative cutoff function $\eta: \mathbf{R} \rightarrow[0,1]$ such that $\eta \equiv 0$ in $(-\infty, 1] \cup[4,+\infty)$ and $\eta \equiv 1$ in $[2,3] .$ Let $\eta_{s}(x)=\eta(r(x) / s) .$ It follows from \eqref{R_{-}integral}that for any $s$ we can take a small constant $\delta_{s}>0$ such that
$$
\left(\int_{U}\left|\left(\eta_{s} R\left(\hat{g}^{s}\right)- \eta_{s}\delta_{s}\right)_{-}\right|^{\frac{d+k}{2}}r^{k}\, \mathrm{d} \mu_{\hat{g}^{s}}\right)^{\frac{2}{d+k}} \leq  \frac{1}{2C_{U}^{2}}.
$$
From \eqref{R integral estimate} and taking $\delta_{s}$ small enough, we also have
\begin{equation}
\left(\int_{U}|\eta_{s} R\left(\hat{g}^{s}\right)- \eta_{s}\delta_{s}|^{\frac{2 (d+k)}{d+k+2}}  r^{\frac{2k(d+k-2)}{(d+k+2)^{2}}} \,\mathrm{d} \mu_{g}\right)^{\frac{d+k+2}{2(d+k)}} \leq Cs^{-\frac{4k}{(d+k)(d+k+2)}} 
\end{equation}
Based on Proposition \ref{3.2likeSY} we can construct a solution $u_{s}$ of the following equation
$$
\Delta_{\hat{g}^{s}} u_{s}-C_{d,k}\left(\eta_{s} R\left(\hat{g}^{s}\right)-\delta_{s} \eta_{s}\right) u_{s}=0,\quad C_{d,k}=\frac{d+k-2}{4(d+k-1)}.
$$
Moreover, $u_{s}$ has the expansion $u_{s}=1+A_{s}r^{-(d-2)}+O(r^{-(d-1-\epsilon)})$ for any small $\epsilon>0$, where
$$
A_{s}=-\frac{C_{d,k}}{(d-2)\left|\mathbf{S}^{d-1}\right||T^{k}|} \int_{U}\left(\eta_{s} R\left(\hat{g}^{s}\right)-\eta_{s}\delta_{s} \right) u_{s} \,\mathrm{d} \mu_{\hat{g}^{s}}
$$

We now show that $A_{s}$ is small if $s$ is large. To prove this, we consider the function $v_{s}=u_{s}-1$. From H\"{o}lder inequality, we have 
\begin{equation}\label{A_{s}term1}
\begin{aligned}
\left|\int_{U} \eta_{s} R\left(\hat{g}^{s}\right) u_{s} \,\mathrm{d} \mu_{\hat{g}^{s}}\right| 
&\leq\left(\int_{\{s \leq r\leq 4 s\}}\left|R\left(\hat{g}^{s}\right)\right|^{\frac{2(d+k)}{d+k+2}}r^{\frac{2k(d+k-2)}{(d+k+2)^{2}}} \,\mathrm{d} \mu_{\hat{g}^{s}}\right)^{\frac{d+k+2}{2 (d+k)}}\\
&\cdot\left(\int_{\{s \leq r \leq 4 s\}}\left|v_{s}\right|^{\frac{2 (d+k)}{d+k-2}} r^{\frac{-2k}{d+k-2}}\,\mathrm{d} \mu_{\hat{g}^{s}}\right)^{\frac{d+k-2}{2 (d+k)}} \\
&\qquad+\left|\int_{\{s \leq r \leq 4 s\}} \eta_{s} R\left(\hat{g}^{s}\right) \mathrm{d} \mu_{\hat{g}^{s}}\right|
\end{aligned}
\end{equation}
and
\begin{equation}\label{A_{s}term2}
\begin{split}
\left|\int_{U} \delta_{s} \eta_{s} u_{s} \,\mathrm{d} \mu_{\hat{g}_{s}}\right| &\leq C\delta_{s}s^{d}
+C\delta_{s}s^{\frac{2k(d+k-2)+d(d+k+2)^{2}}{2(d+k)(d+k+2)}}\\
&\cdot\left(\int_{\{s \leq r \leq 4 s\}}\left|v_{s}\right|^{\frac{2 (d+k)}{d+k-2}}r^{\frac{-2k}{d+k-2}}\, \mathrm{d} \mu_{\hat{g}^{s}}\right)^{\frac{d+k-2}{2 (d+k)}}.
\end{split}
\end{equation}
Recall from the proof of \eqref{estimate of v_{i,R}}, we have
$$
\begin{aligned}
&\left(\int_{\{s \leq r \leq 4 s\}}\left|v_{s}\right|^{\frac{2 (d+k)}{d+k-2}} r^{\frac{-2k}{d+k-2}}\,\mathrm{d} \mu_{\hat{g}^{s}}\right)^{\frac{d+k-2}{2 (d+k)}}\\ \leq& 2 C_{U}^{2}\left(\int_{U}\left|\eta_{s} R\left(\hat{g}^{s}\right)-\eta_{s}\delta_{s} \right|^{\frac{2 (d+k)}{d+k+2}}r^{\frac{2k(d+k-2)}{(d+k+2)^{2}}} \,\mathrm{d} \mu_{\hat{g}^{s}}\right)^{\frac{d+k+2}{2 (d+k)}}\\
\leq& 2 C_{U}^{2}Cs^{-\frac{4k}{(d+k)(d+k+2)}}  \rightarrow 0, \quad \text{as }\quad s\rightarrow \infty.
\end{aligned}
$$
Hence if we choose $\delta_{s}$ small enough, we see that the left hand side of \eqref{A_{s}term2} converges to zero as $s\rightarrow \infty$. Due to \eqref{R integral estimate}, we only need to estimate the left hand side of \eqref{A_{s}term1} to conclude $A_s\to 0$ as $s\to+\infty$.
Recall the divergence structure of scalar curvature up to higher order error (see \cite{Bartnik1986}). That is, with the coordinate system from Lemma \ref{defn of mass} we have
$$R(g)=|g|^{-\frac{1}{2}}\partial_{a}(g_{ab,b}-g_{bb,a})+O(r^{-(2+2\mu)}).$$
For the nonnegativity of $R\left(\hat{g}^{s}\right)$ outside $\{2 s \leq r \leq 3 s\}$, we conclude that there are two possibilities for the left hand side of \eqref{A_{s}term1}:

(i) if the integral of $(\eta_{s} R)_{+}$ larger than integral of $(\eta_{s} R)_{-}$, then
$$
\begin{aligned}
\left|\int_{\{s \leq r \leq 4 s\}} \eta_{s} R\left(\hat{g}^{s}\right) \mathrm{d} \mu_{\hat{g}^{s}}\right|&\leq
\left|\int_{\{s\leq r\leq 4s\}} R\left(\hat{g}^{s}\right) \mathrm{d} \mu_{\hat{g}^{s}}\right|\\
&\leq \left|\int_{\{r=4s\}}(\hat{g}^{s}_{ij,j}-\hat{g}^{s}_{aa,i})\ast\mathrm{d}x^{j}\mathrm{d}\gamma\right.\\
&\qquad-\left.\int_{\{r=s\}}(\hat{g}^{s}_{ij,j}-\hat{g}^{s}_{aa,i})\ast\mathrm{d}x^{j}\mathrm{d}\gamma\right|+ o(1).
\end{aligned}
$$

(ii) if the integral of $(\eta_{s} R)_{+}$ not larger than integral of $(\eta_{s} R)_{-}$, then
$$
\begin{aligned}
\left|\int_{\{s \leq r \leq 4 s\}} \eta_{s} R\left(\hat{g}^{s}\right) \mathrm{d} \mu_{\hat{g}^{s}}\right|&\leq
\left|\int_{\{2s\leq r\leq 3s\}} R\left(\hat{g}^{s}\right) \mathrm{d} \mu_{\hat{g}^{s}}\right|\\
&\leq \left|\int_{r=3s}(\hat{g}^{s}_{ij,j}-\hat{g}^{s}_{aa,i})\ast\mathrm{d}x^{j}\mathrm{d}\gamma\right.
\\
&-\left.\int_{r=2s}(\hat{g}^{s}_{ij,j}-\hat{g}^{s}_{aa,i})\ast\mathrm{d}x^{j}\mathrm{d}\gamma\right|+ o(1).
\end{aligned}
$$
Here the indices are used in the same way as in Lemma \ref{defn of mass}. In both cases, it follows from \eqref{zeromasslimit} and the definition of the total mass that
$$
\left|\int_{\{s \leq r \leq 4 s\}} \eta_{s} R\left(\hat{g}^{s}\right) \mathrm{d} \mu_{\hat{g}^{s}}\right|=o(1), \quad \text { as } \quad s \rightarrow \infty.
$$
Finally we conclude
$$
\left|\int_{U} \eta_{s} R\left(\hat{g}^{s}\right) u_{s} \mathrm{~d} \mu_{\hat{g}^{s}}\right|+\left|\int_{U} \delta_{s} \eta_{s} u_{s} \mathrm{~d} \mu_{\hat{g}_{s}}\right|=o(1), \quad \text { as } \quad s \rightarrow \infty.
$$
By taking $s$ large enough, we can guarantee $\left|A_{s}\right| \leq \frac{d+k-2}{2(d-2)(n-1+k)}\tilde\epsilon$ for any given $\tilde\epsilon>0$. Fix such an $s$ below.

Take
$$
u_{s, \tau}=\frac{u_{s}+\tau}{1+\tau}
$$
with $\tau$ a positive constant to be determined later. We consider the conformal deformation $\bar{g}=\left(u_{s, \tau}\right)^{\frac{4}{d+k-2}} \hat{g}^{s}$. A straightforward computation gives
$$
\begin{aligned}
R(\bar{g}) &=\left(u_{s, \tau}\right)^{-\frac{d+k+2}{d+k-2}}\left(-C_{d,k} \Delta_{\hat{g}^{s}} u_{s, \tau}+R\left(\hat{g}^{s}\right) u_{s, \tau}\right) \\
&=\left(u_{s, \tau}\right)^{-\frac{d+k+2}{d+k-2}}(1+\tau)^{-1}\left(\left(\left(1-\eta_{s}\right) R\left(\hat{g}^{s}\right)+\delta_{s} \eta_{s}\right) u_{s}+R\left(\hat{g}^{s}\right) \tau\right)
\end{aligned}
$$
Note that $R\left(\hat{g}^{s}\right)$ takes possible negative values only in $\{2 s \leq r \leq 3 s\}$, where the term $\left(\left(1-\eta_{s}\right) R\left(\hat{g}^{s}\right)+\delta_{s} \eta_{s}\right) u_{s}$ has a positive lower bound, so we can take $\tau$ small enough such that $R(\bar{g})$ is nonnegative everywhere. Since the function $u_{s, \tau}$ is no less than $\tau(1+\tau)^{-1}$, the metric $\bar{g}$ is still complete. In the region $\{r \geq 3 s\}$, the metric $\bar{g}$ can be expressed as
$$
\bar{g}=\left(u_{s, \tau}\right)^{\frac{4}{d+k-2}}\left(1+\frac{m_{1}}{r^{d-2}}\right)^{\frac{4}{d+k-2}} (g_{euc}+g_{T^{k}})=\left(1+\frac{m_{2}}{r^{d-2}}\right)^{\frac{4}{d+k-2}}(g_{euc}+g_{T^{k}})+w,
$$
where $|w|+r|\partial w|+r^{2}|\partial^{2}w|=O(r^{-(d-1-\epsilon)})$ for any small $\epsilon>0$ and $m_{2}=m_{1}+{A_{s}}({1+\tau})^{-1}$. So we have 
$$|m(M,g,\mathcal E)-m(M,\bar{g},\mathcal E)|=\frac{2(d-2)(d-1+k)}{(d+k-2)(1+\tau)}|A_{s}|\leq \tilde\epsilon.$$
This competes the proof.
\end{proof}

\begin{proof}[Proof for Theorem \ref{Thm: main 4}]
For the inequality part, we argue by contradiction. Suppose that $(M, g, \mathcal{E})$ has negative total mass $m$. Set $\tilde\epsilon=-m / 2$. From Proposition \ref{conformal end}  above we can construct a new complete metric $\bar{g}$ on $M$ such that $(M, \bar{g}, \mathcal{E})$ is an asymptotically Schwarzschild-like manifold with fiber $T^k$, which has nonnegative scalar curvature and total mass $\bar{m}$ no greater than $-m / 2 .$ Moreover, the metric $\bar{g}$ can be expressed as $\bar{g}=u^{\frac{4}{d+k-2}}(g_{euc}+g_{T^{k}})$ near the infinity, which is also scalar flat. Hence the function $u$ is a harmonic function near infinity with respect to the metric $g_{euc}+g_{T^{k}}$. Notice that $u$ has the expansion
$$
u=1+\frac{m_{2}}{r^{d-2}}+O(r^{-(d-1-\epsilon)}), \quad m_{2}<0,
$$
for any small $\epsilon>0$ be small, so we can take $s_{1}$ large enough such that $u<1$ on $\left\{r=s_{1}\right\}$. Denote
$$
\bar\epsilon
=1-\sup _{r(x)=s_{1}} u(x)
$$
It is clear that $u>1-\bar\epsilon / 4$ in $\left\{r \geq s_{2}\right\}$ for sufficiently large $s_{2}>s_{1}$. Take a cutoff function $\zeta:[0,+\infty) \rightarrow[0,1-\bar\epsilon / 2]$ such that $\zeta(t)=t$ when $t \leq 1-3 \bar\epsilon / 4$ and $\zeta(t)=1-\bar\epsilon / 2$ when $t \geq 1-\bar\epsilon / 4$. Moreover, we can also require $\zeta^{\prime} \geq 0$ and $\zeta^{\prime \prime} \leq 0$ in $[0,+\infty)$ as well as $\zeta^{\prime \prime}<0$ in $(1-3 \bar\epsilon/ 4,1-\bar\epsilon / 4)$. Let
$$
v=\left\{\begin{array}{cc}
\zeta \circ u, & r \geq s_{1} \\
u & r \leq s_{1}
\end{array}\right.
$$
Clearly $v$ is a smooth function defined on entire $M$ since $v$ equals to $u$ around $\left\{r=s_{1}\right\}$. A direct computation shows
$$
\Delta v=\zeta^{\prime \prime}|\nabla u|^{2}+\zeta^{\prime} \Delta u \leq 0 \quad \text { in } \quad\left\{r \geq s_{1}\right\}
$$
and further $\Delta v<0$ at some point in $\left\{s_{1}<r<s_{2}\right\}$, where $\Delta$ and $\nabla$ are Laplace and gradient operators with respect to the metric $g_{euc}+g_{T^{k}}$. Define
$$
\tilde{g}=\left(\frac{v}{u}\right)^{\frac{4}{d+k-2}} \bar{g}
$$
It is easy to verify that $\tilde{g}$ is a complete metric on $M$ with nonnegative scalar curvature (positive somewhere), which is exactly the metric $g_{euc}+g_{T^{k}}$ near the infinity of $\mathcal{E}$. In particular, we can close $(M, \tilde{g})$ around the infinity of $\mathcal{E}$ to construct a smooth complete metric on $T^{n} \#_{T^k} \bar M$ with nonnegative scalar curvature (positive somewhere) for some compactification $\bar M$ of $M$ by adding an inifnity $T^k$ to $M$. From our topological assumption such manifold cannot admit any complete metric with positive scalar curvature (see a further explanation in the proof of Theorem \ref{Thm: main 5} in next subsection) and we obtain a contradiction.

Now let us focus on the rigidity part under the assumption that the total mass of $(M, g, \mathcal{E})$ is zero. First we show that the scalar curvature of $g$ has to vanish everywhere . Otherwise, we can take a neighborhood $U$ of $\mathcal{E}$ such that $R(g)$ is positive at some point $p$ in $U$. Take a nonnegative cutoff function $\eta: M \rightarrow[0,1]$ with small compact support in $U$ such that $\eta \equiv 1$ around
point $p$. It follows from Proposition \ref{3.2likeSY} that there is a positive function $u$ solving
$$
\Delta_{g} u-C_{d,k} \eta R(g) u=0,
$$
which has the expansion $u=1+A r^{-(d-2)}+O\left(r^{-(d-1-\epsilon)}\right)$ with $A<0$.
Define
$$
\bar{g}=\left(\frac{u+1}{2}\right)^{\frac{4}{d+k-2}} g
$$
It is not difficult to verify that $(M, \bar{g}, \mathcal{E})$ is an asymptotically flat manifold with fiber $T^k$, which has nonnegative scalar curvature but negative total mass. 

Next we prove the Ricci flatness of the metric $g$. Otherwise we take a neighborhood $U$ of $\mathcal{E}$ such that $\operatorname{Ric}(g)$ does not vanish at some point $p$ in $U$. Take a nonnegative cutoff function $\eta: M \rightarrow[0,1]$ with compact support $S$ in $U$ such that $\eta \equiv 1$ around point $p$. After applying a variation argument to the first Neumann eigenvalue of conformal Laplace operator \cite[Lemma 3.3]{Kazdan82}, we can pick up a positive constant $\varepsilon$ small enough such that the metric $\bar{g}=g-\varepsilon \eta \operatorname{Ric}(g)$ satisfies

\begin{equation}\label{kazdan}\int_{S}\left|\nabla_{\bar{g}} \zeta\right|^{2}+C_{d,k} R(\bar{g}) \zeta^{2} \mathrm{~d} \mu_{\bar{g}}>c \int_{S} \zeta^{2} \mathrm{~d} \mu_{\bar{g}}, \quad \forall \zeta \neq 0 \in C^{\infty}(S),
\end{equation}
for some positive constant $c$. For $\varepsilon$ small, we still have a uniform  Sobolev constant of $U$ with respect to the metric $\bar{g}$, still denoted by $C_{U}$. Since $\bar{g}$ is almost equal to $g$, we can also require
$$
\left(\int_{U}\left|(R(\bar{g}))_{-}\right|^{\frac{d+k}{2}} r^{k}\,\mathrm{d} \mu_{\bar{g}}\right)^{\frac{2}{d+k}} \leq \frac{1}{4C_{U}^{2}}
$$
by further decreasing the value of $\varepsilon$. Let $\tilde{\eta}$ be another nonnegative cutoff function with compact support $\tilde{S}$ in $U$, which has a positive lower bound in the support of $\eta$. For $\delta$ be a constant small enough, we can construct a smooth positive function $u$ from Proposition \ref{3.2likeSY} solving
$$
\Delta_{\bar{g}} u-C_{d,k}(R(\bar{g})-\delta \tilde{\eta}) u=0
$$
The function $u$ has the expansion $u=1+A r^{-(d-2)}+O\left(r^{-(d-1-\epsilon)}\right)$ for any $\epsilon>0$ be small, and we expect $A$ to be negative when $\delta$ is sufficiently small. From integral by parts we have
$$
\begin{aligned}
(d-2)\left|\mathbf{S}^{d-1}\right| \vol(T^k,g_{T^k})A &=-\left(\int_{U}\left|\nabla_{\bar{g}} u\right|^{2}+C_{d,k}(R(\bar{g})-\delta \tilde{\eta}) u^{2} \,\mathrm{d} \mu_{\bar{g}}\right) \\
& \leq-\left(\int_{U}\left|\nabla_{\bar{g}} u\right|^{2}+C_{d,k} R(\bar{g}) u^{2}\, \mathrm{d} \mu_{\bar{g}}\right)+C_{d,k}\delta \int_{\tilde{S}} u^{2} \,\mathrm{d} \mu_{\bar{g}}.
\end{aligned}
$$
From \eqref{kazdan} and the fact $R(\bar{g}) \equiv 0$ outside $\tilde{S}$, it follows from \cite[Lemma 2.1]{Kazdan82} that
$$
\int_{U}\left|\nabla_{\bar{g}} \zeta\right|^{2}+C_{d,k} R(\bar{g}) \zeta^{2} \mathrm{~d} \mu_{\bar{g}}>\tilde{c} \int_{\tilde{S}} \zeta^{2} \mathrm{~d} \mu_{\bar{g}}, \quad \forall \zeta \neq 0 \in C^{\infty}(U),
$$
for some positive constant $\tilde{c}$ independent of $\delta$. So we can guarantee $A<0$ by taking $\delta$ small enough. As before, we consider the conformal metric
$$
\tilde{g}=\left(\frac{u+\tau}{1+\tau}\right)^{\frac{4}{d+k-2}} \bar{g}
$$
with $\tau$ a positive constant to be determined later. A straightfrward computation shows
$$
R(\tilde{g})=\left(\frac{u+\tau}{1+\tau}\right)^{-\frac{d+k+2}{d+k-2}}(1+\tau)^{-1}(\delta \tilde{\eta} u+\tau R(\bar{g})).
$$
In particular, we have $R(\tilde{g}) \geq 0$ if $\tau$ is chosen to be small enough. However, it is easy to verify that $(M, \tilde{g}, \mathcal{E})$ is an asymptotically flat manifold with fiber $T^k$ whose total mass is negative, which contradicts to the discussion above.

Now we can deduce that $M$ has only one end $\mathcal{E}$. Otherwise, the Cheeger-Gromoll splitting theorem (see \cite{CG1971}) yields that $(M, g)$ must be isometric to a Riemannian product manifold $N \times \mathbf{R}$ for a closed manifold $N$. In particular, each $N$-slice is totally geodesic. However, this is impossible since if such slice appears at the infinity of the asymptotically flat end $\mathcal{E}$, we can take a $S_r\times T^{k}$  touching this slice from outside, which contradicts to the maximum principle. Consider the universal covering $(\hat M,\hat g)$. 

It remains to show the flatness of $(M,g)$ and the discussion will be divided into two cases. If $k=\dim F\geq 2$, then we consider the universal covering $(\hat M,\hat g)$ of $(M,g)$. From the incompressible condition $(\hat M,\hat g)$ is a complete Ricci-flat manifold with one of its end diffeomorphic to $(\mathbf R^d-B)\times \mathbf R^k$, where the metric is asymptotic to the Euclidean metric at infinity. The standard volume comparison theorem tells us that $(M,g)$ must be flat.
In the case when $F=\mathbf S^1$, the lifting of $(M,g)$ may not lift the end $\mathcal E$ entirely to $(\mathbf R^d-B)\times \mathbf R$ and so the argument above cannot be applied. Instead we are going to use the Bochner method with a nice observation formulated as the following Proposition \ref{Ricci decay}. From Proposition \ref{Ricci decay} we see that there exist $d$ $g$-parallel one-forms $\alpha^{1},\cdots,\alpha^{d}$, which further yields that $(M,g)$ splits as the standard $\mathbf R^{d}\times \mathbf S^{1}$ (see for instance \cite[Page 945]{Minerbe2008}). 
\end{proof}
\begin{proposition}\label{Ricci decay}
If $(M,g,\mathcal E)$ is an asymptotically flat manifold with fiber $F=S^{1}$ and the unique end $\mathcal E$ such that 
\begin{equation}\label{Eq: metric decay}
\sum\limits^3_{k=0}r^k|\nabla_{g_0}^k (g-g_{0})|_{g_0}=O(r^{-\mu}), \quad\mu>\frac{d-2}{2},
\end{equation}
and
\begin{equation}\label{Eq: Ricci decay}
|\nabla_{g_0} Ric|_{g_0}+r|\nabla_{g_0} Ric|_{g_0}=O(r^{-n-\epsilon})\quad \text{for some}\quad\epsilon>0 .
\end{equation}
Let $\{x^{1},\cdots,x^{d},\theta\}$ be the given coordinate system of $(M,g,\mathcal E)$ at infinity, where $\theta$ is the parameter of $S^{1}$. Then there exists a coordinate system $\{y^{1},\cdots,y^{d},\theta\}$ outside a compact set such that
$\Delta_{g}y^{i}=0,\,i=1,\cdots,d$. 
Moreover, if we denote $\alpha^{i}=dy^{i}$, then we have
\begin{equation}\label{Eq: gb mass}
\sum\limits^{d}_{i=1}\int_{M}|\nabla_g\alpha^{i}|_g^{2}+\Ric(\alpha^{i},\alpha^{i})\mathrm{d}V_{g}=c(d,S^{1})m(M,g,\mathcal{E}),
\end{equation}
where $c(d,S^{1})=d|\mathbf S^{d-1}|\vol(S^{1},g_{S^{1}})$.
\end{proposition}
\begin{proof}
The proof is based on the theory of weighted spaces established in \cite{Minerbe2008} and here we just sketch the analysis part but emphasize on the difference. Without loss of generality, we can assume $\mu\leq d-2$. As in \cite{Minerbe2008}, to find desired functions $y^i$, we try to solve the equation $\Delta_g u^i=\Delta_g(\chi x^i)$ for some suitable cut-off function $\chi$ and then take $y^i=\chi x^i-u^i$.
From the metric decay \eqref{Eq: metric decay}, we compute $\Delta_gx^i=O(r^{-\mu-1})$ and this yields $\Delta_{g}(\chi x^{i})\in L_{\delta'-2}^{2}(M)$ for any $\delta'>1+\frac{d}{2}-\mu$. 
It then follows from \cite[Corollary 2]{Minerbe2008} that there exists a function $u^{i}\in H_{\delta'}^{2}(M)$ such that $\Delta_{g}u^{i}=\Delta_{g}(\chi x^{i})$. The Moser iteration (see \cite[Page 942, (23)]{Minerbe2008}) yieds
$$|u^{i}|= O(r^{1-\mu'}), \quad\mu'= \min\left\{\frac{d}{2}-\delta'+1,\mu\right\}.$$
Fix $\epsilon_{1}>0$ small enough such that if we take $\delta'=1+\frac{d}{2}-\mu+\epsilon_{1}$, then $\mu'=\mu-\epsilon_{1}>\frac{d-2}{2}$. Similar as in step 2 of the proof of Proposition \ref{3.2likeSY}, the Schauder estimate yields
$$r|\partial u^i|+r^{2}|\partial^{2} u^i|= O(r^{1-\mu'})$$
With a similar analysis on the derivative of $\Delta_{g}u^{i}=\Delta_{g}x^{i}$, we further obtain
$$r^{3}|\partial^{3}u^{i}|= O(r^{1-\mu'}),$$
here and above the partial derivative is taken with respect to the coordinate system $\{x^{1},\cdots,x^{d},\theta\}$.

Now let us consider $y^{i}=\chi x^{i}-u^{i}$. These estimates imply that $\{y^{1},\cdots,y^{d},\theta\}$ form a coordinate system outside a large compact subset. A direct calculation shows that with respec to the coordinate system $\{y^{1},\cdots,y^{d},\theta\}$ the metric $g$ has the form
$
g=g_{0}'+\omega,
$ 
where $g_{0}'=dy^{2}\oplus g_{S^{1}}$ and $\omega$ satisfies
$$\sum\limits^2_{k=0}r^k|\nabla_{g'_0}^k \omega|_{g_0'}=O(r^{-\mu}), \quad\mu>\frac{d-2}{2}$$
In the following, we work in this new coordinate system and the partial derivative will be taken with respect to $\{y^{1},\cdots,y^{d},\theta\}$. In particular, $\partial_i$ means $\partial_{y_i}$. Let us also denote $g_{ij}=g(\frac{\partial}{\partial y^{i}},\frac{\partial}{\partial y^{j}})$ and $r=|y|$.

It remains to show \eqref{Eq: gb mass}. From the Bochner formula as well as $\Delta_{g}y^{i}=0$, we have 
\begin{equation}\label{Bochner0}
\Delta_g\left(\frac{1}{2}|dy^{i}|_g^{2}\right)=|\nabla_g dy^{i}|_g^{2}+\Ric(dy^{i},dy^{i})
\end{equation}
and
$$
\frac{1}{2}\Delta_{g}g^{ij}=g(\nabla_g dy^{i},\nabla_g dy^{j})+\Ric(dy^{i},dy^{j}).
$$
From the Ricci decay \eqref{Eq: Ricci decay} we see $\Delta_{g}(g^{ij}-\delta_{ij})=O(r^{-d-\epsilon_{2}})$, where $\epsilon_{2}=\min\{\epsilon,2\mu'+2-d\}>0$. From above discussion we conclude that
$$g^{ij}-\delta_{ij}\in L^{2}_{\delta}(M) \quad\text{for any}\quad \delta>\frac{d}{2}-\mu',$$
and
$$\Delta_{g}(g^{ij}-\delta_{ij})\in L^{2}_{\delta''-2}(M) \quad\text{for any}\quad\delta''>2-\frac{d}{2}-\epsilon_{2}.$$
It follows from \cite[Proposition 4]{Minerbe2008} and \cite[Proposition 4]{Minerbe2008} that
if we can choose $\delta$ and $\delta''$ such that 
\begin{equation}\label{Eq: order picking}
1-\frac{d}{2}\leq\delta''<2-\frac{d}{2}<\delta\leq \frac{d}{2},
\end{equation}
then we have the expansion
$$g^{ij}=\delta_{ij}-c_{ij}r^{2-d}+v^{ij},$$
where $c_{ij}$ are constants with $c_{ij}=c_{ji}$ and the functions $v^{ij}$ belongs to $ H^{2}_{\delta''}(M)$. Recall that $\mu'\leq \mu\leq d-2$ and $\delta''$ can be chosen to be arbitrarily close to $2-\frac{d}{2}$, the inequality \eqref{Eq: order picking} is fairly easy to hold. After repeating the argument in estimates for $u^i$, we see that $v^{ij}$ are higher-order error terms satisfying
$$\sum_{k=0}^2r^k|\partial^kv^{ij}|=O(r^{2-d-\epsilon_{3}}),\quad\text{where}\quad \epsilon_{3}=2-\frac{d}{2}-\delta''>0.$$
After a possible orthogonal transformation of $\{y^{1},\cdots,y^{d}\}$, we can assume $c_{ij}=c_{i}\delta_{ij}$ without loss of generality. Finally we have
$$
g_{ij}=\delta_{ij}+c_{i}\delta_{ij}r^{2-d}+w_{ij},\quad g_{i\theta}=w_{i\theta},\quad g_{\theta\theta}=1+w_{\theta\theta},
$$
such that
$$\sum\limits^2_{k=0}r^{k}|\partial^{k}w_{ij}|=O(r^{2-d-\epsilon_{3}})$$
and
$$ \sum\limits^2_{k=0}r^{k}|\partial^{k}w_{i\theta}|+r^{k}|\partial^{k}w_{\theta\theta}|=O(r^{-\mu'}).$$
It follows from \cite[Theorem 4.2]{Bartnik1986} that the total mass $m(M,g,\mathcal E)$ calculated with respect to $\{y^{1},\cdots,y^{d},\theta\}$ is the same as that with respect to $\{x^{1},\cdots,x^{d},\theta\}$. Recall that we have $\Delta_{g}y^{i}=0$. A straightforward computation yields
\begin{equation}\label{theta expression}
(d-2)\left(c_{i}-\frac{1}{2}\mathop{\Sigma}\limits_{j=1}^{d}c_{j}\right)\frac{y^{i}}{|y|^{d}}+\partial_{\theta}g_{i\theta}+\frac{1}{2}\partial_{i}g_{\theta\theta}+O(r^{-\mu''})=0,
\end{equation}
where $\mu''=\min\{n-1+\epsilon_{3},2\mu'+1\}>d-1$.
Integrating \eqref{theta expression} and suming over $i$, we obtain
$$\lim_{\rho\to+\infty}\int_{S_\rho\times S^{1}}(-\partial_{i}g_{\theta\theta})\ast\mathrm{d}y^{i}\mathrm{d}\theta=\lim_{\rho\to+\infty}\int_{S_\rho\times S^{1}}-\frac{(d-2)^{2}}{d}\left(\sum\limits_{i=1}^{d}c_{i}\right)\rho^{1-d}\mathrm{d}S_{\rho}\mathrm{d}\theta.$$
Therefore we have
\begin{displaymath} 
\begin{aligned}	
m(M,g,\mathcal E)&=\frac{1}{2|\mathbf S^{d-1}|\vol(S^{1},g_{S^{1}})}\lim_{\rho\to+\infty}\int_{S_\rho\times S^{1}}(\partial_{j}g_{ij}-\partial_{i}g_{aa})\ast\mathrm{d}y^{i}\mathrm{d}\theta\\
&=\frac{1}{d-1}\cdot\frac{1}{2|\mathbf S^{d-1}|\vol(S^{1},g_{S^{1}})}\lim_{\rho\to+\infty}\int_{S_\rho\times S^{1}}(\partial_{j}g_{ij}-\partial_{i}g_{jj})\ast\mathrm{d}y^{i}\mathrm{d}\theta.
\end{aligned}
\end{displaymath}
Now let us integrate both sides of \eqref{Bochner0} and take the sum over $i$. This gives
\begin{displaymath} 
\begin{aligned}	
\sum\limits^{d}_{i=1}\int_{M}|\nabla_g\alpha^{i}|_g^{2}+\Ric(\alpha^{i},\alpha^{i})\mathrm{d}V_{g}&=\sum\limits^{d}_{i=1}\lim_{\rho\to+\infty}\int_{B_\rho\times S^{1}}\frac{1}{2}\Delta_g(|dy^{i}|_g^{2})\mathrm{d}V_{g}\\
&=\sum\limits^{d}_{i=1}\lim_{\rho\to+\infty}\int_{S_\rho\times S^{1}}\frac{1}{2}\partial_{j}g^{ii}\ast \mathrm{d}y^{i}\mathrm{d}\theta\\
&=\frac{d-2}{2}\left(\sum\limits^{d}_{i=1}c_{i}\right)|\mathbf S^{d-1}|\vol(S^{1},g_{S^{1}})\\
&=d\cdot|\mathbf S^{d-1}|\vol(S^{1},g_{S^{1}}) m(M,g,\mathcal E),
\end{aligned}
\end{displaymath}
and we complete the proof.
\end{proof}

\subsection{Compactification with quasi-spherical metric and Proof for Theorem \ref{Thm: main 5}} For asymptotically conical manifolds with fiber $F$, Lohkamp compactification based on the conformal deformation is no longer valid since the expansion of a harmonic function on $\mathbf R^2$ is rather difficult due to its blow-up at infinity. As an alternative, we will use the quasi-spherical metric to complete the compactification. Quasi-spherical metric was first introduced by Bartnik in \cite{Bartnik1993} and later used by the third named author and Tam \cite{ST2002} to prove the nonnegativity of the Brown-York mass. Its use in the proof of positive mass theorems (angle estimates) is a new observation in this paper.

Let $(M,g)$ be an asymptotically conical manifold with fiber $T^{n-2}$. Notice that the end $\mathcal E$ (diffeomorphic to $(\mathbf R^2-B)\times T^{n-2}$) can be compactified to be $(\mathbf S^2-B)\times T^{n-2}$ by adding a $T^{n-2}$ at infinity. So we can denote $\bar M$ to be the corresponding compactification of $M$ and $i_2:T^{n-2}\to \bar M$ to be the embedding mapping onto the infinity $T^{n-2}$ in $\bar M$. Let us start with the following
\begin{proposition}\label{Prop: PMT to PSC 2}
If the asymptotically conical manifold $(M,g,\mathcal E)$ with fiber $T^{n-2}$ has nonnegative scalar curvature and its angle at inifnity is (strictly) greater than $2\pi$, then the generalized connected sum $(T^n,i)\#_{T^{n-2}}(\bar M,i_2)$ admits a complete metric with positive scalar curvature.
\end{proposition}
\begin{proof}
Let $r$ be the radical distance function on $\mathbf R^2$ and we denote $S_r$ to be the hypersurface $C_r\times T^{n-2}$ in $M$ for $r >1$, where $C_r$ is the circle in $\mathbf R^2$ with radius $r$ centered at the origin. Denote $g_{flat}$ to be the flat metric on $T^{n-2}$. We can write
$$
g_\beta|_{S_r}=\beta^2r^2\mathrm d\theta^2+g_{flat}+\eta_r,
$$
where the decay condition \eqref{Eq: decay 2} implies
\begin{equation}\label{Eq: eta r}
|\eta_r|_{g_\beta}+r|\nabla_{g_\beta}\eta_r|_{g_\beta}+r^2|\nabla_{g_\beta}^2\eta_r|_{g_\beta}=O(r^{-\mu}).
\end{equation}
Fix a large positive constant $r_0$ to be determined later and pick up a cut-off function $\zeta:[r_0,2r_0]\to [0,1]$ such that it takes value $1$ around $r_0$ and value $0$ around $2r_0$ as well that its derivative satisfies $|\zeta'|\leq 4r_0^{-1}$. For $r_0\leq r\leq 2r_0$, we define
$$
\bar g_r=\beta^2r^2\mathrm d\theta^2+g_{flat}+\zeta(r)\eta_r.
$$
With $u$ a positive smooth function to be determined later, we introduce the following metric
\begin{equation*}
\bar g=\left\{
\begin{array}{cc}
g_\beta,&r\leq r_0;\\
u^2\mathrm dr^2+\bar g_r,&r_0\leq r\leq 2r_0;\\
\beta^2g_{euc}\oplus g_{flat}, &r\geq 2r_0.
\end{array}
\right.
\end{equation*}
Actually we hope that this metric $\bar g$ is a piecewisely smooth metric on $M$ with nonnegative scalar curvature in the distribution sense of Miao \cite{Miao2002}, which has exactly two corners along the hypersurfaces $S_{r_0}$ and $S_{2r_0}$. This motivates us to solve the following quasi-spherical metric equation
\begin{equation}\label{Eq: QS equation}
H_r\frac{\partial u}{\partial r}=u^2\Delta_{\bar g_r}u+\frac{1}{2}R(\bar g_r)(u-u^3)-\frac{1}{2}R(\bar g_{ref})u,\quad r_0\leq r\leq 2r_0
\end{equation}
with boundary value
\begin{equation}\label{Eq: initial condition}
u(r_0,\cdot)=\frac{H_{r_0}}{H_{r_0,g_\beta}}\quad\text{on}\quad S_{r_0},
\end{equation}
where $\bar g_{ref}$ is the reference metric $\mathrm dr^2+\bar g_r$, $H_r$ is the mean curvature of $S_r$ with respect to $\bar g_{ref}$ and $H_{r_0,g_\beta}$ is the mean curvature of $S_{r_0}$ with respect to the metric $g_\beta$. Geometrically, the equation \eqref{Eq: QS equation} with condition \eqref{Eq: initial condition} matches the mean curvature on two sides of the corner along $S_{r_0}$ and guarantees that $\bar g$ has vanishing scalar curvature when $r_0\leq r\leq 2r_0$. In order to control the mean curvature at the second corner, we need to make a careful analysis on the equation above. From \eqref{Eq: eta r} we can compute
$$
H_r=r^{-1}+O(r^{-1-\mu}),\quad R(\bar g_r)=O(r^{-2-\mu})\quad \text{and}\quad R(\bar g_{ref})=O(r^{-2-\mu}).
$$
Notice that both $H_{r_0}$ and $H_{r_0,g_\beta}$ have the order $r^{-1}_0+O(r_0^{-1-\mu})$. So the initial value satisfies $u(r_0,\cdot)=1+O(r_0^{-\mu})$. Denote $v(r)$ to be the solution of the following ordinary differential equation
$$
\frac{\mathrm dv}{\mathrm dr}=\frac{1}{2}f(r)(v-v^3)-\frac{1}{2}g(r)v,\quad v(r_0)=\max_{S_{r_0}}u(r_0,\cdot).
$$
where
$$
f(r)=\max_{S_r}\frac{R(\bar g_r)}{H_r}\quad\text{and}\quad g(r)=\min_{S_r}\frac{R(\bar g_{ref})}{H_r}.
$$
It is not difficult to solve this equation and obtain
$$
v(r)=\left[\left(v^{-2}(r_0)+\int_{r_0}^rf(\tau)e^{-\int_{r_0}^\tau(g(s)-f(s))\mathrm ds}\mathrm d\tau\right)e^{\int_{r_0}^r(g(s)-f(s))\mathrm ds}\right]^{-2}.
$$
Based on the estimates above we conclude $v=1+O(r_0^{-\mu})$, and then it follows from the maximum principle that
$
u\leq 1+O(r^{-\mu}_{0}).
$ Now we focus on the mean curvatures on two sides of the second corner along $S_{2r_0}$. Clearly the mean curvature of $S_{2r_0}$ with respect to the product metric $\beta^2g_{euc}\oplus g_{flat}$ is
$$
H_{2r_0,prod}=\frac{1}{2r_0\beta},
$$
and the mean curvature of $S_{2r_0}$ with respect to the reference metric $\bar g_{ref}$ satisfies
$$
H_{2r_0}=\frac{1}{2r_0 u}\geq \frac{1}{2r_0}+O(r_0^{-1-\mu}).
$$
Since the angle of $(M,g,\mathcal E)$ at infinity is greater than $2\pi$, i.e. $\beta >1$, we can take $r_0$ to be large enough such that $H_{2r_0}>H_{2r_0,prod}$. So $(M,\bar g,\mathcal E)$ has nonnegative scalar curvature in the distribution sense.

Notice that $(M,\bar g,\mathcal E)$ is isometric to the Riemannian product of the standard Euclidean $2$-space and a flat $T^{n-2}$. We can glue the opposite faces of $C\times T^{n-2}$ for a large cube $C\subset \mathbf R^2$ to obtain a piecewisely smooth metric $\tilde g$ on $(T^n,i)\#_{T^{n-2}}(\bar M,i_2)$ with nonnegative scalar curvature in the distribution sense. Combining Miao's smoothing trick in \cite{Miao2002} and Kazdan's deformation in \cite{Kazdan82} we can construct a complete metric on $(T^n,i)\#_{T^{n-2}}(\bar M,i_2)$ with positive scalar curvature. In detail, it follows from \cite[Proposition 3.1]{Miao2002} and \cite[Lemma 4.1]{LM2019} that we can find a smooth perturbed metric such that the first Neumann eigenvalue of the conformal Laplace operator is positive in a large compact region. Now we can use \cite[Theorem A]{Kazdan82} to construct a complete conformal metric on $(T^n,i)\#_{T^{n-2}}(\bar M,i_2)$ with positive scalar curvature.
\end{proof}

In the following, we focus on asymptotically conical manifolds with angle $2\pi$ at infinity. In order to deal with the rigidity case, we need some more analysis.

\begin{lemma}\label{Lem: comparison function}
Given an asymptotically conical manifold $(M,g,\mathcal E)$ with $\beta=1$ and fixed a constant $0<\alpha<1$, there is a universal constant $r^*_0=r^*_0(M,g,\mathcal E,\alpha)$ such that the function 
$
w_1=\log r+\log^{\alpha} r
$ and $w_2=\log r-r^{-\mu/2}$
are superharmonic when $r\geq r^*_0$.
\end{lemma}
\begin{proof}
This comes from a direct calculation. From the decay condition \eqref{Eq: decay 2} we see
\begin{equation*}
\begin{split}
\Delta_g w_1&=\left(1+O(r^{-\mu})\right)\frac{\partial^2w_1}{\partial r^2}+\left(r^{-1}+O(r^{-1-\mu})\right)\frac{\partial w_1}{\partial r}\\
&=\alpha(\alpha-1)r^{-2}\log^{\alpha-2}r+O(r^{-2-\mu}).
\end{split}
\end{equation*}
and
$$
\Delta_g w_2=-\frac{\mu^2}{4}r^{-2-\frac{\mu}{2}}+O(r^{-2-\mu}).
$$
Clearly there is a universal constant $r^*_0$ such that $\Delta_g w_1<0$ and $\Delta_g w_2<0$ when $r\geq r^*_0$. This completes the proof.
\end{proof}

Let $\tilde g$ be a smooth metric on $M$ satisfying the asymptotical behavior \eqref{Eq: decay 2}. In the following, we consider the conformal Laplace operator $\tilde\Delta_{conf}$ with respect to the metric $\tilde g$ given by
$$
\tilde\Delta_{conf}=\Delta_{\tilde g}-\frac{n-2}{4(n-1)}R(\tilde g).
$$
For any compact region $U$, we use $\mu_1(\tilde\Delta_{conf},U)$ to denote the first Neumann eigenvalue of $\tilde\Delta_{conf}$ in $U$.

For convenience, we also extend the radical function $r$ from the end $\mathcal E$ (diffeomorphic to $(\mathbf R^2-B)\times F$) to the whole $M$ by defining it to be one outside $\mathcal E$. For any $\rho>1$ we use $B_\rho$ to denote the region $\{r<\rho\}$.

The following lemma provides nice comformal factors for later use.

\begin{lemma}\label{Lem: conformal factor}
Assume
\begin{itemize}
\item $\mu_1(\tilde\Delta_{conf},U)>0$ for some compact region $U$;
\item $R(\tilde g)\geq 0$ outside $U$.
\end{itemize}
Then we can find a sequence of smooth positive functions $u_k$ such that
\begin{itemize}
\item $-\tilde\Delta_{conf}u_k$ is quasi-positive\footnote{A function is said to be quasi-positive if it is nonnegative everywhere and positive somewhere.} in $B_{\rho_k}$ and $u_k=1$ on $\partial B_{\rho_k}$, where $\rho_k\to+\infty$ as $k\to+\infty$;
\item $u_k$ is bounded below by a positive constant;
\item the normal derivative
$$
\frac{\partial u_k}{\partial r}\geq \frac{\delta}{\rho_k\log\rho_k}
$$
for some positive universal constant independent of $k$.
\end{itemize}
\end{lemma}
\begin{proof}
Denote $\tilde U$ to be a larger compact region containing $U$. Since the first Neumann eigenvalue $\mu_1(\tilde\Delta_{conf}, U)$ is positive, there is a positive constant $c_*$ such that $\mu_1(\tilde\Delta_{conf},\tilde U)\geq c_*$. Namely, it holds
$$
\int_{\tilde U}|\tilde\nabla_{conf}\phi|^2\,\mathrm d\mu_{\tilde g}\geq c_*\int_{\tilde U}\phi^2\,\mathrm d\mu_{\tilde g}
$$
for all $\phi\in C^\infty(\tilde U)$. Here and in the sequel, let us formally write
$$
|\tilde\nabla_{conf}\phi|^2:=|\nabla_{\tilde g}\phi|^2+\frac{n-2}{4(n-1)}R(\tilde g)\phi^2
$$
for short. For later use, we fix a smooth cut-off function $\eta:M\to [0,1]$ with compact support such that $\eta\equiv 1$ in $U$ and $\eta\equiv 0$ outside $\tilde U$.

Take a sequence $\{\rho_k\}$ with $\rho_k\to+\infty$ as $k\to+\infty$ and let us construct the desired function $u_k$ from the exhaustion method. Let
$$
\mathcal E\subset \mathcal V_1\subset\mathcal V_2\subset\cdots
$$
be an exhaustion of $M$ such that the closure of $\mathcal V_j-\mathcal E$ is compact. Without loss of generality, we can assume that the region $U$ is contained in $\mathcal V_1$ and $B_{\rho_1}$. First we point out that the homogenous equation
$$
-\tilde\Delta_{conf}u=\frac{1}{2}\eta c_* u\quad\text{in}\quad \mathcal V_j\cap B_{\rho_k}
$$
with mixed boundary condition
$$
\frac{\partial u}{\partial\vec n_-}=0\quad \text{on}\quad \partial \mathcal V_j,\quad u=0\quad\text{on}\quad \partial B_{\rho_k}
$$
only has zero solution, where $\vec n_-$ is denoted to be the unit normal vector of $\partial\mathcal V_k$ in $M$ opposite to the end $\mathcal E$. To see this, we just need to integrate by parts and see
\[
\begin{split}
0&=\int_{\mathcal V_j\cap B_{\rho_k}}|\tilde\nabla_{conf}u|^2-\frac{1}{2}\eta c_* u^2\,\mathrm d\mu_{\tilde g}\geq \frac{1}{2}c_*\int_{\tilde U}u^2\,\mathrm d\mu_{\tilde g}.
\end{split}
\]
This yields $u\equiv 0$ in $\tilde U$. Since $u$ is a harmonic function outside $\tilde U$, it follows from the maximum principle that $u$ vanishes in the whole $\mathcal V_k\cap B_{\rho_k}$. From the Fredholm alternative, we see that the equation
$$
-\tilde\Delta_{conf}u_{kj}=\frac{1}{2}\eta c_* u_{kj}\quad\text{in}\quad \mathcal V_j\cap B_{\rho_k}
$$
with mixed boundary condition
$$
\frac{\partial u_{kj}}{\partial\vec n_-}=0\quad \text{on}\quad \partial \mathcal V_j,\quad u_{kj}=1\quad\text{on}\quad \partial B_{\rho_k}
$$
is solvable. 

Next let us show that up to a subsequence the functions $u_{kj}$ converge to a smooth positive function $u_k$ satisfying the equation
$$
-\tilde\Delta_{conf}u_{k}=\frac{1}{2}\eta c_* u_{k}\quad\text{in}\quad  B_{\rho_k}
$$
with $u_k=1$ on $\partial B_{\rho_k}$. Denote $v_{kj}=u_{kj}-1$. After integrating by parts and using the fact $\mu_1(\tilde\Delta_{conf},\tilde U)\geq c_*$ as well as $R(\tilde g)\geq 0$ outside $\tilde U$, we see that
\[
\begin{split}
\frac{1}{2}c_*\int_{\tilde U}v_{kj}^2\,\mathrm d\mu_{\tilde g}&\leq\int_{\mathcal V_j\cap B_{\rho_k}}|\tilde\nabla_{conf}v_{kj}|^2-\frac{1}{2}\eta c_*v_{kj}^2\,\mathrm d\mu_{\tilde g}\\
&=\frac{1}{2}c_*\int_{\mathcal V_j\cap B_{\rho_k}}\eta v_{kj}\,\mathrm d\mu_{\tilde g}.
\end{split}
\]
Combined with the H\"older inequality, we conclude
$$
\left(\int_{\tilde U}v_{kj}^2\,\mathrm d\mu_{\tilde g}\right)^{\frac{1}{2}}\leq \vol_{\tilde g}(\tilde U). 
$$
From the standard elliptic theory, the function $v_{kj}$ as well as $u_{kj}$ has a uniform $C^0$-norm in $U$ independent of $k$ and $j$. Since all functions $u_{kj}$ are positive, the Harnack inequality combined with the maximum principle yields that for any compact subset $K$ of $M$ there is a universal constant $C$ depending on $K$ such that $\|u_{kj}\|_{C^0(K)}\leq C$ for all large enough $j$. From the standard elliptic theory the functions $u_{kj}$ converge to a smooth positive function $u_k$ up to a subsequence and $u_k$ also satisfies the estimate 
\begin{equation}\label{Eq: C0 norm uk}
\|u_k\|_{C^0(K)}\leq C.
\end{equation}

At this stage, let us point out some extra useful estimates for the functions $u_k$. From a straightforward calculation we see
$$
\int_{\partial \mathcal V_1}\frac{\partial u_{kj}}{\partial \vec n_-}\,\mathrm d\sigma_{\tilde g}=\int_{\partial \mathcal V_j}\frac{\partial u_{kj}}{\partial \vec n_-}\,\mathrm d\sigma_{\tilde g}=0
$$
and
$$
\int_{\partial \mathcal V_1}u_{kj}\frac{\partial u_{kj}}{\partial \vec n_-}\,\mathrm d\sigma_{\tilde g}=\int_{\partial \mathcal V_j}u_{kj}\frac{\partial u_{kj}}{\partial \vec n_-}\,\mathrm d\sigma_{\tilde g}-\int_{\mathcal V_j-\mathcal V_1}|\nabla_{\tilde g} v_{kj}|^2\,\mathrm d\mu_{\tilde g}\leq 0
$$ 
As a result, the function $u_k$ satisfies
\begin{equation}\label{Eq: uk estimates}
\int_{\partial \mathcal V_1}\frac{\partial u_k}{\partial \vec n_-}\,\mathrm d\sigma_{\tilde g}=0\quad\text{and}\quad \int_{\partial \mathcal V_1}u_k\frac{\partial u_k}{\partial \vec n_-}\,\mathrm d\sigma_{\tilde g}\leq 0.
\end{equation}

In the following, we devote to show the lower bound estimate for normal derivative $\partial_ru_k$ on $\partial B_{\rho_k}$. From previous discussion the functions $u_k$ has a uniform $C^0$-norm in each compact subsets and so $u_k$ converges to a nonnegative function $u_\infty$ up to a subsequence. From the Harnack inequality there are two possibilities:

{\it Case 1. $u_\infty\equiv 0$.} In this case, we can obtain the desired estimate from a comparison argument. Let $r_0^*$ and $w_2$ be the same as in Lemma \ref{Lem: comparison function}. Notice that the value of $u_k$ on $\{r=r_0^*\}$ will be less than $1/3$ when $k$ is large enough. So we have
$$
u_{k}\leq \frac{1}{3}+\frac{2}{3}\frac{w_2}{\log \rho_k}\quad \text{on}\quad\partial B_{r_0^*}
$$
and
$$u_k=\frac{1}{3}+\frac{2}{3}\frac{w_2}{\log \rho_k}=1\quad \text{on}\quad \partial B_{\rho_k}.$$
Since $w_2$ is super-harmonic when $r_0^*\leq r\leq \rho_k$, we conclude
$$
u_k\leq\frac{1}{3}+\frac{2}{3}\frac{w_2}{\log \rho_k}\quad\text{when}\quad r_0^*\leq r\leq \rho_k,
$$
and
$$\frac{\partial u_k}{\partial r}\geq \left(\frac{2}{3}+o(1)\right)\frac{1}{\rho_k\log\rho_k}\geq \frac{1}{2\rho_k\log\rho_k} \quad \text{on}\quad \partial B_{\rho_k}
$$
for $k$ large enough.

{\it Case 2. $u_\infty$ is a positive function.} From \eqref{Eq: uk estimates} we have
$$
\int_{\partial \mathcal V_1}u_\infty\frac{\partial u_\infty}{\partial \vec n_-}\,\mathrm d\sigma_{\tilde g}\leq 0
$$
and so for any $\rho>1$ we have
\[
\begin{split}
\int_{\partial B_\rho}u_\infty\frac{\partial u_\infty}{\partial \vec n_+}\,\mathrm d\sigma_{\tilde g}&=\int_{\mathcal V_1\cap B_\rho}|\tilde\nabla_{conf}u_\infty|^2-\frac{1}{2}\eta c_*u_\infty^2\,\mathrm d\mu_{\tilde g}-\int_{\partial\mathcal V_1}u_\infty\frac{\partial u_\infty}{\partial \vec n_-}\,\mathrm d\sigma_{\tilde g}\\
&\geq \frac{1}{2}c_*\int_{\tilde U}u_\infty^2\,\mathrm d\mu_{\tilde g}>0,
\end{split}
\]
where $\vec n_+$ is denoted to be the unit normal of $\partial B_\rho$ in $M$ pointing to the end $\mathcal E$. From the maximum principle as well as \eqref{Eq: C0 norm uk} we know
$$
u_k\leq\max\left\{\max_{\partial B_1}u_k,1\right\} \leq C
$$
for some universal constant $C$ independent of $k$ and the same estimate holds for $u_\infty$. It then follows from \cite[Theorem 5]{Moser1961} that $u$ has a finite limit as $r\to+\infty$, which then yields
$$
u_\infty=O(1)\quad\text{and}\quad |\nabla_{\tilde g}u_\infty|=o(r^{-1}),\quad \text{as}\quad r\to+\infty.
$$
This implies
$$
\int_{\partial B_\rho}u_\infty\frac{\partial u_\infty}{\partial \vec n_+}\,\mathrm d\sigma_{\tilde g}=o(1),\quad\text{as}\quad \rho\to+\infty,
$$
which is impossible.

Finally, we do some modification for $u_k$ such that it is bounded below by a positive constant. Notice that $-\tilde\Delta_{conf}$ is quasi-positive in $B_{\rho_k}$ and bounded below by a positive constant in $U$. Since the scalar curvature $R(\tilde g)$ can only be negative in $U$, we can replaced $u_k$ by
$$
\frac{u_k+\epsilon_k}{1+\epsilon_k}
$$
with some small positive constant $\epsilon_k$ such that all desired properties are satisfied.
\end{proof}

Now we are ready to prove Theorem \ref{Thm: main 5}.
\begin{proof}[Proof for Theorem \ref{Thm: main 5}]
First we show that the angle of $(M,g,\mathcal E)$ at infinity is no greater than $2\pi$. Otherwise, Proposition \ref{Prop: PMT to PSC 2} yields that the generalized connected sum $(T^n,i)\#_{T^{n-2}}(\bar M,i_2)$ admits a complete metric with positive scalar curvature. In the following, we will investigate the topological restrictions of $(T^n,i)\#_{T^{n-2}}(\bar M,i_2)$ from our assumption on $M$. Recall that the gluing procedure is illustrated by the following Figure \ref{Fig: 7}.
\begin{figure}[htbp]
\centering
\includegraphics[width=12cm]{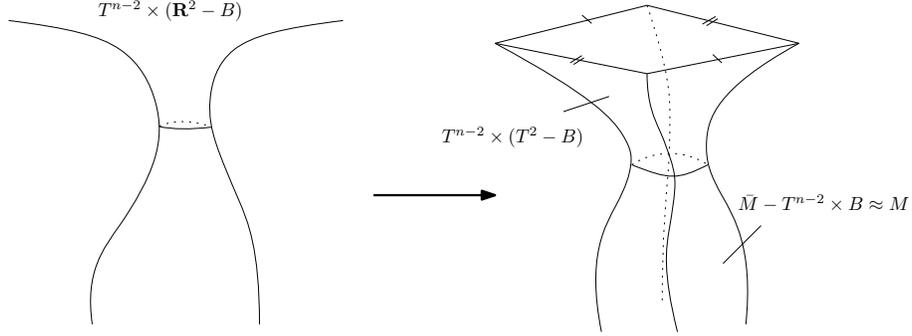}
\caption{The gluing procedure}
\label{Fig: 7}
\end{figure}

With the help of the coordinate system on the end $\mathcal E$, the manifold $T^{n-2}\times (T^2-B)$ is foliated by $T^{n-2}$-slices and this induces a canonical perturbation $s:T^{n-2}\to \partial(T^{n-2}\times B)$ of the embedding $i:T^{n-2}\to T^n$, which maps $T^{n-2}$ to some $T^{n-2}$-slice in $\partial(T^{n-2}\times B)$. Denote 
$$\Phi:\partial\left(T^{n-2}\times(T^2-B)\right)\to \partial\left(\bar M-T^{n-2}\times B\right) $$
to be the gluing map. From the assumption on $M$, we see that
\begin{itemize}
\item the map $\Phi\circ s:T^{n-2}\to \bar M-T^{n-2}\times B$ is incompressible for $n\geq 3$;
\item or $n=3$ and the map $\Phi\circ s:\mathbf S^1\to \bar M-\mathbf S^1\times B$ is homotopically non-trivial.
\end{itemize}
With the exactly same argument in the proof of Proposition \ref{Prop: main 2} and \ref{Prop: main 3}, this condition\footnote{This is enough since in the proof of Proposition \ref{Prop: main 2} and Proposition \ref{Prop: main 3} the contradiction is deduced by showing non-contractibility of certain curves in $M_2-U_2$.} can be used to rule out the existence of complete metrics with positive scalar curvature on
$(T^n,i)\#_{T^{n-2}}(\bar M,i_2)$, but this contradicts to the implication from Proposition \ref{Prop: PMT to PSC 2}.

Next let us show the flatness of $(M,g)$ under the assumption $\beta=1$. The proof will be divided into three steps.

{\it Step 1. $(M,g)$ is scalar flat.}
Otherwise, the scalar curvature $R(g_\beta)$ must be positive at some point $p$ in $M$. Let $\tilde g=g_\beta$ and $U$ be a small geodesic centered at the point $p$. Then it is easy to verify that $\mu_1(\tilde\Delta_{conf},U)>0$ and that $R(\tilde g)$ is nonnegative outside $U$. From Lemma \ref{Lem: conformal factor} we can find a sequence of positive smooth solutions $u_k$ such that 
\begin{itemize}
\item $-\tilde\Delta_{conf}u_k$ is quasi-positive in $B_{\rho_k}$ and $u_k=1$ on $\partial B_{\rho_k}$, where $\rho_k\to+\infty$ as $k\to+\infty$;
\item $u_k$ is bounded below by a positive constant;
\item the normal derivative
$$
\frac{\partial u_k}{\partial r}\geq \frac{\delta}{\rho_k\log\rho_k}
$$
for some positive universal constant independent of $k$.
\end{itemize}
Now we would like to follow along the proof of Proposition \ref{Prop: PMT to PSC 2} with some modifications. In the definition of the piecewisely smooth metric $\bar g$, we add a conformal deformation. Namely, we let
\begin{equation*}
\bar g=\left\{
\begin{array}{cc}
\bar g_\beta:=u_k^{\frac{4}{n-2}} \tilde g,&r\leq \rho_k;\\
u^2\mathrm dr^2+\bar g_r,&\rho_k\leq r\leq 2\rho_k;\\
\beta^2g_{euc}\oplus g_{flat}, &r\geq 2\rho_k.
\end{array}
\right.
\end{equation*}
We point out that the metric $\bar g$ is complete due to the positive lower bound for $u_k$. In the following, the analysis is similar as before. The mean curvature of $S_{\rho_k}$ with respect to $\bar g_\beta$ satisfies
$$
H_{\rho_k,\bar g_\beta}\geq\frac{1}{\rho_k}+\frac{\delta'}{\rho_k\log\rho_k}
$$
for some positive constant $\delta'$ independent of $\rho_k$. As a result, the initial value of $u$ can be chosen to satisfy 
$$u(\rho_k,\cdot)\leq 1-\delta'\log^{-1}\rho_k+O(\rho_k^{-\epsilon}).$$ Doing the same analysis as in the proof of Proposition \ref{Prop: PMT to PSC 2}, we have 
$$u(2\rho_k,\cdot)\leq 1-\delta'\log^{-1}\rho_k +O(\rho_k^{-\epsilon}).$$ This yields
$$
H_{2\rho_k}=\frac{1}{2\rho_k u}\geq \frac{1}{2\rho_k}+\frac{\delta'}{2\rho_k\log\rho_k}+O(\rho_k^{-\epsilon}).
$$
On the other hand, we have
$$
H_{2\rho_k,prod}=\frac{1}{2\rho_k}.
$$
By taking $\rho_k$ to be large enough, we still have the strict inequality $H_{2\rho_k}>H_{2\rho_k,prod}$ and this further implies the existence of a complete metric on the generalized connected sum $(T^n,i)\#_{T^{n-2}}(\bar M,i_2)$ with positive scalar curvature, which is impossible from the previous discussion. 

{\it Step 2. $(M,g)$ is Ricci flat.}
This comes from a standard deformation argument. Assume otherwise that the Ricci curvature of the metric $g$ does not vanish everywhere. From \cite[Lemma 3.3]{Kazdan82} we can pertube the metric $g$ inside a small geodesic ball $U$ such that the purtubed metric (denoted by $\tilde g$) satisfies
$
\mu_1(\tilde\Delta_{conf},U)>0.
$
Since the metric $g$ is unchanged outside $U$, we see $R(\tilde g)\geq 0$ outside $U$. With the same argument as in Step 1, we can deduce the same contradiction. As a result, $(M,g)$ has to be Ricci flat.

{\it Step 3. $(M,g)$ is flat.}
In dimension three, flatness is equivalent to Ricci-flatness and we are already done. So we just need to deal with the case when the dimension of $M$ is no less than four. We try to use the volume comparison theorem for Ricci nonnegative Riemannian manifolds. Let us denote $\tilde {\mathcal E}$ to be the universal covering of $\mathcal E$. Clearly, $\tilde {\mathcal E}$ is diffeomorphic to $[1,+\infty)\times\mathbf R\times \mathbf R^{n-2}$ and the covering map $\pi:\tilde{\mathcal E}\to \mathcal E\approx (\mathbf R^2-B)\times T^{n-2}$ is given by
$$
(r,\theta,x)\mapsto (r,\theta,\bar x)
$$
with the help of the polar coordinate system at infinity, where $x\mapsto \bar x$ is the quotient map from $\mathbf R^{n-2}$ to $T^{n-2}$. Denote $\pi:(\tilde M,\tilde g)\to (M,g)$ to be the universal covering of $(M,g)$ and $e:\mathcal E\to M$ to be the canonical embedding of the end $\mathcal E$. We can lift $e$ to a map $\tilde e:\tilde{\mathcal E}\to \tilde M$ such that we have the following commutative diagram
\begin{equation}\label{Eq: lift diagram}
\xymatrix{\tilde{\mathcal E}\ar[r]^{\tilde e}\ar[d]^{\pi}&\tilde M\ar[d]^{\pi}\\
\mathcal E\ar[r]^{e}&M.}
\end{equation}

Denote
$$
\Omega=\{(r,\theta,x):r\geq 1,\,0<\theta<2\pi\}\subset \tilde{\mathcal E}.
$$
Let us show that $\tilde e|_{\Omega}:\Omega\to \tilde M$ is an embedding. The only thing needs to prove is that the map $\tilde e|_{\Omega}$ is injective. To see this, we pick up two different points $p$ and $q$ in $\tilde{\mathcal E}$ and try to show $\tilde e(p)\neq \tilde e(q)$. Otherwise, a path $\gamma:[0,1]\to \Omega$ with $\gamma(0)=p$ and $\gamma(1)=q$ maps to a closed curve $\pi(\gamma)$ in $\mathcal E$, whose image $(e\circ \pi)(\gamma)$ in $M$ is homotopically trivial due to the commutative diagram \ref{Eq: lift diagram}. From the definition of $\Omega$, the closed curve $(e\circ\pi)(\gamma)$ lies in $(i_{T^{n-2}})_*(\pi_1(T^{n-2}))$. From the assumption we know that
$$
(i_{T^{n-2}})_*:\pi_1(T^{n-2})\to \pi_1(M)
$$
is injective and so $\pi(\gamma)$ can be homotopic to a point in $\mathcal E$. This is impossible since $\mathcal E$ is the universal covering and the path $\gamma$ has different end-points.

Fix a point $\tilde p$ in $\tilde M$. From the comparison theorem the flatness of $(M,g)$ will follow from the estimate
$$
\liminf_{\rho\to+\infty}\frac{\vol_{\tilde g}(B_{\tilde g}(\tilde p,\rho))}{\omega_n\rho^n}\geq 1.
$$
For convenience, let us denote $\Omega_{r_0}=\Omega\cap \{r\geq r_0\}$ for any $r_0>1$. It is clear that for any $\epsilon>0$ there is a $r_0>1$ such that
$$
(1-\epsilon)\tilde g_{euc}\leq \tilde g\leq (1+\epsilon)\tilde g_{euc}\quad\text{on}\quad \Omega_{r_0},
$$
where $\tilde g_{euc}$ is the metric on $\Omega$ given by $\tilde g_{euc}=\mathrm dr^2+r^2\mathrm d\theta^2+\mathrm dx^2$. The figure \ref{Fig: 9} below illustrates how $\Omega_{r_0}$ looks like $\mathbf R^n$ with certain parts removed.
\begin{figure}[htbp]
\centering
\includegraphics[width=9cm]{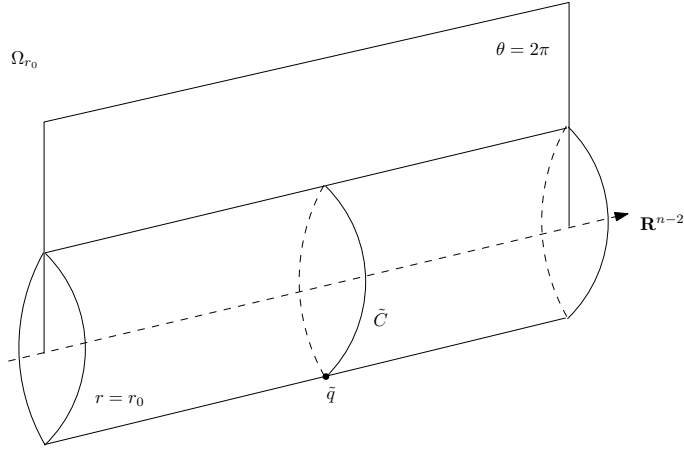}
\caption{The region $\Omega_{r_0}$}
\label{Fig: 9}
\end{figure}

In the following, we will calculate the volume of the intersection $B_{\tilde g}(\tilde p,\rho)\cap\Omega$. For convenience, we denote $\tilde q=(r_0,\pi,0)$ and $\tilde C=\{(r_0,\theta,0):0<\theta<2\pi\}$. Clearly we have $\diam_{\tilde g}\tilde C\leq (1+\epsilon)2\pi r_0$. We also let $\tilde r=\sqrt{r^2+|x|^2}$ be the radical distance function on $\Omega_{r_0}$. For any point $\tilde q'=(r,\theta,x)$ with $\tilde r(\tilde q')=\rho$, we can find the path
$$
\tilde \gamma(t)=((1-t)r_0+tr,\theta,tx),\quad t\in[0,1],
$$
connecting $\tilde C$ and $\tilde q'$. A direct computation shows that the $\tilde g$-length of the curve $\gamma$ does not exceed $(1+\epsilon)\rho$. Denote $\Lambda=\dist(\tilde p,\tilde q)+(1+\epsilon)2\pi r_0$. Now we see
\[
\begin{split}
\vol_{\tilde g}(B_{\tilde g}(\tilde p,\rho+\Lambda))&\geq\vol_{\tilde g}(B_{\tilde g}(\tilde C,\rho)\subset \Omega)\\
&\geq (1-\epsilon)^n\vol_{\tilde g_{euc}}\left(\{\tilde r\leq (1+\epsilon)^{-1}\rho\}\subset \Omega\right)\\
&=\left(\frac{1-\epsilon}{1+\epsilon}\right)^n\omega_n\rho^n+O(\rho^{n-2}),\quad \text{as}\quad \rho\to+\infty.
\end{split}
\]
This implies $$
\liminf_{\rho\to+\infty}\frac{\vol_{\tilde g}(B_{\tilde g}(\tilde p,\rho))}{\omega_n\rho^n}\geq \left(\frac{1-\epsilon}{1+\epsilon}\right)^n,
$$
and the proof is completed by letting $\epsilon\to 0$.
\end{proof}

\section{Proof for Proposition \ref{Prop: main 7}}
This section is devoted to a detailed proof for Proposition \ref{Prop: main 7}. Instead of a direct proof, we are going to show the following stronger result first.

\begin{proposition}\label{Prop: main 6}
Let $(\Sigma,g_\Sigma)$, $\dim\Sigma \geq 2$, be a convex hypersurface or closed curve in Euclidean space and $(F,g_F)$ be a circle $\mathbf S^1$ or a flat $k$-torus $(T^k,g_{flat})$. If $(\Omega,g)$, $\dim \Omega \leq 7$, is a compact manifold with $R(g)\geq 0$ and mean convex boundary such that
\begin{itemize}
\item $\partial\Omega$ with the induced metric is isometric to the Riemannian product $\Sigma\times F$ through a diffeomorphism $\phi=(\phi_1,\phi_2): \partial \Omega\to\Sigma\times F$;
\item the induced map $(\phi_{2}^{-1})_*:\pi_1( F) \to \pi_1(\Omega)$ is injective when $F=T^k$
or $(\phi_{2}^{-1})_*:\pi_1(F) \to \pi_1(\Omega)$ is non-zero when $F=\mathbf S^1$,
\end{itemize}
then 
\begin{equation}\label{Eq: total curvature estimates}
\int_{\partial\Omega} H\,\mathrm d\sigma_g\leq \Lambda(\Sigma,g_{\Sigma})\cdot \vol(F,g_F),
\end{equation}
where $H$ is the mean curvature of $\partial\Omega$ in $\Omega$ with respect to the unit outer normal and $\Lambda(\Sigma,g_\Sigma)$ is denoted to be the total mean curvature of $\Sigma$ in the Euclidean space. Moreover, if we have the equality in \eqref{Eq: total curvature estimates}, then the metric $g$ is flat.
\end{proposition}
\begin{proof}
The proof is based on the quasi-spherical metric. In \cite{ST2002}, it was shown that the quasi-spherical metric over the exterior region of Euclidean space is always asymptotically flat at infinity such that the positive mass theorem comes into play. Here we take a different way combining the quasi-spherical metric with the compactification trick such that Proposition \ref{Prop: main 2} and \ref{Prop: main 3} can be applied. 
Let us divide our discussion into two cases.

{\it Case 1. $\dim \Sigma\geq 2$.} Denote $\Omega_{ext}$ to be the exterior region outside of $\Sigma$ in the Euclidean space. Since $\Sigma$ is convex, we can find an equidistant foliation
$$
\Psi:\Sigma\times [0,+\infty)\to \Omega_{ext}
$$
and the Euclidean metric on $\Omega_{ext}$ can be written as
$
g_{euc}=\mathrm dt^2+g_t
$, 
where $g_t$ is the induced metric of the hypersurface $\Sigma_t=\Psi(\Sigma\times\{t\})$. After taking the $T^k$-component into consideration, let us denote $\bar\Omega_{ext}=\Omega_{ext}\times T^k$ and $\bar\Sigma_t=\Sigma_t\times T^k$. We also write $\bar g_t=g_t+g_{flat}$ and $\bar g=\mathrm dt^2+\bar g_t$. As before, we consider the following quasi-spherical metric equation
\begin{equation}\label{Eq: QS}
\bar H_t\frac{\partial u}{\partial t}=u^2\Delta_{\bar g_t}u+\frac{1}{2}R(\bar g_t)(u-u^3),\quad u(0,\cdot)=u_0>0,
\end{equation}
which guarantees the quasi-spherical metric $\bar g_{qs}=u^2\mathrm dt^2+\bar g_t$ has vanishing scalar curvature. Here $\bar H_t$ and $R(\bar g_t)$ are denoted to be the mean curvature and the scalar curvature of the hypersurface $\bar\Sigma_t$ and $u_0:\Sigma\to \mathbf R$ is a positive function to be determined later.

First we point out that the equation \eqref{Eq: QS} is solvable no matter what initial value $u_0$ is prescribed. To see this, notice that $\Sigma_t$ is convex all the time and so the scalar curvature $R(\bar g_t)$ is nonnegative. Then the parabolic comparison principle yields the following a priori $C^0$-estimate
$$
\min\left\{1,\min_\Sigma u_0\right\}\leq u\leq  \max\left\{1,\max_\Sigma u_0\right\}.
$$
This means that $u$ cannot blow up in finite time and so the solution $u$ exists on $\bar\Omega_{ext}$. 

Now we analyze the behavior of the solution $u$ at infinity. From the same comparison argument as in \cite[Lemma 2.2]{ST2002}, it is not difficult to deduce the basic estimate
\begin{equation}\label{Eq: estimate u}
|u-1|\leq Ct^{2-m},
\end{equation}
where $C$ is a universal constant independent of $t$ and $m=\dim\Sigma+1$. Next let us investigate the function $v=t^{m-2}(u-1)$. From above estimate we see that $v$ is uniformly bounded. From the equation \eqref{Eq: QS} the function $v$ satisfies
\begin{equation}\label{Eq: v}
\bar H_t\frac{\partial v}{\partial t}=u^2\Delta_{\bar g_t}v+\left((m-2)\frac{\bar H_t}{t}-\frac{1}{2}R(\bar g_t)u(u+1)\right)v.
\end{equation}
A flat $k$-torus $(T^k,g_{flat})$ can be viewed as the quotient space $(\mathbf R^k,g_{euc})/\Gamma$ for some lattice $\Gamma$ in $\mathbf R^k$. This induces a covering map
$$
\pi:\Omega_{ext}\times \mathbf R^k\to \bar\Omega_{ext}.
$$
For convenience, we will lift the functions $u$ and $v$ on $\Omega_{ext}\times \mathbf R^k$. Namely, we define $\tilde u=u\circ \pi$ and $\tilde v=v\circ \pi$. Clearly both $\tilde u$ and $\tilde v$ are $\Gamma$-invariant. For any positive integer $l$, we denote the scaling
\begin{equation}\label{Eq: parabolic scaling}
\Phi_l:\Sigma\times \mathbf R^k\times [0,+\infty)\to\Sigma\times \mathbf R^k\times[0,+\infty),\quad (\sigma,x,t)\mapsto (\sigma,lx,lt).
\end{equation}
Now let us consider the function
$
\tilde v_l=\tilde v\circ \Phi_l.
$
Denote $\tilde u_l=\tilde u\circ \Phi_l$. Let
$$\tilde H_{t,l}=l\cdot(\bar H_t\circ \pi\circ \Phi_l)$$
and 
$$
\tilde g_{t,l}=l^{-2}\cdot(\pi\circ\Phi_l)^*(\bar g_{lt}).
$$
It is not difficult to check that $\tilde v_l$ satisfies the equation
\begin{equation}
\tilde H_{t,l}\frac{\partial \tilde v_l}{\partial t}=\tilde u_l^2\Delta_{\tilde g_{t,l}}\tilde v_l+\left((m-2)\frac{\tilde H_{t,l}}{t}-\frac{1}{2}R(\tilde g_{t,l})\tilde u_l(\tilde u_l+1)\right)\tilde v_l.
\end{equation}
Finally we point out that all $\tilde v_l$ are $\Gamma$-invariant.

Now we would like to investigate the limit of functions $\tilde v_l$. First we need to understand the behavior of $\tilde v$ as $t\to+\infty$. Define
$$
\beta(t)=\min_{\Sigma\times \mathbf R^k} \tilde v(t,\cdot).
$$
Let us show that $\beta(t)$ has a limit as $t\to+\infty$. Notice that
$$
\beta(t)=\min_{\Sigma\times T^k}  v(t,\cdot).
$$
From the equation \eqref{Eq: v} and the estimate \eqref{Eq: estimate u} as well as the estimates for $\bar H_t$ and $R(\bar g_t)$ from \cite[Lemma 2.1]{ST2002}, the parabolic comparison principle yields 
$$
\beta(t_2)\geq \beta(t_1)-Ct_1^{-1},\quad \text{for all}\quad 1\leq t_1\leq t_2, 
$$
where $C$ is a universal constant independent of $t_1$ and $t_2$. This implies that $\beta(t)$ has a limit as $t\to+\infty$, denoted by $\beta_0$. Based on this we can show that the functions $\tilde v_l$ converge smoothly to the constant function $\tilde v_\infty\equiv \beta_0$ in compact subsets of $\Sigma\times \mathbf R^k\times (0,+\infty)$. To see this, first notice that all $\tilde v_l$ is uniformly bounded by the bound for $v$. It follows from \cite[Lemma 2.1]{ST2002} that 
$
\tilde H_{t,l}
$
converges to $(m-1)t^{-1}$ smoothly and that $\tilde g_{t,l}$ converges to $\tilde g_{t,\infty}=t^2g_{round}+g_{euc}$ smoothly as $l\to +\infty$, where $g_{round}$ is a smooth metric on $\Sigma$ with constant curvature $1$. As a result, $\tilde v_{l}$ has uniformly bounded $C^k$-estimate for any $k$. Up to a subsequence, $\tilde v_l$ converges to a limit smooth function $\tilde v_\infty$ satisfying
$$
\frac{m-1}{t}\frac{\partial \tilde v_\infty}{\partial t}=\Delta_{\tilde g_{t,\infty}}\tilde v_\infty.
$$
Notice that $\beta(t)$ is always attained by $\tilde v(t,\cdot)$ in a fixed bounded compact region in $\Sigma\times \mathbf R^k$. From previous discussion, $\tilde v_\infty$ attains its minimum $\beta_0$ in the interior of $\Sigma\times \mathbf R^k\times(0,+\infty)$. The strong maximum principle yields that $\tilde v_\infty\equiv \beta_0$ and so
\begin{equation}
\lim_{t\to+\infty}v(t,\cdot)=\beta_0.
\end{equation}

Now let us deduce a contradiction under the assumption
\begin{equation}\label{Eq: contradiction}
\int_{\partial\Omega} H\,\mathrm d\sigma_g> \Lambda(\Sigma,g_{\Sigma})\cdot \vol(F,g_F).
\end{equation}
Correspondingly let us set 
$$
u_0=\frac{\bar H_0}{H}.
$$
From a similar calculation as in \cite[Lemma 4.2]{ST2002}, it follows that the integral
\begin{equation}\label{Eq: total mean curvature slice}
\int_{\bar\Sigma_t}\bar H_t\left(u^{-1}(t,\cdot)-1\right)\mathrm d\sigma_{\bar g_t}
\end{equation}
is monotone increasing as $t$ increases. From the inequality \eqref{Eq: contradiction} we conclude that the constant $\beta_0$ is negative and so for $t$ large enough the mean curvature of $\Sigma_t$ with respect to the quasi-spherical metric $\bar g_{qs}$ is greater than that with respect to $\bar g$. From this we can obtain a contradiction with the same gluing argument as in the proof of Proposition \ref{Prop: PMT to PSC 2}.

For the equality case, we observe that the integral \eqref{Eq: total mean curvature slice} has to vanish all the time. After taking the derivative with respect to $t$ and using the equation of $u$, we see
$$
\int_{\bar \Sigma_t}R(\bar g_t)u^{-1}(t,\cdot)\left(u(t,\cdot)-1\right)^2\mathrm d\sigma_{\bar g_t}\equiv 0
$$
and so $u$ must be one identically. This yields that we can glue $(\Omega,g)$ with $(\bar\Omega_{ext},\bar g)$ along their boundaries with the same mean curvatures on both sides. Now the gluing argument is still valid except that we cannot construct a smooth metric with positive scalar curvature simply from smoothing and conformal tricks. In this case, we can use the Ricci-DeTurk flow (see \cite{ST2018} for instance) to construct a family of smooth metrics with nonnegative scalar curvature (Notice that the underlying space is closed now). From Proposition \ref{Prop: main 2} and \ref{Prop: main 3} it follows that these metrics are flat and so the metric $g$ has to be flat.

{\it Case 2. $\dim\Sigma=1$, i.e. $\Sigma$ is a closed curve.} The proof follows from the same idea as above but some analysis are different. Since all closed curves with the same length are isometric, we can assume $\Sigma$ to be some round curve $C_{r_0}$ without loss of generality. 
Let us adopt the same notations as above and consider the following quasi-spherical metric equation
\begin{equation}\label{Eq: QS 2}
\frac{1}{r_0+t}\frac{\partial u}{\partial t}=u^2\Delta_{\bar g_t}u,\quad u(0,\cdot)=u_0>0.
\end{equation}

First it follows from parabolic comparison principle that
$$
\min_{\Sigma\times T^k} u\leq u\leq \max_{\Sigma\times T^k} u.
$$
From this we conclude that $u$ exists on $\bar\Omega_{ext}$. We are going to show that $u(t,\cdot)$ converges to a constant as $t\to+\infty$. As before, we lift $u$ to a function $\tilde u$ on $\Omega_{ext}\times \mathbf R^k$ and denote $\tilde u_l=\tilde u\circ \Phi_l$ with $\Phi_l$ the scaling given by \eqref{Eq: parabolic scaling}. Denote 
$$
\beta(t)=\min_{\Sigma\times \mathbf R^k} \tilde u(t,\cdot).
$$
Based on the parabolic maximum principle, we know that $\beta(t)$ must be monotone when $t$ is large enough and so it has a limit $\beta_0$ as $t\to+\infty$. Similar as before, we can show that the function $\tilde u_l$ converges to $\beta_0$ as $l\to+\infty$ and so we have
$$
\lim_{t\to+\infty}u(t,\cdot)=\beta_0.
$$
From the equation \eqref{Eq: QS 2} we have
$$
\frac{\mathrm d}{\mathrm dt}\int_{\bar\Sigma_t}u^{-1}(t,\cdot)\,\mathrm d\sigma_{\bar g_t}=\frac{1}{r_0+t}\int_{\bar\Sigma_t}u^{-1}(t,\cdot)\,\mathrm d\sigma_{\bar g_t}.
$$
It then follows
$$
\beta_0=2\pi r_0\vol(T^k,g_{flat})\left(\int_{\bar\Sigma_0}u_0^{-1}\,\mathrm d\sigma_{\bar g_0}\right)^{-1}.
$$

Now let us deduce the contradiction under the assumption
$$
\int_{\partial\Omega}H\,\mathrm d\sigma_g>2\pi\vol(T^k,g_{flat}).
$$
Correspondingly we set $u_0=(r_0 H)^{-1}$ and so $\beta_0<1$. As a result, for $t$ large enough the mean curvature of $\Sigma_t$ with respect to the quasi-spherical metric $u^2\mathrm dt^2+\bar g_t$ is greater than that with respect to $\bar g$. Again we can obtain a contradiction with the same gluing argument as in the proof of Proposition \ref{Prop: PMT to PSC 2}. 

For the equality case, the scalar flatness and the Ricci flatness of the metric $g$ come from the deformation arguments in \cite[Theorem 4.2]{SWY2019} and \cite[Corollary 2.1]{MST2010} respectively. In the following, we would like to show that the metric $g$ is flat. From the proof of Theorem \ref{Thm: main 5} this is true if the regions $(\Omega,g)$ and $(\bar\Omega_{ext},\bar g_{qs})$ are glued in a smooth manner. So the main difficuly lies in the potential singularity along the corner and we plan to deal with this issue by Ricci-DeTurk flow. We hope that the desired flatness can be preserved during the Ricci flow and there are several things to be done. 

First let us check that the quasi-spherical metric $\bar g_{qs}$ is actually asymptotically conical with angle $2\pi$ at infinity. Let $v=u-1$. This is equivalent to the following estimate
\begin{equation}\label{Eq: decay qs metric}
|v|+r|\nabla_{\bar g}v|+r^2|\nabla^2_{\bar g}v|=O(r^{-\mu}),\quad \text{as}\quad r\to +\infty,
\end{equation}
for some $\mu>0$, where $r=r_0+t$ is the radical distance of $\mathbf R^2$. To analyze the behavior of the function $v$, we lift it to a function $\tilde v$ on $\Omega_{ext}\times \mathbf R^k$ and investigate its scaling $\tilde v_l=\tilde v\circ\Phi_l$ as before. From a direct computation, $\tilde v_l$ satisfies the following equation
\begin{equation}\label{Eq: uniform parabolic}
\frac{l}{r_0+l t}\frac{\partial \tilde v_l}{\partial t}=(\tilde v_l+1)^2\Delta_{\tilde g_{t,l}}\tilde v_l\quad\text{in}\quad   \Sigma\times \mathbf R^k\times[0,+\infty),
\end{equation}
where $\tilde g_{t,l}$ is the smooth metric on $\Sigma\times \mathbf R^k$ given by 
$$
\tilde g_{t,l}=\frac{(r_0+l t)^2}{l^2}\mathrm d\theta^2+g_{euc}.
$$
We point out that the equation $\eqref{Eq: uniform parabolic}$ gives a sequence of locally uniform parabolic equations. Denote
$$
M_l(t)=\max_{\Sigma\times \mathbf R^k}\tilde v_l(t,\cdot)\quad\text{and}\quad m_l(t)=\min_{\Sigma\times \mathbf R^k}\tilde v_l(t,\cdot).
$$
From parabolic maximum principle, the function $M_l(t)$ is decreasing and the function $m_l(t)$ is increasing as $t$ increases.
Recall that $\tilde v_l$ is $\Gamma$-invariant, where $\Gamma$ is the descrete group such that $(T^k,g_{flat})$ is the quotient $(\mathbf R^k,g_{euc})/\Gamma$.
From \cite[Theorem 1.1]{GS2021} and a covering argument we can apply the Harnack inequality to the functions $M_l(1)-\tilde v_l$ and $\tilde v_l-m_l(1)$ in $\Sigma\times \mathbf R^k\times[1,4]$. As a result, we obtain
\[
\begin{split}
M_l(1)-m_l(2)&=\max_{2\leq t\leq 3}(M_l(1)-\tilde v_l)\\
&\leq C\min_{2\leq t\leq 3}(M_l(1)-\tilde v_l)=C(M_l(1)-M_l(2))
\end{split}
\]
and
\[
\begin{split}
M_l(2)-m_l(1)&=\max_{2\leq t\leq 3}(\tilde v_l-m_l(1))\\
&\leq C\min_{2\leq t\leq 3}(\tilde v_l-m_l(1))=C(M_l(2)-m_l(1)),
\end{split}
\]
for some universal constant $C$ independent of $\tilde v_l$. Let us denote
$$\Osc_l(t)=M_l(t)-m_l(t)$$ and then we have $\Osc_l(2)\leq c\Osc_l(1)$ for some universal constant $0<c<1$ independent of $l$. From the construction of $\tilde v_l$, this further implies 
$$\Osc(2l)\leq c\Osc(l),\quad \forall\,l\geq 1,$$
where $\Osc(t)$ is denoted to be the oscillation of $v(t,\cdot)$ on $\Sigma\times T^k$. Now it is standard to deduce $\Osc(t)=O(t^{-\mu})$ for some $\mu>0$ as $t\to +\infty$. Since we have
$$
\int_{\bar\Sigma_t}u^{-1}(t,\cdot)\,\mathrm d\sigma_{\bar g_t}\equiv 1,
$$ 
the function $v$ either is identical to zero or changes sign on $\bar\Sigma_t$. As a result, its $C^0$-norm on $\bar\Sigma_t$ is well controlled by its oscillation and so we have the estimate $|v|=O(r^{-\mu})$ with $r=r_0+t$. As a consequence, we see that the function $\tilde v_l$ satisfies $|\tilde v_l(t,\cdot)|=O(l^{-\mu})$ when $1\leq t\leq 5$. With a similar argument as in the proof of \cite[Lemma 2.5]{ST2002}, we conclude that
$$
|\tilde v_l(t,\cdot)|+|\partial \tilde v_l(t,\cdot)|+|\partial^2\tilde v_l(t,\cdot)|=O(l^{-\mu}),\quad 2\leq t\leq 4.
$$
From the definition of the map $\Phi_l$ as well as the fact that the length of $\partial_\theta$ is comparable to $l$ when $2l\leq t\leq 4l$, it is easy to deduce
$$
|v|+t|\nabla_{\bar g}v|+t^2|\nabla^2_{\bar g}v|=O(l^{-\mu}),\quad 2l\leq t\leq 4l.
$$ 
This yields the desired estimate \eqref{Eq: decay qs metric}.

Denote $(\tilde M_{glue},\tilde g_{glue})$ to be the Riemannian manifold from the gluing of $(\Omega,g)$ and $(\bar\Omega_{ext},\bar g_{qs})$ along their boundaries. Next we would like to use a deformation argument to verify the Ricci flatness of $(\tilde M_{glue},\tilde g_{glue})$ in its smooth part as well as the coincidence of second fundamental forms of $\partial\Omega$ in $(\tilde M_{glue},\tilde g_{glue})$ from both sides with respect to the outer unit normal pointing to $(\bar \Omega_{ext},\bar g_{qs})$. Let $U$ be a region in $\tilde M_{glue}$ containing $\Omega$ with compact closure and $h$ be a smooth $(0,2)$-tensor in $U$ with compact support such that $h$ has the form of $h_{\alpha\beta}\mathrm dx^\alpha\otimes \mathrm dx^\beta$ in some Fermi coordinate $(s,x^\alpha)$ around $\partial\Omega$. In the following, we
consider the metrics $\tilde g_{\tau}=\tilde g_{glue}+\tau h$ for $\tau$ small. From a similar discussion as in \cite[P. 423-426]{Kato1995}, the quadratic form
$$
Q_\tau(u)=\int_U|\nabla_{\tilde g_\tau}\phi|^2+cR_{\tilde g_\tau}\phi^2\,\mathrm d\mu_{\tilde g_\tau}, \quad \phi\in C^\infty(U),\quad c=\frac{\dim\Omega-2}{4(\dim\Omega-1)},
$$
is analytic with respect to $\tau$ and so its first Neumann eigenvalue $\tilde \lambda_\tau$ and its first eigenfunction $\tilde u_\tau$ with $\|\tilde u_\tau\|_{L^2(U,\tilde g_\tau)}=\vol(U,\tilde g_{glue})$ are analytic with respect to $\tau$ as well. Since the scalar curvature $R(\tilde g_{glue})$ vanishes everywhere, we see $\tilde\lambda_0=0$ and $\tilde u_0\equiv 1$. From the choise of $h$, we also know that the mean curvatures of $\partial\Omega$ in $(\tilde M,\tilde g_\tau)$ are the same on two sides with respect to the outer unit normal, denoted by $\tilde H_\tau$. From a similar computation as in \cite[Lemma 2.1]{BC2019} (see also \cite[Lemma 2.3]{GZ2021}), we have
\[
\begin{split}
&\vol(U,\tilde g_{glue})\left.\frac{\mathrm d}{\mathrm d\tau}\right|_{\tau =0}\tilde \lambda_\tau\\
=&c\int_{U}\left.\frac{\partial}{\partial \tau}\right|_{\tau=0}R(\tilde g_\tau)\,\mathrm d\mu_{\tilde g_{glue}}\\
=&c\int_{\Omega}\left.\frac{\partial}{\partial \tau}\right|_{\tau=0}R(\tilde g_\tau)\,\mathrm d\mu_{\tilde g_{glue}}+2c\int_{\partial\Omega}\left.\frac{\partial}{\partial \tau}\right|_{\tau=0}\tilde H_\tau\,\mathrm d\sigma_{\tilde g_{glue}}\\
&+c\int_{U-\Omega}\left.\frac{\partial}{\partial \tau}\right|_{\tau=0}R(\tilde g_\tau)\,\mathrm d\mu_{\tilde g_{glue}}+2c\int_{\partial\Omega}\left.\frac{\partial}{\partial \tau}\right|_{\tau=0}(-\tilde H_\tau)\,\mathrm d\sigma_{\tilde g_{glue}}\\
=&2c\int_{U}\langle h,\Ric(\tilde g_{glue})\rangle_{\tilde g_{glue}}\,\mathrm d\mu_{\tilde g_{glue}}+2c\int_{\partial\Omega}\langle A_--A_+,h\rangle_{\tilde g_{glue}}\,\mathrm d\sigma_{\tilde g_{glue}},
\end{split}
\]
where $A_-$ and $A_+$ are second fundamental forms of $\partial\Omega$ in $\Omega$ and $\tilde M-\Omega$ with respect to the outer unit normal respectively. If the smooth part of $\tilde g_{glue}$ is not Ricci-flat or the second fundamental forms $A_-$ and $A_+$ are different, we can pick up a suitable $(0,2)$-tensor $h$ such that the first Neumann eigenvalue $\tilde \lambda_\tau$ becomes positive for some $\tau$ and then a contradiction can be derived from Lemma \ref{Lem: conformal factor} as well as the conformal deformation and compactification arguments in the proof of Theorem \ref{Thm: main 5}.

Now we are ready to investigate the Ricci-DeTurk flow from $(M,\tilde g_{glue})$. The argument here is very similar to that in \cite{MS2012}. Fix a background metric $h$ on $M$ such that $h=g_{euc}\oplus g_{flat}$ around infinity and
$$
(1+\epsilon)^{-1}h\leq \tilde g_{glue}\leq (1+\epsilon)h
$$
for sufficiently small $\epsilon$. M. Simon \cite{Simon2002} proved that there is a Ricci-DeTurk $h$-flow $\{\tilde g(s)\}_{0\leq s\leq T}$ from the metric $\tilde g_{glue}$ satisfying
$$
(1+2\epsilon)^{-1}h\leq \tilde g(s)\leq (1+2\epsilon)h,\quad \forall\,s\in[0,T].
$$
The above $h$-flow $\tilde g_h(s)$ is constructed by first taking a sequence of smoothings $\tilde g^\delta$ converging to $\tilde g_{glue}$ as $\delta\to 0$, then running $h$-flows $\{\tilde g_h^\delta(s)\}_{0\leq s\leq T}$ from $\tilde g^\delta$ on a uniform interval $[0,T]$ and finally taking the limit of $h$-flows $\tilde g_h^\delta(s)$ as $\delta\to 0$. We will have a careful analysis on this procedure. Since the metric $\tilde g_{glue}$ is now $C^{1,1}$, as in \cite{Miao2002} we can construct a family of smooth metrics $\tilde g^\delta$ on $M$ converging to $\tilde g_{glue}$ as $\delta\to 0$ such that
\begin{itemize}
\item $\tilde g^\delta$ differs from $\tilde g_{glue}$ in a fixed compact subset of $M$ containing $\partial\Omega$;
\item $\tilde g^\delta$ satisfies the decay estimate
\begin{equation}\label{Eq: C2 epsilon decay}
|\tilde g^\delta-h|_h+r|\nabla_h(\tilde g^\delta-h)|_h+r^2|\nabla_h^2(\tilde g^\delta-h)|\leq Cr^{-\mu}
\end{equation}
for a universal constant $C$ independent of $\delta$;
\item there are universal positive constants $C_1$, $C_2$ and $C_3$ independent of $\delta$ such that
$
R(\tilde g^\delta)\geq -C_1
$
and
$$
\int_M|R(\tilde g^\delta)|\,\mathrm d\mu_{\tilde g^\delta}\leq C_2,\quad \int_{M}R(\tilde g^\delta)_-\,\mathrm d\mu_{\tilde g^\delta}\leq C_3\delta,
$$
where $R(\tilde g^\delta)_-$ is denoted to be the negative part of $R(\tilde g^\delta)$.
\end{itemize}
In the following, a metric is said to be asymptotically flat in $C^2_\mu$ if it satisfies \eqref{Eq: C2 epsilon decay}. From the work of Shi in \cite{Shi1989}, we can run Ricci flows $\{\tilde g^\delta(s)\}_{0\leq s\leq T}$ from $\tilde g^\delta$ on a unifrom interval $[0,T]$. Based on weighted spaces or the maximum principle, it is not difficult to show that $\tilde g^\delta(s)$ is still asymptotically flat in $C^2_\mu$. Now arguing as in \cite[Section 3]{MS2012} we can show that 
\begin{equation}\label{Eq: almost nonnegative scalar curvature}
\int_{M}R(\tilde g^\delta(s))_-\,\mathrm d\mu_{\tilde g^\delta(s)}\leq C_3'\delta,
\end{equation}
where $C_3'$ is a universal constant depending on $s$ but independent of $\delta$. Now let us turn to the $h$-flow $\{\tilde g^\delta_h(s)\}_{0\leq s\leq T}$ from the smoothing metric $\tilde g^\delta$. The $h$-flow $\{\tilde g^\delta_h(s)\}_{0\leq s\leq T}$ differs from the Ricci flow $\{\tilde g^\delta(s)\}_{0\leq s\leq T}$ by a family of diffeomorphisms, but it has the advantage that we can take the limit $h$-flow $\tilde g_h(s)$ by letting $\delta\to 0$. From the maximum principle we argue as in \cite[Appendix]{MS2012} to conclude that $\tilde g^\delta_h(s)$ is asymptotically flat in $C^2_\mu$ with a uniform constant in the decay estimate \ref{Eq: C2 epsilon decay} (Notice that our starting metric is in $C^2_\mu$ which is better than the original situation in \cite[Appendix]{MS2012}). Notice that the estimate \eqref{Eq: almost nonnegative scalar curvature} also holds for $\tilde g^\delta_h$ and we conclude that the limit $h$-flow $\tilde g_h(s)$ has nonnegative scalar curvature. It follows from Theorem \ref{Thm: main 5} that the metric $\tilde g_h(s)$ has to be flat. Away from the corner $\partial\Omega$, the limit $h$-flow $\tilde g_h(s)$ converges smoothly to $\tilde g_{glue}$ and so $(\Omega,g)$ is flat.
\end{proof}

\begin{proof}[Proof for Proposition \ref{Prop: main 7}]
The proof will be divided into two cases:

{\it Case 1. $(T^2,g_{flat})$ is isometric to $\mathbf S^1(a)\times \mathbf S^1(b)$.} From Proposition \ref{Prop: main 6}, all we need to show is that if $\Omega$ is a compact $3$-manifold with boundary diffeomorphic to $\mathbf S^1(a)\times \mathbf S^1(b)$, then one of $\mathbf S^1(a)$ and $\mathbf S^1(b)$ is homotopically non-trivial in $\Omega$. This is actually a simple fact in topology and the reasoning is as follows. If $\mathbf S^1(a)$ is homotopically non-trivial, then we are done. Now let us assume that $\mathbf S^1(a)$ is homotopic to a point in $\Omega$. From Dehn's lemma we can find an embedded disk $D$ in $\Omega$ with boundary to be $\mathbf S^1(a)$. Notice that $\partial\Omega\cup D$ is homeomorphic to a solid torus minus a ball. So $\Omega$ is homeomorphic to a compact $3$-manifold in the form of $B^2\times \mathbf S^1(b)\#N$, where $N$ is a closed $3$-manifold. Now the circle $\mathbf S^1(b)$ must be homotopically non-trivial in $\Omega$. 
From above discussion we can apply Proposition \ref{Prop: main 6} to conclude
\[
\begin{split}
\Lambda(T^2,g_{flat})&\leq \max\left\{2\pi b\cdot\Lambda(\mathbf S^1(a)),2\pi a\cdot\Lambda(\mathbf S^1(b))\right\}\\
&=4\pi^2\max\{a,b\}.
\end{split}
\]

{\it General case.} The proof is the same as in \cite{SWWZ2021} which is based on the quasi-spherical metric and a gluing procedure. First observe that all flat metrics on $T^2$ actually form a connected space. To see this, we start with an arbitrary flat metric $g_{flat}$ on $T^2$. It is well-known that $( T^2,g_{flat})$ can be viewed as a quotient space $\mathbf R^2/\Gamma$ for some lattice  $\Gamma$ in $\mathbf R^2$. We can find an orientation-preserving affine transformation $\Phi:\mathbf R^2\to \mathbf R^2$ such that $\Phi(\mathbf Z^2)=\Gamma$. From \cite[Corollary 3.6]{GM2018} we can pick up a smooth family of orientation-preserving affine transformations
$\{\Phi_t\}_{0\leq t\leq 1}$ with $\Phi_0=\Phi$ and $\Phi_1=\id$. Then the pullback metric $\Phi_t^*(g_{euc})$ induces a smooth family of flat metrics $\{g_t\}_{0\leq t\leq 1}$ on $T^2$ with $g_0=g_{flat}$ and $g_1$ to be the product metric $g_{prod}=\mathrm d\theta_1^2+\mathrm d\theta_2^2$. 

Now the basic idea is to extend any admissible fill-in of $(T^2,g_{flat})$ to an admissible fill-in of $(T^2,\lambda^2 g_{prod})$ for some $\lambda>0$ and track the change of mean curvature. The key turns out to be the construction of the neck region illustrated in Figure \ref{Fig: 8}. 
\begin{figure}[htbp]
\centering
\includegraphics[width=7cm]{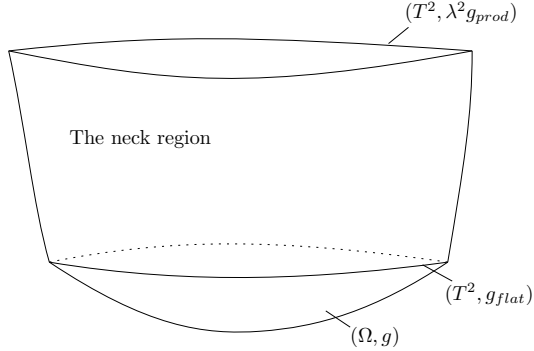}
\caption{The extension of admissible fill-in}
\label{Fig: 8}
\end{figure}

Notice that there is a positive constant $k$ depending only on $\{\Phi_t\}_{0\leq t\leq 1}$ such that the metric $\bar g=\mathrm ds^2+s^2g_{\frac{s-1}{k}}$ on $[1,k+1]\times T^2$ satisfying the estimate
\begin{equation}\label{Eq: second fundamental form}
\left|\bar A_s-\frac{1}{s}\bar g_{\frac{s-1}{k}}\right|_{g_{\frac{s-1}{k}}}\leq \frac{1}{2s},\quad \text{where}\quad \bar g_{\frac{s-1}{k}}=s^2g_{\frac{s-1}{k}},
\end{equation}
where $\bar A_s$ is the second fundamental form of $\{s\}\times T^2$ with respect to $\partial_s$. We now investigate the following  quasi-spherical equation
$$
\bar H_s\frac{\partial u}{\partial s}=u^2\Delta_{\bar g_{\frac{s-1}{k}}}u-\frac{1}{2}R(\bar g)u,\quad u(0,\cdot)=u_0>0.
$$
The solution $u$ exists on $[1,k+1]\times T^2$ for any positive initial value $u_0$ due to the same argument as in the proof of Proposition \ref{Prop: main 6}.
A direct computation combined with \eqref{Eq: second fundamental form} yields
\begin{equation}\label{Eq: change of mean curvature}
\begin{split}
\frac{\mathrm d}{\mathrm ds}\int_{\{s\}\times  T^2}\bar H_s u^{-1}\,\mathrm d\sigma_s &=\frac{1}{2}\int_{\{s\}\times  T^2}\left(\bar H_s^2-|\bar A_s|^2\right)u^{-1}\,\mathrm d\sigma_s\\
&\geq \frac{1}{6s}\int_{\{s\}\times  T^2}\bar H_s u^{-1}\,\mathrm d\sigma_s.
\end{split}
\end{equation}

Let us start with an admissible fill-in $(\Omega,g)$ of $(T^2,g_{flat})$, whose boundary has total mean curvature $T_0$. After setting suitable initial value $u_0$, we can construct a quasi-spherical metric $\tilde g=u^2\mathrm ds^2+s^2\gamma_{\frac{s-1}{k}}$ on the neck region such that the mean curvature on two sides of the corner coincide. With a handle of corners as in \cite{Miao2002}, we can construct an admissible fill-in $(\Omega',g')$ of $\left(T^2,(k+1)^2g_{prod}\right)$.  From \eqref{Eq: change of mean curvature} it follows that the total mean curvature of $\partial\Omega'$ is no less than $C(k)T_0$.  From the discussion of Case 1, we conclude
$$\Lambda(T^2,g_{flat})\leq 4\pi^2(k+1)C(k)^{-1}<+\infty.$$
This completes the proof.
\end{proof}

\newpage
\appendix
\section{Topological preparation}
In this appendix, we collect some basic topological results used in previous sections. First let us start with the following
\begin{lemma}\label{Lem: realize}
Let $M^n$ be an orientable differentable $n$-manifold (compact or non-compact) without boundary. Then any homology class in $H_{n-1}(M,\mathbf Z)$ can be represented by an embedded hypersurface.
\end{lemma}
\begin{proof}
We divide the discussion into two cases.

{\it Case 1. $M$ is compact.} Let $\beta$ be a homology class in $H_{n-1}(M,\mathbf Z)$. If $\beta$ is zero, then we take a small geodesic sphere around a point. Otherwise, its Poincar\'e dual $\beta^*$ is a non-trivial element in $H^1(M,\mathbf Z)$. From algebraic topology (see \cite[Theorem 4.57]{Hatcher2002}), the cohomology group $H^1(M,\mathbf Z)$ is isomorphic to $[M,\mathbf S^1]$, where $[M,\mathbf S^1]$ is the set of all homotop classes of continuous maps from $M$ to $\mathbf S^1$. Take $f$ to be a smooth map whose homotopy class corresponds to $\beta^*$. Then $f$ is surjective since $\beta^*$ is non-trivial. Pick up a regular value $q$ of $f$. The homology class $\beta$ is now represented by the embedded hypersurface $f^{-1}(q)$.

{\it Case 2. $M$ is non-compact.} Let $\beta$ be a homology class in $H_{n-1}(M,\mathbf Z)$. As before, we just need to deal with the case when $\beta$ is non-trivial. Take a cycle $C$ representing the homology class $\beta$ and fix a smooth compact subset $K$ of $M$ containing the support of $C$. If $C$ represents a trivial homology class in $H_{n-1}(K,\partial K,\mathbf Z)$, then $C$ is homologous to a linear combination of boundary components of $\partial K$. It is clear that the latter provides an embedded hypersurface representing the class $\beta$ since each component with multiplicity can be perturbed to disjoint copies nearby. Otherwise, the Poincar\'e dual of $C$ is a non-trivial cohomology class in $H^1(K,\mathbf Z)$. The same argument as in Case 1 yields the existence of an embedded hypersurface $\Sigma$ homologous to $C$ in $H_{n-1}(K,\partial K,\mathbf Z)$. Without loss of generality, we can assume that $\Sigma$ intersects with $\partial K$ transversely. By definition, we can find $n$-chain $\Omega$ in $K$ and $(n-1)$-chain $C'$ in $\partial K$ such that $C=\partial\Omega+\Sigma+C'$. Since $C$ is a cycle, we see $\partial C'=\partial\Sigma$ and its support is $\Sigma\cap\partial K$, which consists of disjoint hypersurfacs in $\partial K$. From the decomposition theorem from geometric measure theory (see \cite[Theorem 27.6]{Leon1983}), we can find a sequence of decreasing open subsets $\{U_i\}_{i\in \mathbf Z}$ in $\partial K$ such that
$$
C'=\sum_{i>0}U_i-\sum_{i\leq 0}(\partial K-U_i)
$$
and
$$
\partial C'=\sum_{i\in \mathbf Z}\partial U_i.
$$
Since the intersection $\Sigma\cap \partial K$ has multiplicity one, we see that the boundaries $\partial U_i$ are pairwisely disjoint and so there are only finitely many open sets $U_i$, where $i=-p,1-p,\ldots,0,\ldots, q$. As shown in the Figure \ref{Fig: 2}, we can perturbe $U_i$ along the normal direction a little bit such that the union
$$
\Sigma\cup\left(\bigcup_{i=1}^q U_i\right)\cup\left(\bigcup_{i=-p}^0(\partial K-U_i)\right)
$$
is perturbed to be an embedded hypersurface homologous to $C$. This just completes the proof when $M$ is non-compact.
\end{proof}
\begin{figure}[htbp]
\centering
\includegraphics[width=5cm]{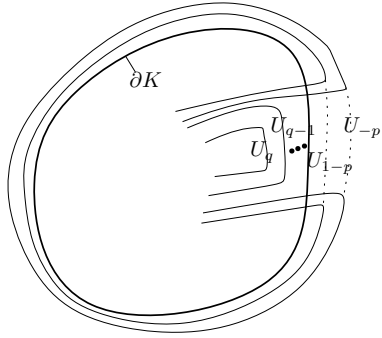}
\caption{Perturbation of open subsets $U_i$}
\label{Fig: 2}
\end{figure}

\begin{definition}
Let $M$ be a differential manifold with non-empty closed connected boundary $\partial M$ and a marked point $p$ in $\partial M$. We say that $M$ has the lifting property if there is a positive integer $k$ such that for any subgroup $H$ of $\pi_1(\partial M,p)$ we can find a covering $\tilde p:\tilde M\to M$ satisfying
\begin{itemize}
\item $\partial\tilde M$ has $k$ components;
\item
 each component $\tilde C$ of $\partial\tilde M$ satisfies
$
\tilde p_*(\pi_1(\tilde C,\tilde p))=H
$
for some point $\tilde p$ in $\tilde C$.
\end{itemize}
\end{definition}

\begin{lemma}\label{Lem: lifting property example}
The manifold $T^k\times (T^{n-k}-B)$ with any marked point $p$ satisfies the lifting property.
\end{lemma}
\begin{proof}
Since $\pi_1(T^k\times N)=\pi_1(T^k)\oplus \pi_1(N)$, all we need to show is that $T^n-B$ satisfies the lifting property. For $n\geq 3$, since the fundamental group $\pi_1(\partial(T^n-B))$ is trivial, the lifting property holds for $T^n-B$ without lifting.

Next we focus on the case when $n=2$ and the proof is a direct construction. First notice that $T^2-B$ can be viewed as the quotient space of a `+' shaped stripe by identifying the opposite edges (see Figure \ref{Fig: 4}). Fix a point $x_0$ in $\partial(T^2-B)$, then the fundamental group $\pi_1(\partial(T^2-B),x_0)$ is generated by the boundary curve $\gamma$.
\begin{figure}[htbp]
\centering
\includegraphics[width=5cm]{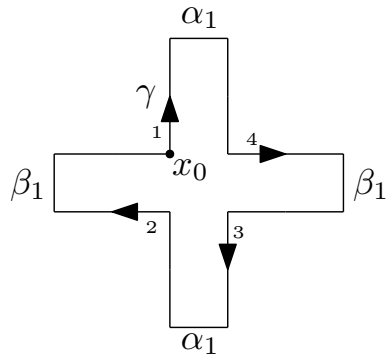}
\caption{$T^2-B$ comes from a `+' shaped stipe by identifying opposite edges $\alpha_1$ and $\beta_1$. The fundamental group $\pi_1(\partial(T^2-B),x_0)$ is generated by above closed curve $\gamma$.}
\label{Fig: 4}
\end{figure}

With the help of this polygon representation, given any subgroup $H$ of $\pi_1(\partial(T^2-B),x_0)$ it is not difficult to construct the desired covering $\tilde M$ of $T^2-B$ with two boundary components. The construction is divided into the following two cases:

{\it Case 1. The subgroup $H$ is generalized by some $\gamma^k$.} The desired covering is given by the polygon representation in Figure \ref{Fig: 5}.
\begin{figure}[htbp]
\centering
\includegraphics[width=8cm]{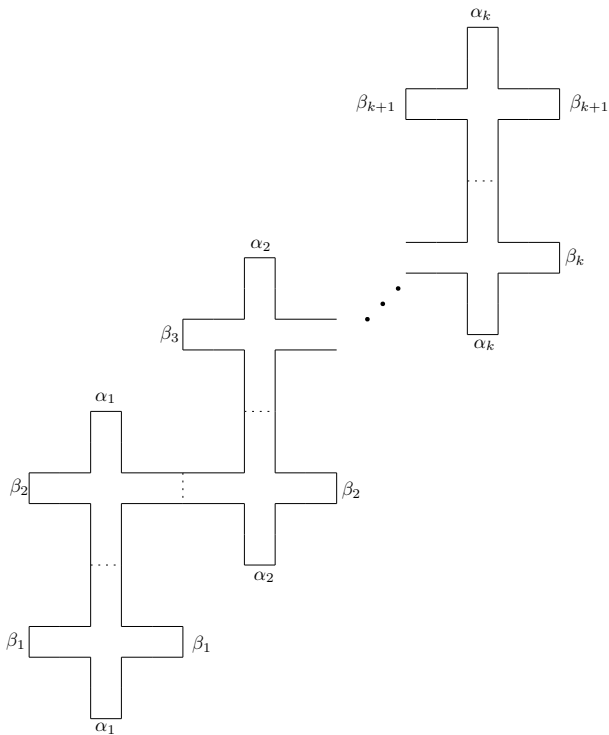}
\caption{The covering of $T^2-B$ with two boundary components corresponding to the subgroup $H=\langle \gamma^k\rangle$.}
\label{Fig: 5}
\end{figure}

{\it Case 2. The subgroup $H$ is the trivial group.} Now it is clear that the desired covering must be non-compact and so we need to find the desired covering with an infinite polygon representation. The construction is illustrated in the following Figure \ref{Fig: 6}.
\begin{figure}[htbp]
\centering
\includegraphics[width=8cm]{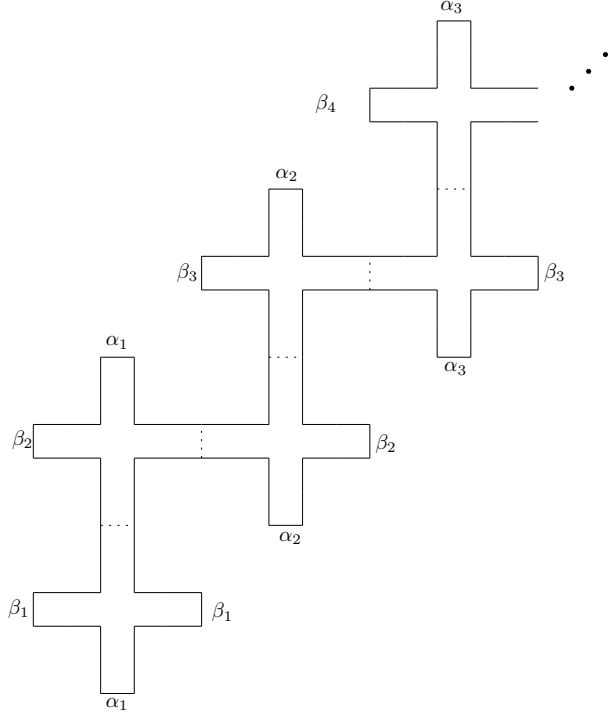}
\caption{The covering of $T^2-B$ with two boundary components corresponding to the trivial subgroup.}
\label{Fig: 6}
\end{figure}
\end{proof}

\begin{lemma}\label{Lem: no contraction extension}
Let $M$ be a manifold with non-empty closed connected boundary $\partial M$, which satisfies the lifting property. Then for any gluing space $M\sqcup_\phi N$ with $\phi:\partial M\to \partial N$ a diffeomorphism, a closed curve $\gamma:\mathbf S^1\to N$ is homotopically non-trivial in $M\sqcup_\phi N$ if and only if it is homotopically non-trivial in $N$.
\end{lemma}
\begin{proof}
One direction is clear. Namely, if we know $\gamma$ is already homotopically non-trivial in $M\sqcup_\phi N$, then it is also homotopically non-trivial in $N$. Next we show the converse. Let us fix a marked point $y$ in $\partial N$ and take the universal covering $p:\tilde N\to N$. Since the fundamental group of a manifold is always countable, the preimage $p^{-1}(y)$ contains countably many points. As a result, we know that $\partial\tilde N$ has countably many components. Let us label them as $\tilde\partial_i$ with $i\in\mathbf N_+$. Clearly the boundary $\partial\tilde N$ is just the disjoint union of all $\tilde \partial_i$. Notice also that the map $p$ restricted to each $\tilde\partial_i$ gives a covering map from $\tilde\partial _i$ to $\partial N$. We will fix a marked point $\tilde y_i$ in $p^{-1}(y)\cap \tilde \partial_i$ for each $\tilde \partial_i$ in the following discussion.

Next we will extend $\tilde N$ to be a covering of $M\sqcup_\phi N$ based on the lifting property of $M$. Denote $x=\phi^{-1}(y)$ and $H_i$ to be the subgroup of $\pi_1(\partial M,x)$ given by $(\phi^{-1}\circ p)_*(\pi_1(\tilde\partial_i))$. From the lift property of $M$ there is a positive integer $k$ independent of $i$ such that we can find a covering $p_i:\tilde M_i\to M$ such that
\begin{itemize}
\item $\partial \tilde M_i$ has $k$ components,
\item each component $\tilde C_{ij}$ has a marked point $\tilde x_{ij}$ such that
$$
(p_i)_*(\pi_1(\tilde C_{ij},\tilde x_{ij}))=H_i.
$$
\end{itemize}
From this we can lift $\phi$ to be a diffeomorphism $\tilde \phi_{ij}:(\tilde C_{ij},\tilde x_{ij})\to (\tilde \partial_i,\tilde y_i)$ such that the following diagram
\begin{equation*}
\xymatrix{(\tilde C_{ij},\tilde x_{ij})\ar[r]^{\tilde\phi_{ij}}\ar[d]^{p_{ij}}&(\tilde \partial_i,\tilde y_i)\ar[d]^{p}\\
(\partial M,x)\ar[r]^{\phi}&(\partial N,y)}
\end{equation*}
is commutative. Take the disjoint union
$
\tilde M=\sqcup_i \tilde M_i
$. Then we can glue $\tilde M$ with $k$ copies of $\tilde N$ to obtain a covering $\tilde p:\tilde M\sqcup k\tilde N \to M\sqcup_{\phi} N$. The gluing is illustrated in Figure \ref{Fig: 3}.
\begin{figure}[htbp]
\centering
\includegraphics[width=5cm]{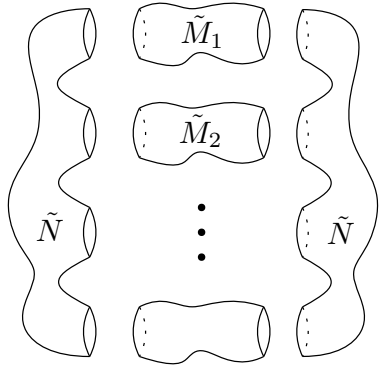}
\caption{$\tilde M$ glued with $2$ copies of $\tilde N$ when $M$ satisfies the lifting property with $k=2$.}
\label{Fig: 3}
\end{figure}
\\
It is clear we have the following commutative diagram
\begin{equation*}
\xymatrix{\tilde N\ar[r]^{i\quad}\ar[d]^{\tilde p}&\tilde M\sqcup k\tilde N\ar[d]^{\tilde p}\\
N\ar[r]^{i\quad}&M\sqcup_\phi N,}
\end{equation*}
where $i$ represents the canonical inclusion map.

Now let us assume that $\gamma:\mathbf S^1\to N$ is homotopically non-trivial in $N$ and explain the reason why $i\circ \gamma$ is homotopically non-trivial in $M\sqcup_\phi N$. Since $\tilde p:\tilde N\to N$ is the universal covering, $\gamma$ can be lift to a path $\tilde \gamma:[0,1]\to \tilde N$ with different end points. Notice that $i\circ \tilde\gamma$ is exactly the lift of $i\circ \gamma$ under the covering $\tilde p:\tilde M\sqcup k\tilde N\to M\sqcup_\phi N$ and it still has different end points. This yields that $i\circ \gamma$ cannot be homotopic to a point in $M\sqcup_\phi N$. This completes the proof.
\end{proof}

\bibliography{bib}
\bibliographystyle{amsplain}
\end{document}